\documentclass[11pt]{article}

\usepackage[margin=1cm]{geometry}
\usepackage{fullpage}
\usepackage{amsmath}
\usepackage{amsthm}
\usepackage{amssymb}
\usepackage{enumerate}
\usepackage[hidelinks]{hyperref}
\usepackage{cleveref}
\usepackage{textcomp}

\newcommand{\mbb}[1]{\mathbb{#1}}
\newcommand{\mbf}[1]{\mathbf{#1}}
\newcommand{\mc}[1]{\mathcal{#1}}
\newcommand{\bs}{\boldsymbol}
\newcommand{\tr}{\textup{tr}\,}
\newcommand{\wt}[1]{\widetilde{#1}}

\newcommand{\diag}{\textup{diag}\,}
\newcommand{\pf}{\textup{pf}\,}
\newcommand{\Tr}[1]{\left\langle{#1}\right\rangle}
\newcommand{\diff}{\,\mathrm{d}}
\renewcommand{\det}{\textup{det}}
\renewcommand{\Re}{\textup{Re}}
\renewcommand{\Im}{\textup{Im}}
\numberwithin{equation}{section}

\newtheorem{theorem}{Theorem}[section]

\newtheorem{lemma}{Lemma}[section]

\newtheorem{proposition}{Proposition}[section]

\theoremstyle{remark}

\theoremstyle{definition}
\newtheorem{definition}{Definition}[section]

\title{Bulk Universality for Real Matrices with Independent and Identically Distributed Entries}
\author{Mohammed Osman\footnote{mohammed.osman@qmul.ac.uk}}
\date{\small Queen Mary, University of London}

\begin{document}

\maketitle

\abstract{We consider real, Gauss-divisible matrices $A_{t}=A+\sqrt{t}B$, where $B$ is from the real Ginibre ensemble. We prove that the bulk correlation functions converge to a universal limit for $t=O(N^{-1/3+\epsilon})$ if $A$ satisfies certain local laws. If $A=\frac{1}{\sqrt{N}}(\xi_{jk})_{j,k=1}^{N}$ with $\xi_{jk}$ independent and identically distributed real random variables having zero mean, unit variance and finite moments, the Gaussian component can be removed using local laws proven by Bourgade--Yau--Yin, Alt--Erd\H{o}s--Kr\"{u}ger and Cipolloni--Erd\H{o}s--Schr\"{o}der and the four moment theorem of Tao--Vu.}

\section{Introduction}
A central question in random matrix theory is to determine the local statistics of the eigenvalues of $N\times N$ random matrices in the large $N$ limit. Under suitable assumptions, the local statistics are expected to fall into certain universality classes depending only on the symmetry of the matrix in question. For Hermitian random matrices, this is the content of the Wigner-Dyson-Mehta conjecture \cite{mehta_random_1990}. This conjecture has been established with increasing generality by several authors; we refer to \cite{johansson_universality_2001,soshnikov_universality_1999,erdos_bulk_2010,tao_random_2011,erdos_universality_2011,erdos_rigidity_2012,bourgade_fixed_2016,landon_fixed_2019,aggarwal_bulk_2019} among many other works. An important breakthrough was made by Erd\H{o}s, P\'{e}ch\'{e}, Ram\'{i}rez, Schlein and Yau \cite{erdos_bulk_2010}, who introduced a three-step strategy upon which most subsequent proofs of local universality were based. Roughly speaking, these steps consist of:
\begin{enumerate}
\item a deterministic approximation of the resolvent;
\item a proof of universality for matrices with a small Gaussian perturbation;
\item the removal of the Gaussian perturbation.
\end{enumerate}
A detailed account of the method can be found in the notes by Erd\H{o}s and Yau \cite{erdos_dynamical_2017}.

For non-Hermitian matrices one can formulate an analogue of the Wigner-Dyson-Mehta conjecture but the situation is less well developed. Tao and Vu \cite{tao_random_2015} established the conjecture for matrices with independent and identically distributed entries with finite moments under the assumption that the first four moments of the common distribution of the matrix entries match those of the corresponding Gaussian matrix. Cipolloni, Erd\H{o}s and Schr\"{o}der \cite{cipolloni_edge_2021} were able to remove this condition at the edge of the spectrum. In \cite{maltsev_bulk_2023}, the moment matching condition was removed in the bulk for matrices with complex entries. The goal of this paper is to do the same for real matrices. The method we use is the same as in \cite{maltsev_bulk_2023}, where the second step of the above-mentioned three-step strategy was adapted to non-Hermitian matrices using an explicit formula for the correlation functions derived by a partial Schur decomposition. In the real case, the Jacobian of the partial Schur decomposition involves a product of the absolute values of characteristic polynomials,
\begin{align*}
    \prod_{j=1}^{m}\left|\det(A-\lambda_{j})\right|.
\end{align*}
Evaluating the expectation value of such products is equivalent to evaluating ratios of the form
\begin{align*}
    \prod_{j=1}^{m}\frac{\left|\det(A-\lambda_{j}-i\delta)\right|^{2}}{\det^{1/2}\left(A-\lambda_{j}-i\delta\right)\det^{1/2}\left(A-\lambda_{j}+i\delta\right)},
\end{align*}
in the limit $\delta\downarrow0$. By the supersymmetry method, one can derive a representation for the Gaussian expectation value of such a ratio in terms of an integral over supermatrices of fixed dimension, i.e. matrices containing both commuting and anti-commuting elements. The task is then to rigourously justify the saddle-point approximation which leads to the so-called ``non-linear $\sigma$ model"; see chapter 4 of Nock's thesis \cite{nock_characteristic_2016} for some heuristic calculations and the work of Shcherbina and Shcherbina \cite{shcherbina_transfer_2016,shcherbina_universality_2018} for an example of such a justification in the context of random band matrices.

Instead of tackling this problem we use an idea of Tribe and Zaboronski \cite{tribe_ginibre_2014-1,tribe_averages_2023} which involves the introduction of ``spin" variables
\begin{align}
    s_{x}(A)&=\frac{\det(A-x)}{|\det(A-x)|},\quad x\in\mbb{R}.
\end{align}
If $x$ is an eigenvalue of $A$ then we define $s_{x}(A)$ by taking the limit from {the left}, so that $s_{x}$ is left-continuous in $x$. The reason for introducing these variables is that the absolute value in the denominator will cancel the absolute value coming from the Jacobian of the partial Schur decomposition. Tribe and Zaboronski give a new derivation of the correlation functions of real eigenvalues of Gaussian matrices, starting from an identity relating the empirical measure of real eigenvalues and the distributional derivative of a product of two spin variables. We use a different approach based on an identity relating the spin variable at a point $x$ to a sum of spin variables at real eigenvalues of $A$. We can freely insert factors of $s^{2}_{x}(A)=1$ inside expectation values and use the above-mentioned identity to obtain an expression that is linear in each factor of $s_{x}$. Evaluating the new expectation value by the partial Schur decomposition we find that we need only consider products of characteristic polynomials without absolute values, which we can treat by the supersymmetry method.

Due to this extra step and the fact that the complex eigenvalues occur in conjugate pairs, we have to place extra conditions on $A$ in order to prove universality of $A+\sqrt{t}B$, as compared to the corresponding conditions for complex matrices in \cite{maltsev_bulk_2023}. Nevertheless, using local laws proven by Alt--Erd\H{o}s--Kr\"{u}ger \cite{alt_local_2018} and Cipolloni--Erd\H{o}s--Schr\"{o}der \cite{cipolloni_central_2023}, one can show that matrices with i.i.d. entries with finite moments belong to this set with high probability. Hence we are still able to perform the third step and conclude the proof of universality.

\section{Main Results}
\paragraph{Notation} $\mbb{M}_{n}(\mbb{F}),\,\mbb{M}^{h}_{n}(\mbb{F}),\,\mbb{M}^{sym}_{n}(\mbb{F}),\,\mbb{M}^{skew}_{n}(\mbb{F})$ denote respectively the spaces of general, Hermitian, symmetric and skew-symmetric matrices with entries in $\mbb{F}$. For $M\in\mbb{M}_{n}(\mbb{F})$, $|M|=\sqrt{M^{*}M}$, $\|M\|$ denotes the operator norm and $\|M\|_{2}$ the Frobenius norm. The real and imaginary parts of $M$ are defined by $\Re M=\frac{1}{2}(M+M^{*})$ and $\Im M=\frac{1}{2i}(M-M^{*})$ respectively. $U(n),O(n)$ and $USp(n)$ denote the unitary, orthogonal and unitary symplectic groups respectively. We define the matrices
{\begin{align}
    E_{1}&=\begin{pmatrix}1_{n}&0\\0&0\end{pmatrix},\quad E_{2}=\begin{pmatrix}0&0\\0&1_{n}\end{pmatrix},\quad F=\begin{pmatrix}0&1_{n}\\0&0\end{pmatrix},\label{eq:Pauli}
\end{align}
and 
\begin{align}
    E_{\pm}&=E_{1}\pm E_{2}.\label{eq:Epm}
\end{align}
We will use the same symbols $E,F$ for matrices with different dimensions but the dimension will always be clear from the context.} We denote by $\mbb{C}_{+}$ the open upper half-plane. When $x$ belongs to a coset space of a compact Lie group (e.g. $U(n),\,O(n)/O(n-m)$), we denote by $\,\mathrm{d}_{H}x$ the Haar measure. For a probability distribution $\mc{D}$, we write $X\sim\mc{D}$ when $X$ is a random variable with distribution $\mc{D}$.

The eigenvalues of a real random matrix $X$ are either real or pairs of complex conjugates. The correlation functions are therefore indexed by two positive integers according to the number of real and complex arguments: for $f\in C_{c}(\mbb{R}^{m_{r}}\times\mbb{C}^{m_{c}}_{+})$ we define the correlation function $\rho^{(m_{r},m_{c})}_{X}$ by
\begin{align}
    \mbb{E}\left[\sum f(u_{j_{1}},...,u_{j_{m_{r}}},z_{k_{1}},...,z_{k_{m_{c}}})\right]&=\int_{\mbb{R}^{m_{r}}\times\mbb{C}^{m_{c}}_{+}}f(\mbf{u},\mbf{z})\rho^{(m_{r},m_{c})}_{X}(\mbf{u},\mbf{z})\diff\mbf{u}\diff\mbf{z},
\end{align}
where the sum is over tuples $1\leq j_{1}\neq\cdots\neq j_{m_{r}}\leq N_{\mbb{R}}$, $1\leq k_{1}\neq\cdots\neq k_{m_{c}}\leq N_{\mbb{C}}$ with distinct elements and $N_{\mbb{R}},\,N_{\mbb{C}}$ are the numbers of real and complex conjugate pairs of eigenvalues respectively. {We emphasize that $\rho^{(m_{r},m_{c})}_{X}$ depends on $N$ through $X$.}

As a result of this additional structure, there are two bulk regimes: the real bulk, in which we centre about a point $u\in\mbb{R}$, and the complex bulk, in which we centre about $z\in\mbb{C}_{+}$. Centering about $z$ in the complex bulk means we only need to consider the correlation functions with $m_{r}=0$, i.e. $\rho^{(0,m_{c})}_{X}$. In the real bulk we have to consider general values of $m_{r}$ and $m_{c}$. For scale factors $\sigma_{u}$ and $\sigma_{z}$, we define the rescaled correlation functions
\begin{align}
    \rho^{(m_{r},m_{c})}_{X,r}(\mbf{u},\mbf{z};u,\sigma_{u})&=\frac{1}{(N\sigma_{u})^{m_{r}/2+m_{c}}}\rho^{(m_{r},m_{c})}_{X}\left(u+\frac{\mbf{u}}{\sqrt{N\sigma_{u}}},u+\frac{\mbf{z}}{\sqrt{N\sigma_{u}}}\right),\\
    \rho^{(m)}_{X,c}(\mbf{z};z,\sigma_{z})&=\frac{1}{(N\sigma_{z})^{m}}\rho^{(0,m)}_{X}\left(z+\frac{\mbf{z}}{\sqrt{N\sigma_{z}}}\right).
\end{align}
This scaling assumes that the typical eigenvalue spacing is $O(N^{-1/2})$.

In the Gaussian case the limits of these rescaled correlation functions have been obtained by Borodin--Sinclair \cite{borodin_ginibre_2009}, Forrester--Nagao \cite{forrester_eigenvalue_2007} and Sommers--Wieczorek \cite{sommers_general_2008} {and will be denoted by $\rho^{(m_{r},m_{c})}_{GinOE,r}$ and $\rho^{(m)}_{GinOE,c}$}. In the real bulk the limit is a rather involved expression in terms of a Pfaffian; in the complex bulk it is the same as the complex Ginibre ensemble:
{
\begin{align*}
    \rho^{(m)}_{GinOE,c}(\mbf{z};z,1)&=\det\left[\frac{1}{\pi}e^{-\frac{1}{2}(|z_{j}|^{2}+|z_{k}|^{2})+\bar{z}_{j}z_{k}}\right]_{j,k=1}^{m}.
\end{align*}
}
In the rest of the paper, $N$ will represent a canonical matrix size that will be taken to infinity and $n$ a positive integer such that $n=O(N)$. We define the normalised trace $\Tr{X}=N^{-1}\tr X$ and stress that we always normalise by $N$ and not {by} the dimension of $X$. For $z\in\mbb{C}$ we define the shifted matrix $X_{z}=X-z$ and the Hermitisation
\begin{align}
    {W}^{X}_{z}&=\begin{pmatrix}0&X_{z}\\X^{*}_{z}&0\end{pmatrix},\label{eq:Hermitisation}
\end{align}
For $\eta\in\mbb{R}$ we define the resolvent $G^{X}_{z}(i\eta)=({W}^{X}_{z}-i\eta)^{-1}$ of the Hermitisation on the imaginary axis. We define the resolvents $H^{X}_{z}(\eta)$ and $\wt{H}^{X}_{z}(\eta)$ by the block representation of $G^{X}_{z}$:
\begin{align}
    G^{X}_{z}(i\eta)&=\begin{pmatrix}i\eta \wt{H}^{X}_{z}(\eta)&X_{z}H^{X}_{z}(\eta)\\H^{X}_{z}(\eta)X^{*}_{z}&i\eta H^{X}_{z}(\eta)\end{pmatrix}.\label{eq:blockRep}
\end{align}
Explicitly,
\begin{align}
    H^{X}_{z}(\eta)&=\left(\eta^{2}+|X_{z}|^{2}\right)^{-1},\label{eq:H}\\
    \wt{H}^{X}_{z}(\eta)&=\left(\eta^{2}+|X^{*}_{z}|^{2}\right)^{-1}.\label{eq:HTilde}
\end{align}

We now come to the class of matrices for which we will be able to prove universality after adding a small Gaussian component. For $\omega>0$, let
\begin{align}
    \mbb{D}_{\omega}&=\left\{z\in\mbb{C}:|z|<\sqrt{1-\omega}\right\}
\end{align}
be the open disk of radius $\sqrt{1-\omega}$, and $\mbb{D}_{+,\omega}=\mbb{C}_{+}\cap\mbb{D}_{\omega}$.
\begin{definition}
We say that $X$ belongs to the class $\mc{X}_{n}(\gamma,\omega)$ for $\gamma,\omega>0$ if the following conditions are met.
\paragraph{C0}
There is a constant $c_{0}$ such that
\begin{align}
    \tag{C0}
    \|X\|&\leq c_{0}.\label{eq:C0}
\end{align}
\paragraph{C1}
There is a constant $c_{1}$ such that
\begin{align}
    \tag{C1.1}
    1/c_{1}\leq\eta\Tr{H_{z}(\eta)}&\leq c_{1},\label{eq:C1.1}\\
    \tag{C1.2}
    1/c_{1}\leq\eta^{3}\Tr{(H_{z}(\eta))^{2}}&\leq c_{1},\label{eq:C1.2}
\end{align}
for $z\in\mbb{D}_{\omega}$ and $N^{-\gamma}\leq\eta\leq10$.
\paragraph{C2}
There is a constant $c_{2}$ such that 
\begin{align}
    \tag{C2.1}
    \eta_{1}\eta_{2}\Tr{\wt{H}_{z_{1}}(\eta_{1})H_{z_{2}}(\eta_{2})}&\geq c_{2}\cdot\frac{\eta_{1}\wedge\eta_{2}}{|z-w|^{2}+\eta_{1}\vee\eta_{2}},\label{eq:C2.1}\\
    \tag{C2.2}
    \eta^{2}\Tr{H_{z}(\eta)H_{\bar{z}}(\eta)}&\geq \frac{c_{2}}{(\Im z)^{2}+\eta},\label{eq:C2.2}
\end{align}
for $z_{1}\in\mbb{D}_{\omega}$, $z_{2}\in\mbb{C}$, $N^{-\gamma}\leq\eta_{1}\leq10$ and $N^{-\gamma}\leq\eta_{2}\leq10\|X\|$.
\paragraph{C3}
There is a constant $c_{3}$ such that
\begin{align}
    \tag{C3.1}
    \left|\Tr{G_{z_{1}}(i\eta_{1})B_{1}G_{z_{2}}(i\eta_{2})B_{2}}\right|&\leq c_{3},\label{eq:C3.1}\\
    \tag{C3.2}
    \left|\Tr{G_{z}(i\eta){E_{-}}G_{\bar{z}}(i\eta){E_{-}}}\right|&\leq \frac{c_{3}}{(\Im z)^{2}+\eta},\label{eq:C3.2}
\end{align}
for $z_{j}\in\mbb{D}_{\omega}$, $N^{-\gamma}\leq\eta_{j}\leq10$ and {$B_{1},B_{2}\in\{E_{+},F,F^{*}\}$.}
\end{definition}
We remark that for small $\eta<N^{-\epsilon}$, \eqref{eq:C1.2} is implied by \eqref{eq:C3.1}. As can be seen from \eqref{eq:C2.1}, \eqref{eq:C2.2} and \eqref{eq:C3.2}, we need to control traces involving two resolvents $G_{z}$ and $G_{w}$ for $z,w\in\mbb{C}$ separated by a macroscopic distance. This is in contrast to the corresponding conditions for complex matrices in \cite{maltsev_bulk_2023}, where it was enough to consider all resolvents at a single point $z$. One reason for this difference that is intrinsic to real matrices is the fact that in the complex bulk, complex conjugate pairs of eigenvalues are separated by a fixed distance. We will see later that the need for \eqref{eq:C2.1} specifically arises as a by-product of the method of spin variables and as such is not intrinsic to real matrices.

{The norm condition in \eqref{eq:C0} can be relaxed to $\|X\|<o(e^{\log^{2}N})$ in the real bulk without changing the proof. In the complex bulk we make use of this condition in Lemma \ref{lem:Lfar}, specifically \eqref{eq:normCondition}. Shortly after posting to the arXiv, the work of Dubova and Yang \cite{dubova_bulk_2024} dealing with the complex bulk appeared. They prove a determinantal identity which allows one to relax the norm condition in the complex bulk as well; we sketch this in the appendix.}

Our main result is the pointwise convergence to a universal limit of the bulk correlation functions of $A+\sqrt{t}B$ when $B\sim GinOE(n)$ for any $A\in\mc{X}_{n}(\gamma,\omega)$.
\begin{theorem}\label{thm1}
Let $\gamma\in(0,1/3),\,\omega\in(0,1)$, $A\in\mc{X}_{N}(\gamma,\omega)$ and $B\sim GinOE(N)$. Let $\rho^{(m_{r},m_{c})}_{A_{t}}$ denote the correlation function of $A_{t}=A+\sqrt{t}B$. Let $\epsilon\in(0,\gamma/2)$ and $N^{-\gamma+\epsilon}\leq t\leq N^{-\epsilon}$. For any $z\in\mbb{D}_{\omega}$, there is an $\eta_{z}$ such that $t\Tr{H^{A}_{z}(\eta_{z})}=1$. Let
\begin{align}
    \sigma^{A}_{z}&=\eta_{z}^{2}\Tr{H^{A}_{z}(\eta_{z})\wt{H}^{A}_{z}(\eta_{z})}+\frac{\left|\Tr{(H^{A}_{z}(\eta_{z}))^{2}A_{z}}\right|^{2}}{\Tr{(H^{A}_{z}(\eta_{z}))^{2}}},\label{eq:sigma}
\end{align}
{where $A_{z}:=A-z$}. Let $K_{r}\subset \mbb{R}^{m_{r}}\times\mbb{C}^{m_{c}}_{+}$ be compact; then {there is a $\delta>0$} and a null set $Q_{r}\subset\mbb{R}^{m_{r}}\times\mbb{C}^{m_{c}}_{+}$ such that, for any $z=u\in(-1+\omega,1-\omega)$,
\begin{align}
    \rho^{(m_{r},m_{c})}_{A_{t},r}(\mbf{u},\mbf{z};u,\sigma^{A}_{u})&=\rho^{(m_{r},m_{c})}_{GinOE,r}(\mbf{u},\mbf{z})+{O(N^{-\delta})},\label{eq:realBulk}
\end{align}
uniformly in $K_{r}\setminus Q_{r}$ for sufficiently large $N$. Let $K_{c}\subset\mbb{C}^{m}$ be compact; then {there is a $\delta>0$} and a null set $Q_{c}\subset\mbb{C}^{m}$ such that, for any $z\in\mbb{D}_{+,\omega}$,
\begin{align}
    \rho^{(m)}_{A_{t},c}(\mbf{z};z,\sigma^{A}_{z})&=\rho^{(m)}_{GinOE,c}(\mbf{z})+{O(N^{-\delta})},\label{eq:complexBulk}
\end{align}
uniformly on $K_{c}\setminus Q_{c}$ for sufficiently large $N$.
\end{theorem}
{Note that if $t$ is fixed, $A$ converges in $*$-moments to $a$ in a von Neumann algebra $\mc{A}$ and $c_{t}\in\mc{A}$ is a circular element of variance $t$ such that $a$ and $c_{t}$ are free, then the limit of $\sigma^{A}_{z}/\pi$ is the density (with respect to the Lebesgue measure) of the Brown measure of $a+c_{t}$ (see Zhong \cite[Theorem 4.2]{zhong_brown_2021}).}

The size of $t$ is not optimal; in the complex case, it was stated in \cite{maltsev_bulk_2023} that, assuming $A$ satisfies a local law for a product of three resolvents with deterministic matrices, one can reduce it to $t=N^{-1/2+\epsilon}$. We expect that the same is true here but a bit more work is needed to get there. In any case, to remove the Gaussian component using the four moment theorem, we only need $t=N^{-\epsilon}$ for some $\epsilon>0$. Using the single-resolvent local law of Alt--Erd\H{o}s--Kr\"{u}ger \cite{alt_local_2018} and the two-resolvent local law of Cipolloni--Erd\H{o}s--Schr\"{o}der \cite{cipolloni_central_2023}, we deduce that a matrix with i.i.d. real entries with finite moments satisfies the assumptions in \ref{thm1}, which implies bulk universality via the four moment theorem of Tao and Vu \cite{tao_random_2015}.
\begin{theorem}\label{thm2}
Let $A=\frac{1}{\sqrt{N}}(\xi_{jk})_{j,k=1}^{N}$ such that $\xi_{jk}$ are i.i.d. real random variables with zero mean, unit variance and finite moments. Then for any fixed $u\in(-1,1)$ {there is a $\delta>0$} such that
\begin{align}
    &\int_{\mbb{R}^{m_{r}}\times\mbb{C}^{m_{c}}_{+}}f(\mbf{u},\mbf{z})\rho^{(m_{r},m_{c})}_{A,r}(\mbf{u},\mbf{z};u,1)\diff\mbf{u}\diff\mbf{z}\nonumber\\
    &=\int_{\mbb{R}^{m_{r}}\times\mbb{C}^{m_{c}}_{+}}f(\mbf{u},\mbf{z})\rho^{(m_{r},m_{c})}_{GinOE,r}(\mbf{u},\mbf{z})\diff\mbf{u}\diff\mbf{z}+{O(N^{-\delta})},
\end{align}
and for any fixed $z\in\mbb{D}_{0,+}$ {there is a $\delta>0$} such that
\begin{align}
    \int_{\mbb{C}^{m}}f(\mbf{z})\rho^{(m)}_{A,c}(\mbf{z};z,1)\diff\mbf{z}&=\int_{\mbb{C}^{m}}f(\mbf{z})\rho^{(m)}_{GinOE,c}(\mbf{z})\diff\mbf{z}+{O(N^{-\delta})},
\end{align}
for sufficiently large $N$.
\end{theorem}

The rest of the paper is organised as follows. In Section \ref{sec3} we collect some basic results from linear algebra and some properties of the class $\mc{X}_{n}(\gamma,\omega)$. In Section \ref{sec4} we discuss the real partial Schur decomposition and derive an integral formula for the correlation functions. In Section \ref{sec:thm1proof} we prove Theorem \ref{thm1} based on the results in Section \ref{sec:xi}. The latter in turn depend on Section \ref{sec:expectation}, which deals with the expectation value of a product of determinants of Gauss-divisible matrices, and Section \ref{sec:gaussian}, which studies the properties of certain ``Gaussian" measures on the Stiefel manifold $O(n)/O(n-m)$ for $m=1,2$. In Section \ref{sec:thm2proof} we prove Theorem \ref{thm2}.

\section{Preliminaries}\label{sec3}
We begin with some elementary results from linear algebra. The proof of the following lemma can be found in \cite[Lemma 3.1]{maltsev_bulk_2023}.
\begin{lemma}\label{lem:minorresolvent}
Let $A\in \mbb{M}_{n}(\mbb{C})$; $U=\left(U_{k},U_{n-k}\right)\in U(n)$, where $U_{k}$ and $U_{n-k}$ are $n\times k$ and $n\times\left(n-k\right)$ respectively; and $B=U_{n-k}^{*}AU_{n-k}$. Then
\begin{align}\label{eq:minorresolvent}
    U\begin{pmatrix}0&0\\0&B^{-1}\end{pmatrix}U^{*}&=A^{-1}-A^{-1}U_{k}\left(U_{k}^{*}A^{-1}U_{k}\right)^{-1}U_{k}^{*}A^{-1},
\end{align}
and, if $\Re A>0$,
\begin{align}\label{eq:minornorm}
    \left\|A^{-1}U_{k}\left(U_{k}^{*}A^{-1}U_{k}\right)^{-1}U_{k}^{*}A^{-1}\right\|&\leq\left\|\left(\Re A\right)^{-1}\right\|.
\end{align}
If $\Im A>0$ then \eqref{eq:minornorm} holds with $\left\|\left(\Im A\right)^{-1}\right\|$ on the right hand side.
\end{lemma}

The next lemma is known as Fischer's inequality.
\begin{lemma}[Fischer's inequality \cite{fischer_uber_1908}]\label{lem:fischerineq}
Let $A$ be a positive semi-definite block matrix with diagonal blocks $A_{jj}$. Then
\begin{align}\label{eq:fischerineq}
    \det A&\leq\prod_{j}\det A_{jj}.
\end{align}
\end{lemma}

The next lemma is a lower bound for the determinant of a matrix with positive Hermitian part.
\begin{lemma}\label{lem:detLower}
Let $X,Y\in\mbb{M}^{h}_{n}(\mbb{C})$ such that $X>0$. Then
\begin{align}
    \left|\det\left(X+iY\right)\right|&\geq\det X\cdot\det^{1/2}\left(1+\frac{Y^{2}}{\|X\|^{2}}\right).
\end{align}
\end{lemma}

The next lemma is a duality formula for integrals on the real Stiefel manifold $O(n,k)=O(n)/O(n-k)=\{V\in\mbb{R}^{n\times k}:V^{T}V=1_{k}\}$. It is the real counterpart to \cite[Lemma 3.4]{maltsev_bulk_2023}.
\begin{lemma}\label{lem:sphericalint}
Let $f\in L^{1}\left(\mbb{R}^{n\times k}\right)$ be continuous on a neighbourhood of $O(n,k)$ and define $\hat{f}: \mbb{M}^{sym}_{k}(\mbb{R})\to\mbb{C}$ by
\begin{align}
    \hat{f}(P)&=\frac{1}{\pi^{k(k+1)/2}}\int_{\mbb{R}^{n\times k}}e^{-i\tr PM^{T}M}f\left(M\right)\diff M.
\end{align}
If $\hat{f}\in L^{1}$ then
\begin{align}\label{eq:sphericalint}
    \int_{O(n,k)}f\left(U\right)\,\mathrm{d}_{H}U&=\int_{ \mbb{M}^{sym}_{k}}e^{i\tr P}\hat{f}(P)\diff P.
\end{align}
\end{lemma}
{We prove Lemmas \ref{lem:detLower} and \ref{lem:sphericalint} in the appendix.}

We now study the class of matrices $\mc{X}_{n}(\gamma,\omega)$. In the remainder of this section, we fix $\gamma\in(0,1/2),\,\omega\in(0,1),\,\epsilon\in(0,\gamma/2),\,X\in\mc{X}_{n}(\gamma,\omega)$ and $N^{-\gamma+\epsilon}\leq t\leq N^{-\epsilon}$. {All resolvents are resolvents of $X$ and so for ease of notation we drop the superscript $X$.} For $z\in\mbb{C}$, we define $\eta^{X}_{z}$ by 
\begin{align}
    \eta^{X}_{z}&=\begin{cases}t\eta^{X}_{z}\Tr{H_{z}(\eta^{X}_{z})}&\quad t\Tr{H_{z}(\sqrt{t/N})}\geq1\\
    \sqrt{\frac{t}{N}}&\quad t\Tr{H_{z}(\sqrt{t/N})}<1
    \end{cases}.\label{eq:eta}
\end{align}
{Note that when $t\Tr{H_{z}(\sqrt{t/N})}\geq1$, $\eta^{X}_{z}$ is defined implicitly by \eqref{eq:eta_z} below.} In the rest of this section we write $G_{z}:=G_{z}(i\eta^{X}_{z}),\,H_{z}:=H_{z}(\eta^{X}_{z}),\,\wt{H}_{z}(\eta^{X}_{z})$, i.e. we suppress the argument of resolvents evaluated at $\eta^{X}_{z}$. Since $\Tr{H_{z}(\eta)}$ is monotonically decreasing in $\eta$, it is clear that $\eta^{X}_{z}$ is well defined and 
\begin{align}
    \eta^{X}_{z}\geq\sqrt{\frac{t}{N}},\label{eq:etaBound}
\end{align}
for all $z\in\mbb{C}$. When $z\in\mbb{D}_{\omega}$, we can say slightly more.
\begin{lemma}\label{lem:eta}
For any $z\in\mbb{D}_{\omega}$ there is a unique $\eta^{X}_{z}>0$ such that
\begin{align}
    t\Tr{H_{z}(\eta^{X}_{z})}&=1,\label{eq:eta_z}\\
    t/C\leq\eta^{X}_{z}&\leq Ct,\label{eq:eta_zB1}
\end{align}
and
\begin{align}
    \eta^{X}_{z}-\eta^{X}_{w}&=\frac{t\Im\Tr{G_{z}B_{z,w}G_{w}}}{2-t\Re\Tr{G_{z}G_{w}}}=O(t|z-w|),\label{eq:eta_zB2}
\end{align}
for any $w\in\mbb{D}_{\omega}$, where
\begin{align}
    B_{z,w}&=(z-w)F+(\bar{z}-\bar{w})F^{*}.\label{eq:Bzw}
\end{align}
\end{lemma}
\begin{proof}
Consider the fixed point equation
\begin{align*}
    \eta&=t\eta\Tr{H_{z}(\eta)}.
\end{align*}
By \eqref{eq:C1.1}, the right hand side maps $[t/C,Ct]$ to itself for sufficiently large $C$ and so there is at least one solution $\eta^{X}_{z}\in[t/C,Ct]$. Since \eqref{eq:eta_z} is monotonic in $\eta$, there is at most one solution.

The difference $\eta^{X}_{z}-\eta^{X}_{w}$ can be expressed as
\begin{align*}
    \eta^{X}_{z}-\eta^{X}_{w}&=\frac{t}{2}\Im\Tr{G_{z}-G_{w}}\\
    &=\frac{t(\eta^{X}_{z}-\eta^{X}_{w})}{2}\Re\Tr{G_{z}G_{w}}+\frac{t}{2}\Im\Tr{G_{z}B_{z,w}G_{w}}.
\end{align*}
By \eqref{eq:C3.1}, we have $\Tr{G_{z}G_{w}}<C$ and $\Tr{G_{z}FG_{w}}<C|z-w|$ and so \eqref{eq:eta_zB2} follows after taking the first term to the left hand side and dividing by the resulting coefficient of $\eta_{z}-\eta_{w}$, which is strictly positive for sufficiently small $t$.
\end{proof}

There are two quantities that are fundamental for the asymptotic analysis to follow. The first is the function $\phi_{z}:[0,\infty]\to\mbb{R}$ defined by
\begin{align}
    \phi^{X}_{z}(\eta)&=\frac{\eta^{2}}{t}-\Tr{\log\left({\eta^{2}+|X_{z}|^{2}}\right)}.\label{eq:phi}
\end{align}
{We will see in Section \ref{sec:expectation} that this function is central to the application of Laplace's method to obtain the asymptotics of the expectation value
\begin{align*}
    \mbb{E}_{Y}\prod_{j=1}^{m}\det\left(X+\sqrt{\frac{Nt}{n}}Y-\lambda_{j}\right),
\end{align*}
where $Y\sim GinOE(n)$.} We collect some basic properties of this function in the following lemma.
\begin{lemma}\label{lem:phi}
For any $z\in\mbb{C}$ we have
\begin{align}
    \phi^{X}_{z}(\eta)-\phi^{X}_{z}(\eta^{X}_{z})&\geq-\frac{1}{N},\label{eq:phiBound1}
\end{align}
for all $\eta\geq 0$. For any $z\in\mbb{D}_{\omega}$, we have
\begin{align}
    \phi^{X}_{z}(\eta)-\phi^{X}_{z}(\eta^{X}_{z})&\geq\frac{C(\eta-\eta^{X}_{z})^{2}}{t},\label{eq:phiBound2}
\end{align}
for all $\eta\geq 0$, and for $|\eta-\eta^{X}_{z}|<CN^{-1}$ we have
\begin{align}
    \left|\phi^{X}_{z}(\eta)-\phi^{X}_{z}(\eta^{X}_{z})\right|&\leq\frac{C}{N^{2}t}.\label{eq:phiBound3}
\end{align}
\end{lemma}
\begin{proof}
{For convenience we drop the superscript $X$.} Using the inequality $\log(1+x)\leq x$ {we have
\begin{align*}
    \Tr{\log(\eta^{2}+|X_{z}|^{2})}-\Tr{\log(\eta_{z}^{2}+|X_{z}|^{2})}&=\Tr{\log\left[1+(\eta^{2}-\eta_{z}^{2})H_{z}\right]}\\
    &\leq(\eta^{2}-\eta_{z}^{2})\Tr{H_{z}},
\end{align*}}
and so
\begin{align*}
    \phi_{z}(\eta)-\phi_{z}(\eta_{z})&=\frac{\eta^{2}-\eta_{z}^{2}}{t}-\Tr{\log\left[1+(\eta^{2}-\eta_{z}^{2})H_{z}\right]}\\
    &\geq\frac{\left(1-t\Tr{H_{z}}\right)(\eta^{2}-\eta_{z}^{2})}{t}.
\end{align*}
If $t\Tr{H_{z}}=1$, then we have $\phi_{z}(\eta)-\phi_{z}(\eta_{z})\geq0$. If $t\Tr{H_{z}}<1$, then by definition this means that $\eta_{z}=\sqrt{t/N}$. When $\eta>\eta_{z}$ we have $\phi_{z}(\eta)-\phi_{z}(\eta_{z})\geq0$, and when $\eta<\eta_{z}$ we have
\begin{align*}
    \phi_{z}(\eta)-\phi_{z}(\eta_{z})&\geq-\frac{\eta_{z}^{2}}{t}=-\frac{1}{N}.
\end{align*}

The first and second derivatives of $\phi_{z}$ with respect to $\eta$ are given by
\begin{align*}
    \partial_{\eta}\phi_{z}(\eta)&=\frac{2\eta}{t}\left(1-t\Tr{H_{z}(\eta)}\right),\\
    \partial^{2}_{\eta}\phi_{z}(\eta)&=\frac{2}{t}-\Re\Tr{(G_{z}(i\eta))^{2}}.
\end{align*}
When $z\in\mbb{D}_{\omega}$, we have by \eqref{eq:eta_zB1} that $\eta_{z}>t/C$. Therefore we can choose an $\eta_{0}$ such that $\eta_{z}>\eta_{0}>t/C$ for some large $C>0$. If $\eta\geq\eta_{0}$, then by \eqref{eq:C3.1} we have $|\Re\Tr{(G_{z}(i\eta))^{2}}|<C<2/t$. Therefore $1/Ct\leq\partial^{2}_{\eta}\phi_{z}(\eta)\leq C/t$ and since $\partial_{\eta}\phi_{z}(\eta_{z})=0$ we have by Taylor's theorem
\begin{align*}
    \frac{(\eta-\eta_{z})^{2}}{Ct}\leq\phi_{z}(\eta)-\phi_{z}(\eta_{z})&\leq\frac{C(\eta-\eta_{z})^{2}}{t}.
\end{align*}
If $\eta<\eta_{0}$, then $t\Tr{H_{z}(\eta)}>1$ and so $\partial_{\eta}\phi_{z}(\eta)<0$. Therefore
\begin{align*}
    \phi_{z}(\eta)-\phi_{z}(\eta_{z})&=\phi_{z}(\eta_{0})-\phi_{z}(\eta_{z})-\int_{\eta}^{\eta_{0}}\partial_{\sigma}\phi_{z}(\sigma)d\sigma\\
    &\geq\frac{C(\eta_{0}-\eta_{z})^{2}}{t}\\
    &\geq\frac{C(\eta-\eta_{z})^{2}}{t}.
\end{align*}
\end{proof}

The second fundamental quantity $\psi^{X}_{z}$ is the exponential of $\phi^{X}_{z}$ evaluated at $\eta^{X}_{z}$:
\begin{align}
    \psi^{X}_{z}&=\exp\left\{-\frac{N}{2}\phi^{X}_{z}(\eta^{X}_{z})\right\}=\exp\left\{-\frac{N}{2t}(\eta^{X}_{z})^{2}\right\}\det^{1/2}\left[(\eta^{X}_{z})^{2}+|X_{z}|^{2}\right].\label{eq:psi}
\end{align}
{We will see in Sections \ref{sec:expectation} and \ref{sec:gaussian} that $\log \psi^{X}_{z}$ gives the leading order behaviour of the integral
\begin{align*}
    \log\int_{S^{n-1}}e^{-\frac{N}{2t}\|X_{z}\mbf{v}\|^{2}}\,\mathrm{d}_{H}\mbf{v},
\end{align*}
and the expectation value
\begin{align*}
    \log\mbb{E}_{Y}\left|\det\left(X+\sqrt{\frac{Nt}{n}}Y-z\right)\right|,
\end{align*}
where $Y\sim GinOE(n)$. In fact $\psi^{X}_{z}$ has an analogue in free probability. Let $\mc{A}$ be a non-commutative probability space with tracial state $\phi$ and unit $\mbf{1}$. For $a\in\mc{A}$, let $|a|=(a^{*}a)^{1/2}$ and
\begin{equation*}
    \Delta(a):=\exp\phi(\log |a|)
\end{equation*}
be the Fuglede-Kadison determinant. If $c_{t}$ is a circular element of variance $t$ and $a$ and $c_{t}$ are free, then Bercovici--Zhong \cite[Theorem 5.9]{bercovici_brown_2022} show that
\begin{equation*}
    \Delta(a+c_{t}-z\mbf{1})=\left(\Delta\bigl((a-z\mbf{1})^{*}(a-z\mbf{1})+\eta_{z}^{2}\bigr)\right)^{1/2}e^{-\frac{1}{2t}\eta^{2}_{z}},
\end{equation*}
where $\eta_{z}$ is determined by
\begin{equation*}
    t\phi\left[\bigl((a-z\mbf{1})^{*}(a-z\mbf{1})+\eta^{2}_{z}\bigr)\right]=1.
\end{equation*}
See \cite{bercovici_brown_2022} for more detail on these notions. In our setting this is equivalent to the following. Let $X\in\mbb{M}_{N}$ converge in $*$-distribution to $x\in\mc{A}$. Then for fixed $t>0$ and $z\in\mbb{D}$ we have
\begin{equation*}
    \lim_{N\to\infty}\left(\frac{1}{\psi^{X}_{z}}\mbb{E}\left[\left|\det\left(X+\sqrt{t}Y-z\right)\right|\right]\right)^{1/N}=1,
\end{equation*}
where the expectation is with respect to $Y\sim GinOE(N)$. The function $z\mapsto\eta^{X}_{z}$ is the finite $N$ version of the subordination function $w(0;z,t)$ from \cite[Section 3.1]{zhong_brown_2021}.}

In the sequel we will need to consider $\psi^{X}_{\lambda}$ at nearby points $\lambda=z$ and $\lambda=w$, which is the subject of the following lemma.
\begin{lemma}\label{lem:psi}
Let $z,w\in\mbb{D}_{\omega}$ such that $|z-w|<CN^{-1/2}\log N$. Then
\begin{align}
    \frac{\psi^{X}_{z}}{\psi^{X}_{w}}&=\left[1+O\left(\frac{\log^{3} N}{\sqrt{Nt^{3}}}\right)\right]\exp\left\{-\frac{N\Tr{(G_{z})^{2}B_{z,w}}^{2}}{16\Tr{(\eta^{X}_{z}H_{z})^{2}}}-\frac{1}{2}\tr G_{z}B_{z,w}+\frac{1}{4}\tr(G_{z}B_{z,w})^{2}\right\},\label{eq:psiRatio}
\end{align}
where $B_{z,w}$ is defined by \eqref{eq:Bzw}.
\end{lemma}
\begin{proof}
In this proof we abbreviate $B:=B_{z,w}$. Since $|z-w|<CN^{-1/2}$, we have $|\eta_{z}-\eta_{w}|<CN^{-1/2}t\log N$ by \eqref{eq:eta_zB2}. Consider the ratio of determinants
\begin{align*}
    \frac{\det^{1/2}\left(\eta_{z}^{2}+|X_{z}|^{2}\right)}{\det^{1/2}\left(\eta_{z}^{2}+|X_{w}|^{2}\right)}&=\det^{-1/2}\left(1+G_{z}B\right)\\
    &=\exp\left\{-\frac{1}{2}\tr G_{z}B+\frac{1}{4}\tr(G_{z}B)^{2}-\int_{0}^{1}s^{2}\tr\left(1+sG_{z}B\right)^{-1}(G_{z}B)^{3}ds\right\}.
\end{align*}
The last term can be bounded by Cauchy-Schwarz and \eqref{eq:C3.1}:
\begin{align*}
    \left|\tr(1+sG_{z}B)^{-1}(G_{z}B)^{3}\right|&\leq\left(\tr \frac{\Im G_{z}}{\eta_{z}}B\frac{\Im G_{z}}{\eta_{z}}B\right)^{1/2}\left(\tr BG_{z}B|1+s G_{z}B|^{-2}B(G_{z})^{*}B\right)^{1/2}\\
    &\leq\frac{N|z-w|^{2}}{(\eta_{z})^{3/2}(1-|z-w|/\eta_{z})}\Tr{\Im G_{z}B\Im G_{z}B}^{1/2}\cdot\Tr{\Im G_{z}}\\
    &\leq \frac{C\log^{3} N}{\sqrt{Nt^{3}}}.
\end{align*}
Here we have used the identity
\begin{align*}
    \Tr{\Im G_{z} B\Im G_{z}B}&={2}|z-w|^{2}\Tr{G_{z}FG_{z}F^{*}}
\end{align*}
in order to apply \eqref{eq:C3.1}.

Now consider the ratio
\begin{align*}
    \frac{e^{-N\eta_{z}^{2}/2t}\det^{1/2}\left(\eta_{z}^{2}+|X_{w}|^{2}\right)}{e^{-N\eta_{w}^{2}/2t}\det^{1/2}\left(\eta_{w}^{2}+|X_{w}|^{2}\right)}&=\exp\left\{-\frac{N}{2t}\left[\eta_{z}^{2}-\eta_{w}^{2}\right]+\frac{1}{2}\tr\log\left(1-i(\eta_{z}-\eta_{w})G_{w}\right)\right\}\\
    &=\exp\left\{-\frac{N}{2t}\left[\eta_{z}^{2}-\eta_{w}^{2}\right]+\frac{1}{2}(\eta_{z}-\eta_{w})\Im\tr G_{w}\right.\\
    &\left.+\frac{1}{4}(\eta_{z}-\eta_{w})^{2}\tr(G_{w})^{2}+O\left(\frac{1}{\sqrt{Nt}}\right)\right\}.
\end{align*}
{Since $\Im\tr G_{w}=2N\eta_{w}/t$ by definition, the last line is equal to}
\begin{align*}
    \exp\left\{-\frac{N(1-t\Tr{(G_{w})^{2}}/2)}{2t}(\eta_{z}-\eta_{w})^{2}+O\left(\frac{1}{\sqrt{Nt}}\right)\right\}.
\end{align*}
Using {the equality in} \eqref{eq:eta_zB2} and replacing $G_{w}$ with $G_{z}$ by a combination of the resolvent identity, Cauchy-Schwarz and \eqref{eq:C3.1}, we obtain \eqref{eq:psiRatio}.
\end{proof}

We also need to extend the lower bound in \eqref{eq:C1.2} to nearby resolvents.
\begin{lemma}\label{lem:traceEstimates}
For any $z,w\in\mbb{D}_{\omega}$ and $N^{-\gamma+\epsilon}<\sigma\leq\eta\leq10$, we have
\begin{align}
    \Tr{H_{z}(\eta)H_{w}(\sigma)}&\geq\frac{1}{C\eta^{3}}\left(1-\frac{C|z-w|}{\sqrt{\eta}}\right).\label{eq:traceEstimate1}
\end{align}
\end{lemma}
\begin{proof}
We again abbreviate $B:=B_{z,w}$. Since $\sigma\leq\eta$, we have $H_{w}(\sigma)\geq H_{w}(\eta)$ and so
\begin{align*}
    \Tr{H_{z}(\eta)H_{w}(\sigma)}&\geq\Tr{H_{z}(\eta)H_{w}(\eta)}\\
    &=\frac{1}{\eta^{2}}\Tr{\Im G_{z}(i\eta)E_{2}\Im G_{w}(i\eta)}.
\end{align*}
By the resolvent identity we have
\begin{align*}
    \frac{1}{\eta^{2}}\Tr{\Im G_{z}(i\eta)E_{2}\Im G_{w}(i\eta)}&=\frac{1}{\eta^{2}}\Tr{(\Im G_{z}(i\eta))^{2}E_{2}}\\
    &-\frac{1}{\eta^{2}}\Tr{\Im G_{z}(i\eta)E_{2}\Im\left[G_{w}(i\eta)BG_{z}(i\eta)\right]}.
\end{align*}
The first term is
\begin{align*}
    \frac{1}{\eta^{2}}\Tr{(\Im G_{z}(i\eta))^{2}E_{2}}&=\Tr{(H_{z}(\eta))^{2}}\geq\frac{C}{\eta^{3}}.
\end{align*}
The second term is a sum of traces of the form
\begin{align*}
    \frac{1}{\eta^{2}}\Tr{G_{z}(i\eta_{1})E_{2} G_{w}(i\eta_{2})BG_{z}(i\eta_{3})}
\end{align*}
for $\eta_{j}=\pm\eta$, which we can bound by Cauchy-Schwarz and \eqref{eq:C3.1}:
\begin{align*}
    \frac{1}{\eta^{2}}\Tr{G_{z}(i\eta_{1})E_{2} G_{w}(i\eta_{2})BG_{z}(i\eta_{3})}&\leq\frac{1}{\eta^{2}}\Tr{\frac{\Im G_{z}(i\eta)}{\eta}E_{2}}^{1/2}\cdot\Tr{\frac{\Im G_{w}(i\eta)}{\eta}B\frac{\Im G_{z}(i\eta)}{\eta}B}^{1/2}\\
    &\leq\frac{C|z-w|}{\eta^{7/2}}.
\end{align*}
\end{proof}

Finally, we show that if $X\in\mc{X}_{n}(\gamma,\omega)$ then projections of $X$ onto subspaces of codimension $k$ are in $\mc{X}_{n-k}(\gamma,\omega)$ for small $k$ and large $n$.
\begin{lemma}\label{lem:projection}
Let $k<n^{1-2\gamma}$ {and $U=(U_{k},U_{n-k})\in O(n)$, where $U_{k}\in\mbb{R}^{n\times k}$. Then $U_{n-k}^{T}XU_{n-k}\in\mc{X}_{n-k}(\gamma,\omega)$ uniformly in $U$.}
\end{lemma}
\begin{proof}
Let $X^{(k)}=U_{n-k}^{T}XU_{n-k}$ {and denote by a superscript $(k)$ resolvents of $X^{(k)}$}. The first condition \eqref{eq:C0} on the norm is immediate since $\|X^{(k)}\|\leq\|X\|$. Since the singular values of $X^{(k)}$ interlace those of $X$,
\begin{align*}
    s_{j+2k}(X)\leq s_{j}(X^{(k)})&\leq s_{j}(X),\quad j=1,...,n-k,
\end{align*}
where we set $s_{j}(X)=0$ for $j>n$, we have
\begin{align*}
    \Tr{(H^{(k)}_{z}(\eta))^{m}}&=\Tr{(H_{z}(\eta))^{m}}+O\left(\frac{k}{n\eta^{2m}}\right),
\end{align*}
and so \eqref{eq:C1.1} and \eqref{eq:C1.2} follow.

Now let $B_{j}\in\mbb{M}_{n}$ and $B^{(k)}_{j}=U_{n-k}^{T}B_{j}U_{n-k}$. By Lemma \ref{lem:minorresolvent}, we have
\begin{align}
    U\begin{pmatrix}0&0\\0&G^{(k)}_{z}(i\eta)\end{pmatrix}U^{T}&=G_{z}(i\eta)-G_{z}(i\eta)U_{k}(U_{k}^{T}G_{z}(i\eta)U_{k})^{-1}U_{k}^{T}G_{z}(i\eta),\label{eq:GY}
\end{align}
and 
\begin{align*}
    \left\|G_{z}(i\eta)U_{k}(U_{k}^{T}G_{z}(i\eta)U_{k})^{-1}U_{k}^{T}G_{z}(i\eta)\right\|&\leq\frac{1}{\eta}.
\end{align*}
Thus, writing
\begin{align*}
    \Tr{G^{(k)}_{z}(i\eta_{1})B^{(k)}_{1}G^{(k)}_{z}(i\eta_{2})B^{(k)}_{2}}&=\Tr{U\begin{pmatrix}0&0\\0&G^{(k)}_{z}(i\eta_{1})\end{pmatrix}U^{T}B_{1}U\begin{pmatrix}0&0\\0&G^{(k)}_{z}(i\eta_{2})\end{pmatrix}U^{T}B_{2}}
\end{align*}
and noting that the error in \eqref{eq:GY} has rank $k$, we obtain
\begin{align*}
    \left|\Tr{G^{(k)}_{z}(i\eta_{1})B^{(k)}_{1}G^{(k)}_{z}(i\eta_{2})B^{(k)}_{2}}-\Tr{G_{z}(i\eta_{1})B_{1}G_{z}(i\eta_{2})B_{2}}\right|&\leq\frac{Ck\|B_{1}\|\|B_{2}\|}{n\eta_{1}\eta_{2}},
\end{align*}
which implies the conditions in \eqref{eq:C2.1},\eqref{eq:C2.2}, \eqref{eq:C3.1} and \eqref{eq:C3.2} hold for $Y$.
\end{proof}

\section{Real Partial Schur Decomposition}\label{sec4}
In this section we follow Edelman--Kostlan--Shub \cite{edelman_how_1994} and Edelman \cite{edelman_probability_1997}. Let $X\in\mbb{M}_{n}(\mbb{R})$ have distinct eigenvalues, which are either real or complex conjugate pairs. If $\mbf{v}\in S^{n-1}$ is an eigenvector associated to a real eigenvalue $u$, then we have
\begin{align}
    RXR&=\begin{pmatrix}u&\mbf{w}^{T}\\0&X^{(1,0)}\end{pmatrix},
\end{align}
where $R$ is the Householder transformation that exchanges the first coordinate vector $\mbf{e}_{1}$ with $\mbf{v}$, {i.e.
\begin{align*}
    R&=1_{n}-2\frac{(\mbf{e}_{1}-\mbf{v})(\mbf{e}_{1}-\mbf{v})^{*}}{\|\mbf{e}_{1}-\mbf{v}\|^{2}}.
\end{align*}}
Choosing the first element of $\mbf{v}$ to be positive, this map is unique up to the choice of eigenvalue $u$. The Jacobian is given by (Lemma 3.3 in \cite{edelman_how_1994})
\begin{align}
    J_{r}(u,X^{(1,0)})&=\left|\det(X^{(1,0)}-u)\right|,\label{eq:JR}
\end{align}
with respect to \[\diff u\diff \mbf{w}\diff X^{(1,0)}\diff _{H}\mbf{v}.\]

Let $z=x+iy$ with $y>0$ be a complex eigenvalue of $X$. Then there is an orthogonal projection $V$ onto the two-dimensional subspace spanned by the eigenvectors corresponding to $z$ and $\bar{z}$ such that
\begin{align}
    Q^{T}XQ&=\begin{pmatrix}Z&W^{T}\\0&X^{(0,1)}\end{pmatrix},
\end{align}
where $Q$ is an orthogonal matrix whose first two columns are $V$,
\begin{align}
    Z&=\begin{pmatrix}x&b\\-c&x\end{pmatrix},
\end{align}
and $b\geq c$, $y=\sqrt{bc}$. {Note that $b\geq c$ includes both $b\geq c\geq0$ and $c\leq b\leq 0$.} Choosing the first element of $V$ to be positive, this is again unique up to the choice of eigenvalue $z$. Note that in \cite{edelman_probability_1997} it is stated that $V$ is unique after fixing the sign of the first row, but if $Z$ is conjugated by a reflection $\text{diag}(1,-1)$ the resulting matrix has $b'\leq c'$, i.e. it is no longer of the same form. Therefore, once we have specified $Z$, we can fix the sign of one element of the first row but not both. This explains the missing factor of 2 in \cite{edelman_probability_1997}, which reappears by mistake in the evaluation of eq. (14) in that paper so that the end result is still correct. The Jacobian of this transformation is (Theorem 5.2 in \cite{edelman_probability_1997})
\begin{align}
    \wt{J}_{c}(x,b,c,X^{(0,1)})&=2(b-c)\det\left[(X^{(0,1)}-x)^{2}+y^{2}\right],
\end{align}
with respect to \[\diff x\diff b\diff c\diff W\diff X^{(0,1)}\diff _{H}V.\] It is convenient to change variables from $(b,c)$ to $(y,\delta)$, where $\delta=b-c\geq0$. The combined  Jacobian is
\begin{align}
    J_{c}(x,y,\delta,X^{(0,1)})&=\frac{4y\delta}{\sqrt{\delta^{2}+4y^{2}}}\left|\det(X^{(0,1)}-z)\right|^{2}.\label{eq:JC}
\end{align}

Composing these two transformations $p$ and $q$ times respectively we obtain the $(p,q)$-partial Schur decomposition:
\begin{align}
    X&=U\begin{pmatrix}
    u_{1}&\mbf{w}_{1}^{T}&&&&&&\\
    0&u_{2}&\mbf{w}_{2}^{T}&&&&&\\
    &\ddots&\ddots&\ddots&&&&\\
    &&\ddots&u_{p}&\mbf{w}_{p}^{T}&&&\\
    &&&0&Z_{1}&W_{1}^{T}&&\\
    &&&&0&\ddots&\ddots&\\
    &&&&&\ddots&Z_{q}&W_{q}^{T}\\
    &&&&&&0&X^{(p,q)}
    \end{pmatrix}U^{T}.
\end{align}
The matrix $U$ is an orthogonal matrix constructed from $R_{j}$ and $Q_{j}$ which depend on $\mbf{v}_{j}$ and $V_{j}$ respectively. The variables range over the following sets:
\begin{enumerate}[i)]
\item $u_{j}\in\mbb{R},\,j=1,...,p$;
\item $\mbf{v}_{j}\in S^{n-j}_{+},\,j=1,...,p$;
\item $\mbf{w}_{j}\in\mbb{R}^{n-j}$;
\item $x_{k}\in\mbb{R},\,y_{k}>0,\,\delta_{k}\geq0,\,k=1,...,q$;
\item $V_{k}\in O_{+}(n-p-2k,2),\,k=1,...,q$;
\item $W_{k}\in\mbb{R}^{(n-p-2k)\times2},\,k=1,...,q$;
\item $X^{(p,q)}\in\mbb{M}_{n-p-2q}(\mbb{R})$,
\end{enumerate}   
where $S^{n-1}_{+}$ and $O_{+}(n,2)$ are the respective subsets of $S^{n-1}$ and $O(n,2)$ with positive first element. We denote by $\Omega$ the product of these sets and $\omega\in\Omega$ a generic element. Strictly speaking the upper triangular entries $\mbf{w}_{j}^{T}$ and $W_{k}^{T}$ should be multiplied by orthogonal matrices constructed from $R_{j}$ and $Q_{k}$, but these can be removed by rotation. {To obtain the Jacobian, we take the product of \eqref{eq:JR} and \eqref{eq:JC} and use
\begin{align*}
    \det(X^{(j-1,0)}-\lambda)&=(u_{j}-\lambda)\det(X^{(j,0)}-\lambda),\\
    \det(X^{(p,k-1)}-\lambda)&=\det(Z_{k}-\lambda)\det(X^{(p,k)}-\lambda)=|z_{k}-\lambda|^{2}\det(X^{(p,k)}-\lambda).
\end{align*}
We obtain}
\begin{align}
    J^{(p,q)}(\omega)&=2^{q}\left(\prod_{k=1}^{q}\frac{2y_{k}\delta_{k}}{\sqrt{\delta_{k}^{2}+y_{k}^{2}}}\right)\left|\wt{\Delta}(\mbf{u},\mbf{z})\right|\prod_{\mu=1}^{p+2q}|\det(X^{(p,q)}-\lambda_{\mu})|,\label{eq:Jnl}
\end{align}
with respect to
\begin{align}
    \diff \omega&=\left(\prod _{i=1}^{p}\diff u_{i}\diff \mbf{w}_{i}\diff _{H}\mbf{v}_{i}\right)\left(\prod _{j=1}^{q}\diff z_{j}\diff \delta_{j}\diff W_{j}\,\mathrm{d}_{H}V_{j}\right)\diff X^{(p,q)},\label{eq:domega}
\end{align}
where $\bs\lambda=(u_{1},...,u_{p},z_{1},...,z_{q},\bar{z}_{1},...,\bar{z}_{q})$ and
\begin{align}
    \wt{\Delta}(\mbf{u},\mbf{z})&=\left(\prod_{j<l}^{p}(u_{j}-u_{l})\right)\cdot\left(\prod_{k<l}^{q}|z_{k}-z_{l}|^{2}|z_{k}-\bar{z}_{l}|^{2}\right)\cdot\left(\prod_{j=1}^{p}\prod_{k=1}^{q}|u_{j}-z_{k}|^{2}\right).
\end{align}

Let $\pi:\Omega\to\mbb{R}^{p}\times\mbb{C}^{q}_{+}$ denote the projection $\pi(\omega)=(\mbf{u},\mbf{z})$. Following \cite[Proposition 4.1]{maltsev_bulk_2023}, we have
\begin{proposition}\label{prop:partialSchur}
Let $X\in\mbb{M}_{n}(\mbb{R})$ be a random matrix with density $\rho$ with respect to $\diff X$. Then for any bounded and measurable $f:\mbb{R}^{p}\times\mbb{C}^{q}_{+}\to\mbb{C}$ we have
\begin{align}
    \mbb{E}\left[\sum f(u_{j_{1}},...,u_{j_{p}},z_{k_{1}},...,z_{k_{q}})\right]&=\int_{\Omega}f(\pi(\omega))\rho(X(\omega))J^{(p,q)}(\omega)\diff\omega,\label{eq:klpoint}
\end{align}
where the sum is over distinct tuples $1\leq j_{1}<\cdots<j_{p}\leq N_{\mbb{R}}$ and $1\leq k_{1}<\cdots<k_{q}\leq N_{\mbb{C}}$.
\end{proposition}

\section{Proof of Theorem \ref{thm1}}\label{sec:thm1proof}
{We start with a broad overview of the proof strategy. Proposition \ref{prop:partialSchur} gives us a formula for the correlation functions of matrices with a continuous distribution. In particular, for fixed $A$ and Gaussian $B$, we obtain a formula for the correlation functions of $A_{t}:=A+\sqrt{t}B$. The formula is an integral of the product of the density of $A_{t}$ and the Jacobian $J^{(p,q)}$ from \eqref{eq:Jnl} with respect to all the variables of the partial Schur decomposition except $\{\mbf{u},\mbf{z}\}$ (which are the arguments of the correlation function). We want to proceed as in the case of complex matrices treated in \cite{maltsev_bulk_2023}, where one of the main steps is to represent the determinants in the Jacobian as integrals over anti-commuting variables. However, for each real argument $u_{j}$ we have the absolute value of the determinant, which cannot be directly represented in this way.}

The first step towards the proof of Theorem \ref{thm1} is therefore the removal of these absolute values, which we do following an idea of Tribe and Zaboronski \cite{tribe_ginibre_2014-1,tribe_averages_2023}. Consider the ``spin" variables $s_{u}(X)$ defined for $u\in\mbb{R}$ by
\begin{align}
    s_{u}(X)&=\frac{\det(X-u)}{|\det(X-u)|},\label{eq:spin}
\end{align}
where if $u$ is an eigenvalue of $X$ we define $s_{u}(X)$ by the limit from {the left} so that $s_{u}$ is left-continuous in $u$ (this choice is not important). We have the following identity.
\begin{lemma}
Let $X\in\mbb{M}_{N}(\mbb{R})$ such that the real eigenvalues $u_{j},\,j=1,...,N_{\mbb{R}}$ are distinct; then 
\begin{align}
    s_{u_{j}}(X)&=(-1)^{N-1}-2\sum_{u_{k}>u_{j}}s_{u_{k}}(X).\label{eq:spineq}
\end{align}
\end{lemma}
\begin{proof}
The fact that the real eigenvalues are distinct and the complex eigenvalues come in conjugate pairs implies that $s_{u_{j}}(X)=(-1)^{p_{j}-1}$, where $p_{j}=\left|\{k:u_{k}\leq u_{j}\}\right|$. We also have
\begin{align*}
    \sum_{u_{k}>u_{j}}s_{u_{k}}(X)&=\begin{cases}0&\quad N_{\mbb{R}}-p_{j}\text{ even}\\
    -s_{u_{j}}(X)&\quad N_{\mbb{R}}-p_{j}\text{ odd}
    \end{cases},
\end{align*}
since there are $N_{\mbb{R}}-p_{j}$ terms in the sum which alternate between $-s_{u_{j}}(X)$ and $+s_{u_{j}}(X)$; this again requires the real eigenvalues to be distinct. If $N_{\mbb{R}}-p_{j}$ is even, then 
\begin{align*}
    s_{u_{j}}(X)&=(-1)^{p_{j}-1}=(-1)^{N_{\mbb{R}}-1}=(-1)^{N_{\mbb{R}}-1}-2\sum_{u_{k}>u_{j}}s_{u_{k}}(X),
\end{align*}
and if $N_{\mbb{R}}-p_{j}$ is odd then 
\begin{align*}
    s_{u_{j}}(X)&=(-1)^{N_{\mbb{R}}-1}+2s_{u_{j}}(X)=(-1)^{N_{\mbb{R}}-1}-2\sum_{u_{k}>u_{j}}s_{u_{k}}(X).
\end{align*}
Finally, we note that since there are an even number of complex eigenvalues, we have $(-1)^{N_{\mbb{R}}}=(-1)^{N}$.
\end{proof}

This identity allow us to remove the absolute values from the determinants that appear in the Jacobian of the partial Schur decomposition. Let $p,l\in\mbb{N}$, $\mbf{a}=(a_{1},...,a_{l})\in\mbb{N}^{l}$ such that $1\leq a_{1}<\cdots<a_{l}\leq p$ and $\Omega^{(p)}_{\mbf{a}}$ be defined by
\begin{align}
    \Omega^{(p)}_{\mbf{a}}=\bigcup_{\sigma\in S_{p}}\big\{(\mbf{x},\mbf{y})\in\mbb{R}^{p+l}:x_{\sigma_{1}}&<\cdots<x_{\sigma_{a_{1}}}<y_{1}<x_{\sigma_{a_{1}}+1}\nonumber\\
    &<\cdots<x_{\sigma_{a_{l}}}<y_{l}<x_{\sigma_{a_{l}+1}}<\cdots<x_{\sigma_{p}}\big\},\label{eq:Omega}
\end{align}
where the union is over permutations of $p$ elements and we set $x_{\sigma_{p+1}}=\infty$. In words, $\Omega^{(p)}_{\mbf{a}}$ is such that the $(p+j)$-th coordinate lies between the $a_{j}$-th and $(a_{j}+1)$-th smallest elements of the first $p$ coordinates. When $a_{l}=p$, then the final coordinate $y_{p+l}$ lies in $(\max_{1\leq j\leq p}x_{j},\infty)$.

Let $f\in C_{c}(\mbb{R}^{m_{r}}\times\mbb{C}^{m_{c}}_{+})$, $\chi_{\Omega^{(m_{r})}_{\mbf{a}}}$ be the indicator function of $\Omega^{(m_{r})}_{\mbf{a}}$ and consider the following statistics
\begin{align}
    S^{(m_{r},m_{c})}(f)&=\sum f(u_{j_{1}},...,u_{j_{m_{r}}},z_{k_{1}},...,z_{k_{m_{c}}}),\label{eq:Sf}\\
    S^{(m_{r},m_{c})}_{\mbf{a}}(f)&=\sum f(u_{j_{1}},...,u_{j_{m_{r}}},z_{k_{1}},...,z_{k_{m_{c}}})\chi_{\Omega^{(m_{r})}_{\mbf{a}}}(u_{j_{1}},...,u_{j_{m_{r}+l}})\prod_{j=1}^{m_{r}+l}s_{u_{j}},\label{eq:Sf_a}
\end{align}
where the sums are over tuples $\{j_{q}\}\in[1,...,N_{\mbb{R}}]$ and $\{k_{q}\}\in[1,...,N_{\mbb{C}}]$ of distinct indices. From the definition of $\rho^{(m_{r},m_{c})}_{X}$ we have
\begin{align*}
    \mbb{E}\left[S^{(m_{r},m_{c})}(f)\right]&=\int_{\mbb{R}^{m_{r}}\times\mbb{C}^{m_{c}}_{+}}f(\mbf{u},\mbf{z})\rho^{(m_{r},m_{c})}_{X}(\mbf{u},\mbf{z})\diff\mbf{u}\diff\mbf{z},
\end{align*}
and we define $\rho^{(m_{r},m_{c})}_{X,\mbf{a}}$ analogously by
\begin{align}
    \mbb{E}\left[S^{(m_{r},m_{c})}_{\mbf{a}}(f)\right]&=\int_{\mbb{R}^{m_{r}}\times\mbb{C}^{m_{c}}_{+}}f(\mbf{u},\mbf{z})\rho^{(m_{r},m_{c})}_{X,\mbf{a}}(\mbf{u},\mbf{z})\diff\mbf{u}\diff\mbf{z}.\label{eq:rho_a}
\end{align}
We come to the central lemma relating $\rho^{(m_{r},m_{c})}_{X}$ and $\rho^{(m_{r},m_{c})}_{X,\mbf{a}}$.
\begin{lemma}\label{lem:rhoEquality}
Let $X$ have distinct eigenvalues with probability 1. Then
\begin{align}
    \rho^{(m_{r},m_{c})}_{X}&=\sum_{\mbf{a}}\varepsilon_{\mbf{a}}\rho^{(m_{r},m_{c})}_{X,\mbf{a}}\label{eq:rhoEquality}
\end{align}
pointwise almost everywhere, where the sum is over tuples $\mbf{a}$ of length $l=0,...,m_{r}$ such that $1\leq a_{1}<\cdots<a_{l}\leq m_{r}$, $a_{1}$ and $a_{j+1}-a_{j}$ are odd for $j=1,...,l-1$, and
\begin{align}
    \varepsilon_{\mbf{a}}&=\begin{cases}
    (-1)^{\frac{m_{r}-l}{2}}\cdot 2^{l}&\quad m_{r}-l\text{ even}\\
    (-1)^{N+\frac{m_{r}-l+1}{2}}\cdot 2^{l}&\quad m_{r}-l\text{ odd}
    \end{cases}.
\end{align}
\end{lemma}
\begin{proof}
Consider one term in the summand of \eqref{eq:Sf} with $u_{j_{1}}\neq\cdots\neq u_{j_{m_{r}}}$ and let $p_{k}$ be the index of the $k$th smallest value in $\{u_{j_{1}},...,u_{j_{m_{r}}}\}$. For the smallest value $u_{p_{1}}$ we use \eqref{eq:spineq} to introduce the factor $1$ in the following form
\begin{align*}
    1&=s_{u_{p_{1}}}^{2}\\
    &=s_{u_{p_{1}}}\left[(-1)^{N-1}-2\sum_{u_{j_{m_{r}}+1}>u_{p_{1}}}s_{u_{j_{m_{r}}+1}}\right]\\
    &=-s_{u_{p_{1}}}\left(s_{u_{p_{2}}}+2\sum_{u_{p_{1}}<u_{j_{m_{r}+1}}<u_{p_{2}}}s_{u_{j_{m_{r}+1}}}\right),
\end{align*}
which gives us two terms for which we perform the above step again with $1=s_{u_{p_{3}}}^{2}$ and $1=s_{u_{p_{2}}}^{2}$ respectively. We repeat this procedure until each resulting term has a product of spin variables $s_{u_{j_{1}}}\cdots s_{u_{j_{m_{r}}+l}}$ for different values of $l=0,...,m_{r}$. We write the constrained sum as a sum over all distinct tuples weighted by an indicator function. This leads to the equality
\begin{align*}
    S^{(m_{r},m_{c})}(f)&=\sum_{\mbf{a}}\varepsilon_{\mbf{a}}S^{(m_{r},m_{c})}_{\mbf{a}}(f),
\end{align*}
almost surely, where the sum is over tuples $\mbf{a}$ of varying size $l=0,...,m_{r}$ ($l=0$ means the empty tuple). Since each step introduces either a term $s_{p_{j}}s_{p_{j+1}}$ or a term 
\begin{align*}
    \sum_{u_{p_{j}}<u_{m_{r}+k}<u_{p_{j+1}}}s_{p_{j}}s_{m_{r}+k},
\end{align*}
$a_{1}$ must be odd; starting again from $s_{a_{1}+1}$, we see that $a_{2}$ must be even, $a_{3}$ must be odd and so on.

Equating expectation values and taking the coefficients $\varepsilon_{\mbf{a}}$ outside since they are constant, we obtain
\begin{align*}
    \int_{\mbb{R}^{m_{r}}\times\mbb{C}^{m_{c}}}f(\mbf{u},\mbf{z})\left[\rho^{(m_{r},m_{c})}_{X}(\mbf{u},\mbf{z})-\sum_{\mbf{a}}\varepsilon_{\mbf{a}}\rho^{(m_{r},m_{c})}_{X,\mbf{a}}(\mbf{u},\mbf{z})\right]\diff\mbf{u}\diff\mbf{z}&=0,
\end{align*}
for any bounded and measurable $f$, and so \eqref{eq:rhoEquality} holds almost everywhere.
\end{proof}

The are two main points to note about the formula in \eqref{eq:rhoEquality}. Firstly, we can deduce universality of $\rho^{(m_{r},m_{c})}_{X}$ from that of $\rho^{(m_{r},m_{c})}_{X,\mbf{a}}$, which we will see are easier to study. Secondly, from the condition on the tuples $\mbf{a}$ that contribute to the sum we note that if $a_{l}=m_{r}$ then $m_{r}+l$ must be even. This is important because $\Omega^{(m_{r})}_{\mbf{a}}$ is unbounded when $a_{l}=m_{r}$: it contains the semi-infinite interval $(\max_{j\leq m_{r}}u_{j},\infty)$. When $a_{l}<m_{r}$, then the $u_{m_{r}+j},\,j=1,...,l$ are bounded below by $\min_{j\leq m_{r}}u_{j}$ and above by $\max_{j\leq m_{r}}u_{j}$. Since $u_{j},\,j=1,...,m_{r}$ are restricted to the compact support of the test function $f$, this means that when $a_{l}<m_{r}$ we only need to consider $\rho^{(m_{r},m_{c})}_{X,\mbf{a}}$ on compact sets. Besides these two points, the precise form of the identity relating $\rho^{(m_{r},m_{c})}_{X}$ and $\rho^{(m_{r},m_{c})}_{X,\mbf{a}}$ is not important.

We are now in a position to apply the partial Schur decomposition. Let $A_{t}=A+\sqrt{t}B$ and
\begin{align}
    A_{t}&=U\begin{pmatrix}
    u_{1}&\mbf{w}_{1}^{T}&&&&&&\\
    0&u_{2}&\mbf{w}_{2}^{T}&&&&&\\
    &\ddots&\ddots&\ddots&&&&\\
    &&\ddots&u_{m_{r}+l}&\mbf{w}_{m_{r}+l}^{T}&&&\\
    &&&0&Z_{1}&W_{1}^{T}&&\\
    &&&&0&\ddots&\ddots&\\
    &&&&&\ddots&Z_{m_{c}}&W_{m_{c}}^{T}\\
    &&&&&&0&A_{t}^{(m_{r}+l,m_{c})}
    \end{pmatrix}U^{T}
\end{align}
be the $(m_{r}+l,m_{c})$ partial Schur decomposition of $A_{t}$; from the definition in \eqref{eq:spin} we have
\begin{align}
    s_{u_{j}}(A_{t})&=\lim_{\delta\downarrow0}\frac{\det(A_{t}-u_{j}+\delta)}{|\det(A_{t}-u_{j}+\delta)|}\nonumber\\
    &=\left(\prod_{k\neq j}\frac{\lambda_{k}-u_{j}}{|\lambda_{k}-u_{j}|}\right)\frac{\det(A_{t}^{(m_{r}+l,m_{c})}-u_{j})}{|\det(A_{t}^{(m_{r}+l,m_{c})}-u_{j})|},\quad j=1,...,m_{r}+l,
\end{align}
where $\bs\lambda=(u_{1},...,u_{m_{r}+l},z_{1},...,z_{m_{c}},\bar{z}_{1},...,\bar{z}_{m_{c}})$. The summand in the definition of $S^{(m_{r},m_{c})}_{X,\mbf{a}}(f)$ in \eqref{eq:Sf_a} is a bounded and measurable function on $\mbb{R}^{m_{r}+l}\times\mbb{C}^{m_{c}}_{+}$. We can therefore apply the formula in \eqref{eq:klpoint} to obtain
\begin{align*}
    \mbb{E}\left[S^{(m_{r},m_{c})}_{A_{t},\mbf{a}}(f)\right]&=\int f(\mbf{u},\mbf{z})\chi_{\Omega^{(m_{r})}_{\mbf{a}}}(\mbf{u},\mbf{u}')\left(\prod_{j=1}^{m_{r}+l}s_{u_{j}}(A_{t}(\omega))\right)\\&\times\rho(A_{t}(\omega))J^{(m_{r}+l,m_{c})}(\omega)\diff\omega_{0}\diff\mbf{u}'\diff\mbf{u}\diff\mbf{z},
\end{align*}
where $\omega_{0}$ is shorthand for all the variables of the partial Schur decomposition except $\{\mbf{u},\mbf{u}',\mbf{z}\}$, and $\omega=(\mbf{u},\mbf{u}',\mbf{z},\omega_{0})$. We note that since $f\cdot\chi\cdot\prod_{j}s_{u_{j}}$ is bounded we can integrate in any order by Fubini's theorem. We read off the following formula
\begin{align}
    \rho^{(m_{r},m_{c})}_{A_{t},\mbf{a}}(\mbf{u},\mbf{z})&=\int_{\pi(\Omega^{(m_{r})}_{\mbf{a}})}\xi^{(m_{r}+l,m_{c})}(u_{1},...,u_{m_{r}+l},\mbf{z})\diff u_{m_{r}+l}\cdots \diff u_{m_{r}+1},\label{eq:rhoHatToRho}
\end{align}
where $\pi:\mbb{R}^{m_{r}+l}\mapsto\mbb{R}^{l}$ is the projection onto the last $l$ coordinates,
\begin{align}
    \xi^{(p,q)}(\mbf{u},\mbf{z})&=\wt{\Delta}(\mbf{u},\mbf{z})\int\rho(A_{t}(\omega))\left(\prod_{j=1}^{p+2q}\det\left[A_{t}^{(p,q)}-\lambda_{j}\right]\right)\diff \omega_{0},\label{eq:xi}
\end{align}
and $\bs\lambda=(\mbf{u},\mbf{z},\bar{\mbf{z}})$. For concereteness we fix without loss an ordering $u_{1}<\cdots<u_{m_{r}}$, in which case \eqref{eq:rhoHatToRho} becomes
\begin{align}
    \rho^{(m_{r},m_{c})}_{A_{t},\mbf{a}}(\mbf{u},\mbf{z})&=\int_{u_{a_{1}}}^{u_{a_{1}+1}}\cdots\int_{u_{a_{l}}}^{u_{a_{l}+1}}\xi^{(m_{r}+l,m_{c})}(u_{1},...,u_{m_{r}+l},\mbf{z})\diff u_{m_{r}+l}\cdots \diff u_{m_{r}+1},\label{eq:rhoHatToRho1}
\end{align}
if $a_{l}<m_{r}$, and
\begin{align}
    \rho^{(m_{r},m_{c})}_{A_{t},\mbf{a}}(\mbf{u},\mbf{z})&=\int_{u_{a_{1}}}^{u_{a_{1}+1}}\cdots\int_{u_{a_{l}}}^{\infty}\xi^{(m_{r}+l,m_{c})}(u_{1},...,u_{m_{r}+l},\mbf{z})\diff u_{m_{r}+l}\cdots \diff u_{m_{r}+1},\label{eq:rhoHatToRho3}
\end{align}
if $a_{l}=m_{r}$.

The first part of the theorem, namely the existence of $\eta_{\lambda}$ satisfying \[t\Tr{H^{A}_{\lambda}(\eta_{\lambda})}=1\] for any $\lambda\in\mbb{D}_{\omega}$, follows from Lemma \ref{lem:phi}. {We can then define $\sigma_{\lambda}:=\sigma^{A}_{\lambda}$ by \eqref{eq:sigma} and the rescaled function
\begin{align}
    \xi^{(p,q)}_{\lambda}(\mbf{u},\mbf{z})&=\frac{1}{(N\sigma_{\lambda})^{p/2+q}}\xi^{(p,q)}\left(\lambda+\frac{\mbf{u}}{\sqrt{N\sigma_{\lambda}}},\lambda+\frac{\mbf{z}}{\sqrt{N\sigma_{\lambda}}}\right).\label{eq:xi_r}
\end{align}}
The proof of Theorem \ref{thm1} in the real bulk amounts to the following pointwise asymptotics for $\xi^{(p,q)}_{r}$.
\begin{lemma}\label{lem:xi_r}
Let $\omega>0$ and $\lambda=u\in(-1+\omega,1-\omega)$. The following estimates hold uniformly in any compact subset of $\mbb{R}^{p-1}\times\mbb{C}_{+}^{2q}$:
\begin{enumerate}[i)]
\item if $p$ is odd and $|u_{p}|<C$, then
\begin{align}
    \xi^{(p,q)}_{u}(\mbf{u},\mbf{z})&\leq e^{-c\log^{2}N};\label{eq:xi1}
\end{align}
\item if $p$ is even and $|u_{p}|>4\sqrt{N\sigma_{u}}\|A\|$, then
\begin{align}
    \xi^{(p,q)}_{u}(\mbf{u},\mbf{z})&\leq e^{-\frac{cN}{t}u_{p}^{2}};\label{eq:xi2}
\end{align}
\item if $p$ is even and $\log N<|u_{p}|<4\sqrt{N\sigma_{u}}\|A\|$ then
\begin{align}
    \xi^{(p,q)}_{u}(\mbf{u},\mbf{z})&\leq e^{-c\log^{2}N};\label{eq:xi3}
\end{align}
\item if $p$ is even and $|u_{p}|<\log N$ then
\begin{align}
    &\xi^{(p,q)}_{u}(\mbf{u},\mbf{z})=\left[1+O\left(\frac{\log N}{\sqrt{Nt^{3}}}\right)\right]\frac{1}{2^{p}}\left[\prod_{j=1}^{q}2\Im z_{j}\textup{erfc}\left(\sqrt{2}\Im z_{j}\right)\right]\frac{I_{m}(\mbf{u},\mbf{z})}{\pi^{m/2}},\label{eq:xi4}
\end{align}
where $m=p+2q$ and $I_{m}$ is defined in \eqref{eq:Im}.
\end{enumerate}
\end{lemma}
In the complex bulk, we use the convention that an empty product and an empty integral are equal to 1, so that we can define 
\begin{align}
    \xi^{(q)}(\mbf{z}):=\xi^{(0,q)}(\emptyset,\mbf{z}),
\end{align}
and
\begin{align}
    \xi^{(q)}_{\lambda}(\mbf{z})&=\xi^{(0,q)}_{\lambda}(\emptyset,\mbf{z}).
\end{align}
\begin{lemma}\label{lem:xi_c}
Let $\omega>0$ and $\lambda=z\in\mbb{C}_{+}$ such that $\Im z>\omega>0$. Let $K_{c}\subset\mbb{C}^{q}$ be compact; then we have
\begin{align}
    \xi^{(q)}_{z}(\mbf{z})&=\left[1+O\left(\frac{\log N}{\sqrt{Nt^{3}}}\right)\right]\det\left[\frac{1}{\pi}e^{-|z_{j}|^{2}-|z_{k}|^{2}+\bar{z}_{j}z_{k}}\right]_{j,k=1}^{q},\label{eq:xi_c}
\end{align}
uniformly on $K_{c}$ for sufficiently large $N$.
\end{lemma}

We postpone the proofs of these lemmas to the next section and prove Theorem \ref{thm1}.
\begin{proof}[Proof of Theorem \ref{thm1}]
If we extend the definition of $I_{m}$ to odd $m=2k+1$ by setting $I_{2k+1}=0$, then it follows from Lemma \ref{lem:xi_r} that for any compact $K_{r}\subset \mbb{R}^{m_{r}}\times\mbb{C}^{m_{c}}_{+}$ there is a $\delta>0$ such that
\begin{align}
    &\frac{1}{(N\sigma_{u})^{m/2}}\rho^{(m_{r},m_{c})}_{A_{t},\mbf{a}}\left(u+\frac{\mbf{u}}{\sqrt{N\sigma_{u}}},u+\frac{\mbf{z}}{\sqrt{N\sigma_{u}}}\right)=\frac{1}{2^{p}}\prod_{j=1}^{m_{c}}2\Im z_{j}\textup{erfc}\left(\sqrt{2}\Im z_{j}\right)\nonumber\\
    &\times\frac{1}{\pi^{m/2}}\int_{\Omega^{(m_{r})}_{\mbf{a}}}I_{m}(u_{1},...,u_{m_{r}+l},z_{1},...,z_{m_{c}})\diff u_{m_{r}+l}\cdots \diff u_{m_{r}+1}+O(N^{-\delta}),
\end{align}
uniformly on $K_{r}$. Indeed, if $m_{r}+l$ is odd or $m_{r}+l$ is even and $a_{l}<m_{r}$, then $\Omega^{(m_{r})}_{\mbf{a}}$ {contains only bounded intervals (see the discussion following Lemma \ref{lem:rhoEquality}) and so the integral of the error is of the same order as the error itself}. If $m_{r}+l$ is even and $a_{l}=m_{r}$, then {$\{u_{j}\}_{j=m_{r}+1}^{m_{r}+l-1}$ is integrated over a product of bounded intervals and $u_{m_{r}+l}$ is integrated over $(u_{m_{r}},\infty)$.} We can truncate the integral over $u_{m_{r}+l}$ to the interval $(u_{m_{r}},u_{m_{r}}+\log N)$ using \eqref{eq:xi2} and \eqref{eq:xi3}, replace $\xi^{(m_{r}+l,m_{c})}_{u}$ with the asymptotic in \eqref{eq:xi4}, then extend the integral back to $(u_{m_{r}},\infty)$ using the fact that $I_{m}$ decays as $e^{-|u_{m_{r}+l}|^{2}}$ for large $u_{m_{r}+l}$.

Using the formula for $\rho^{(m_{r},m_{c})}_{A_{t}}$ in terms of $\rho^{(m_{r},m_{c})}_{A_{t},\mbf{a}}$ in \eqref{eq:rhoEquality}, we conclude that the correlation functions have a universal limit independent of $A\in\mc{X}_{N}(\gamma,\omega)$. In principle we should repeat the calculation when $A=0$ and $t=1$ in order to conclude that the limit obtained is indeed that of the GinOE. To save ourselves the repeat calculation, we observe that by Lemma \ref{lem:AinX} below, matrices with i.i.d. entries belong to $\mc{X}_{N}(\gamma,\omega)$ with probability $1-N^{-D}$ for any $D$. In particular, $B\sim GinOE(N)\in\mc{X}_{N}(\gamma,\omega)$ with high probability. This proves \eqref{eq:realBulk}.

{In the complex bulk, we note that the use of spin variables is not necessary: we have an equality
\begin{align*}
    \rho^{(q)}_{A_{t}}(\mbf{z})&=\xi^{(q)}(\mbf{z}).
\end{align*}
Combined with the asymptotics in Lemma \ref{lem:xi_c}, we obtain \eqref{eq:complexBulk}.}
\end{proof}

\section{Proofs of Lemmas \ref{lem:xi_r} and \ref{lem:xi_c}}\label{sec:xi}
Recall the expression for $\xi^{(p,q)}$ in \eqref{eq:xi}, which involves an integral of a product of determinants with respect to the density $\rho(A_{t})$ over the variables of the partial Schur decomposition (except $\{\mbf{u},\mbf{z}\}$. Thus we need an expression for $\rho$ in terms of these variables. Since we have fixed $A$, $\rho$ is a shifted Gaussian:
\begin{align*}
    \rho(A_{t})&=\left(\frac{2\pi t}{N}\right)^{N^{2}/2}\exp\left\{-\frac{N}{2t}\left\|A_{t}-A\right\|^{2}_{2}\right\}.
\end{align*}
To obtain the desired expression we need to conjugate $A$ by the orthogonal $U$ that decomposes $A_{t}$. For this we define \[\{a_{j},\mbf{d}_{j},\mbf{f}_{j},A^{(j,0)}\}_{j=1}^{p}\] and \[\{B_{k},D_{k},F_{k},A^{(p,j)}\}_{k=1}^{q}\] recursively by
\begin{align}
    R_{j}A^{(j-1,0)}R_{j}&=\begin{pmatrix}a_{j}&\mbf{d}_{j}^{T}\\\mbf{f}_{j}&A^{(j,0)}\end{pmatrix},\label{eq:A^{(j,0)}}\\
    Q_{k}^{T}A^{(p,k)}Q_{k}&=\begin{pmatrix}B_{k}&D_{k}^{T}\\F_{k}&A^{(p,k+1)}\end{pmatrix},\label{eq:A^{(m_r,j)}}
\end{align}
where $R_{j}$ and $Q_{k}$ are the unitaries that make up the successive steps of the partial Schur decomposition. The resolvents of $A^{(j,0)}$ and $A^{(p,k)}$ will be denoted with the superscripts $(j,0)$ and $(p,k)$ respectively, e.g. $H^{(j,0)}_{z}(\eta)=\left(\eta^{2}+|A^{(j,0)}_{z}|^{2}\right)^{-1}$ and $H^{(0,0)}_{z}(\eta)\equiv H_{z}(\eta)\equiv H^{A}_{z}(\eta)$.

In terms of these definitions we have
\begin{align*}
    \rho(A_{t}(\omega_{0}))=\left(\frac{2\pi t}{N}\right)^{N^{2}/2}\exp\Bigg\{&-\frac{N}{2t}\sum_{j=1}^{p}\left[(u_{j}-a_{j})^{2}+\left\|\mbf{w}_{j}-\mbf{d}_{j}\right\|^{2}+\left\|\mbf{f}_{j}\right\|^{2}\right]\\
    &-\frac{N}{2t}\sum_{k=1}^{q}\left[\left\|Z_{k}-B_{k}\right\|^{2}_{2}+\left\|W_{k}-D_{k}\right\|^{2}_{2}+\left\|F_{k}\right\|^{2}_{2}\right]\\
    &-\frac{N}{2t}\left\|A^{(p,q)}_{t}-A^{(p,q)}\right\|^{2}_{2}\Bigg\}.
\end{align*}
After multiplying this expression by the Jacobian, {we must integrate over the domains 
\begin{itemize}
\item $\mbf{w}_{j}\in\mbb{R}^{N-j},\,j=1,...,p$,
\item $W_{k}\in\mbb{R}^{(N-p-2j)\times2},\,k=1,...,q$,
\item $\mbf{v}_{j}\in S^{N-j}_{+},\,j=1,...,p$,
\item $V_{k}\in O_{+}(N-p-2k+2,2),\,k=1,...,q$,
\item $A^{(p,q)}_{t}\in\mbb{M}_{N-m}(\mbb{R})$,
\end{itemize}}
where $m=p+2q$.

The variables $\mbf{w}_{j}$ and $W_{k}$ can be immediately integrated out since the dependence is purely Gaussian:
\begin{align*}
    \int_{\mbb{R}^{N-j}}e^{-\frac{N}{2t}\left\|\mbf{w}_{j}-\mbf{d}_{j}\right\|^{2}}\diff\mbf{w}_{j}&=\left(\frac{2\pi t}{N}\right)^{(N-j)/2},\\
    \int_{\mbb{R}^{(N-p-2k)\times2}}e^{-\frac{N}{2t}\left\|W_{k}-D_{k}\right\|^{2}_{2}}\diff W_{k}&=\left(\frac{2\pi t}{N}\right)^{N-p-2k}.
\end{align*}

For the integrals over $\mbf{v}_{j}$ and $V_{k}$ we find it advantageous to define new measures on $S^{n-1}$ and $O(n,2)$. Let $X\in\mbb{M}_{n}(\mbb{R})$, $u\in\mbb{R}$ and $z\in\mbb{C}$. We define the probability measure $\mu_{n}$ on $S^{n-1}$ by
\begin{align}
    \diff\mu_{n}(\mbf{v};u,X)&=\frac{{\psi_{u}}}{K_{n}(u,X)}{\left(\frac{N}{2\pi t}\right)^{n/2-1}}e^{-\frac{N}{2t}\left\|X_{u}\mbf{v}\right\|^{2}}\,\mathrm{d}_{H}\mbf{v},\label{eq:mu_n}
\end{align}
where $\psi_{u}:=\psi^{X}_{u}$ is defined in \eqref{eq:psi}, with normalisation
\begin{align}
    K_{n}(u,X)&={\psi_{u}\left(\frac{N}{2\pi t}\right)^{n/2-1}}\int_{S^{n-1}}e^{-\frac{N}{2t}\|X_{u}\mbf{v}\|^{2}}\,\mathrm{d}_{H}\mbf{v},\label{eq:K_n}
\end{align}
and the probability measure $\nu_{n}$ on $O(n,2)$ by
\begin{align}
    \diff\nu_{n}(V;\delta,z,X)&=\frac{{(\psi_{z})^{2}}}{L_{n}(\delta,z,X)}{\left(\frac{N}{2\pi t}\right)^{n-3}}\nonumber\\&\times\exp\left\{-\frac{N}{2t}\tr\left(V^{T}X^{T}XV-2Z^{T}V^{T}XV+Z^{T}Z\right)\right\}\,\mathrm{d}_{H}V,\label{eq:nu_n}
\end{align}
with normalisation
\begin{align}
    L_{n}(\delta,z,X)&=(\psi_{z})^{2}\left(\frac{N}{2\pi t}\right)^{n-3}\nonumber\\&\times\int_{O(n,2)}\exp\left\{-\frac{N}{2t}\tr\left(V^{T}X^{T}XV-2Z^{T}V^{T}XV+Z^{T}Z\right)\right\}\,\mathrm{d}_{H}V.\label{eq:L_n}
\end{align}
{The asymptotic properties of these measures are studied in Section \ref{sec:gaussian} and we will make reference to the results therein in the remainder of this section.}

In anticipation of the integral over $A^{(p,q)}_{t}$, for $X\in\mbb{M}_{n}(\mbb{R})$ we define
\begin{align}
    F_{n}(\mbf{u},\mbf{z};X)&=\frac{\wt{\Delta}(\mbf{u},\mbf{z})}{\prod_{j=1}^{m}\lambda_{j}}\mbb{E}_{Y}\left[\prod_{j=1}^{m}\det\left(X+\sqrt{\frac{Nt}{n}}Y-\lambda_{j}\right)\right],
\end{align}
where $m=p+2q$,
\begin{align}
    \mbf{\lambda}&=(u_{1},...,u_{p},z_{1},...,z_{p},\bar{z}_{1},...,\bar{z}_{p}),
\end{align}
and the expectation is with respect to $Y\sim GinOE(n)$. {The asymptotics of this expectation value are studied in Section \ref{sec:expectation}.}

{Now let
\begin{align}
    \wt{\mbf{u}}&:=\lambda+\frac{\mbf{u}}{\sqrt{N\sigma_{\lambda}}},\\
    \wt{\mbf{x}}+i\wt{\mbf{y}}=:\wt{\mbf{z}}&:=\lambda+\frac{\mbf{z}}{\sqrt{N\sigma_{\lambda}}}=:\lambda+\frac{\mbf{x}+i\mbf{y}}{\sqrt{N\sigma_{\lambda}}},\\
    \wt{\bs\delta}&:=\frac{\bs\delta}{\sqrt{N\sigma_{\lambda}}}.
\end{align}}
For $j=1,...,p$ and $k=1,...,q$ we define $\eta^{(j,0)}_{\lambda}$ and $\eta^{(p,k)}_{\lambda}$ by \eqref{eq:eta} with $X=A^{(j,0)}$ and $X=A^{(p,k)}$ respectively. We use the convention that resolvents of $A^{(j,0)}_{\lambda}$ and $A^{(p,k)}_{\lambda}$ without arguments are evaluated at $\eta^{(j,0)}_{\lambda}$ and $\eta^{(p,k)}_{\lambda}$ respectively, {e.g. $H^{(p,k)}_{\lambda}:=H^{(p,k)}_{\lambda}(\eta^{(p,j)}_{\lambda})$}.

Given $\eta^{(j,0)}_{\lambda}$ and $\eta^{(p,k)}_{\lambda}$ we can then define $\psi^{(j,0)}_{\lambda}:=\psi^{A^{(j,0)}}_{\lambda}$ and $\psi^{(p,k)}_{\lambda}:=\psi^{A^{(p,k)}}_{\lambda}$ according to \eqref{eq:psi}, and in turn
\begin{align}
    \diff \mu_{j}(\mbf{v})&=\diff\mu_{N-j+1}\left(\mbf{v};\wt{u}_{j},A^{(j-1,0)}\right),\\
    \diff\nu_{k}(V;\delta_{k})&=\diff\nu_{N-p-2k+2}\left(V;\wt{\delta}_{k},\wt{z}_{k},A^{(p,k-1)}\right),
\end{align}
with normalisation
\begin{align}
    K_{j}&=K_{N-j+1}\left(\wt{u}_{j},A^{(j-1,0)}\right),\\
    L_{k}(\delta_{k})&=L_{N-p-2k+2}\left(\wt{\delta}_{k},\wt{z}_{k},A^{(p,k-1)}\right).
\end{align}
{Note that $\mu_{j}$ depends on $\{\mbf{v}_{l}\}_{l=1}^{j-1}$ and $\nu_{k}$ depends on $\{\mbf{v}_{l}\}_{l=1}^{p}$, $\{V_{l}\}_{l=1}^{k-1}$ and $\{\delta_{l}\}_{l=1}^{k}$, so we choose the order of integration to be \[\diff\nu_{q}\diff\delta_{q}\cdots \diff\nu_{1}\diff\delta_{1}\diff\mu_{p}\cdots \diff\mu_{1}.\] We find it more convenient to absorb the factor $y^{2}/\sqrt{\delta^{2}+4y^{2}}$ in the measure and so define
\begin{align}
    \diff\alpha_{y}(\delta)&=\frac{2\sqrt{N\sigma_{\lambda}}y\delta}{\sqrt{4N\sigma_{\lambda}y^{2}+\delta^{2}}}\diff\delta.
\end{align}}

After these definitions we have the expression
\begin{align}
    &\xi^{(p,q)}_{\lambda}(\mbf{u},\mbf{z})=v_{p,q}(\lambda){K}_{1}\mbb{E}_{{\mu}_{1}}\Big[\cdots {K}_{p}\mbb{E}_{{\mu}_{p}}\Big[\int_{0}^{\infty}{L}_{1}(\delta_{1})\mbb{E}_{{\nu}_{1}}\Big[\cdots\nonumber\\
    &\int_{0}^{\infty}{L}_{q}(\delta_{q})\mbb{E}_{{\nu}_{q}}\Big[{F}_{N-m}(\wt{\mbf{u}},\wt{\mbf{z}};A^{(p,q)}){\Psi}^{(p,q)}(\mbf{u},\mbf{z})\Big]\diff\alpha_{\wt{y}_{q}}(\delta_{q})\nonumber\\
    &\cdots\Big]\diff\alpha_{\wt{y}_{1}}(\delta_{1})\Big]\cdots\Big]\cdots\Big],
\end{align}
where
\begin{align}
    v_{p,q}(\lambda)&\equiv v_{p,q}(\lambda,N):=\frac{N^{p/2+q}}{2^{2p+3q}(\pi t)^{p+3q}(\sigma_{\lambda})^{p/2+2q}},
\end{align}
and
\begin{align}
    \Psi^{(p,q)}(\mbf{u},\mbf{z})&=\prod_{j=1}^{p}\frac{\psi^{(p,q)}_{\wt{u}_{j}}}{\psi^{(j-1,0)}_{\wt{u}_{j}}}\cdot\prod_{k=1}^{q}\left(\frac{\psi^{(p,q)}_{\wt{z}_{j}}}{\psi^{(p,k-1)}_{\wt{z}_{k}}}\right)^{2}.\label{eq:Psi}
\end{align}
Note that we have divided by $2^{-p-q}$ and lifted the integrals over $S^{n-1}_{+}$ and $O_{+}(n,2)$ to $S^{n-1}$ and $O(n,2)$ respectively. 

To prove Lemmas \ref{lem:xi_r} and \ref{lem:xi_c}, we first need the following estimates on $\Psi^{(p,q)}$.
\begin{lemma}\label{lem:Psi}
There is a constant $C$ such that
\begin{align}
    \Psi^{(p,q)}(\mbf{u},\mbf{z})&\leq N^{C},\label{eq:PsiBound}
\end{align}
uniformly in $\mbf{u},\,\mbf{z},\,\{\mbf{v}_{j}\}_{j=1}^{p}$ and $\{V_{k}\}_{k=1}^{q}$, and
\begin{align}
    \Psi^{(p,q)}\left(\mbf{u},\mbf{z}\right)&=\left[1+O\left(\frac{\log N}{\sqrt{Nt^{2}}}\right)\right]\prod_{j=1}^{p}\left|\det^{j/2}\left[(1_{2}\otimes\mbf{v}_{j}^{T})G^{(j-1,0)}_{\wt{u}_{j}}(1_{2}\otimes\mbf{v}_{j})\right]\right|\nonumber\\
    &\times\prod_{k=1}^{q}\left|\det^{p/2+k}\left[(1_{2}\otimes V_{k}^{T})G^{(p,k-1)}_{\wt{z}_{k}}(1_{2}\otimes V_{k})\right]\right|.\label{eq:PsiAsymp}
\end{align}
uniformly in $\|\mbf{u}\|_{\infty}<\log N,\,\|\mbf{z}\|_{\infty}<\log N$ and $\{\mbf{v}_{j}\}_{j=1}^{p},\,\{V_{k}\}_{k=1}^{q}$.
\end{lemma}
\begin{proof}
Let $w\in\mbb{C}$; we write each ratio in $\Psi^{(p,q)}(\mbf{u},\mbf{z})$ as a telescoping product:
\begin{align*}
    \frac{\psi^{(p,q)}_{w}}{\psi^{(j-1,0)}_{w}}&=\prod_{l=j}^{p}\frac{\psi^{(l,0)}_{w}}{\psi^{(l-1,0)}_{w}}\prod_{k=1}^{q}\frac{\psi^{(p,k)}_{w}}{\psi^{(p,k-1)}_{w}},\\
    \frac{\psi^{(p,q)}_{w}}{\psi^{(p,k-1)}_{w}}&=\prod_{l=k}^{q}\frac{\psi^{(p,l)}_{w}}{\psi^{(p,l-1)}_{w}}.
\end{align*}
Consider one factor in the first product; using Cramer's rule we find
\begin{align}
    \frac{\psi^{(j,0)}_{w}}{\psi^{(j-1,0)}_{w}}&=\det^{1/2}\left[(1_{2}\otimes\mbf{v}_{j}^{T})G^{(j-1,0)}_{w}(1_{2}\otimes\mbf{v}_{j})\right]\nonumber\\&\times\exp\left\{-N\left[\phi^{(j-1,0)}_{w}(\eta^{(j,0)}_{w})-\phi^{(j-1,0)}_{w}\right]\right\},\label{eq:ratio}
\end{align}
where $\phi^{(j,0)}_{w}(\eta):=\phi^{A^{(j,0)}}_{w}(\eta)$ and $\phi^{(j,0)}_{w}:=\phi^{(j,0)}_{w}(\eta^{(j,0)}_{w})$. If $|w|\geq1$, then we use the bounds in \eqref{eq:etaBound} and \eqref{eq:phiBound1} to obtain
\begin{align}
    \frac{\psi^{(j,0)}_{w}}{\psi^{(j-1,0)}_{w}}&\leq\frac{C}{\eta^{(j,0)}_{w}}\leq C\sqrt{\frac{N}{t}}.\label{eq:ratioBound1}
\end{align}

When $|w|<1$, we have $\eta^{(j-1,0)}_{w},\eta^{(j,0)}_{w}=O(t)$ by \eqref{eq:eta_zB1}. Using Lemma \ref{lem:minorresolvent} to replace $A^{(j,0)}$ with $A^{(j-1,0)}$ and then using the resolvent identity we find
\begin{align}
    \eta^{(j,0)}_{w}-\eta^{(j-1,0)}_{w}&=\frac{t}{2}\Im\left(\Tr{G^{(j,0)}_{w}}-\Tr{G^{(j-1,0)}_{w}}\right)\nonumber\\
    &=\frac{t}{2}\Im\Tr{G^{(j-1,0)}_{w}(\eta^{(j,0)}_{w})-G^{(j-1,0)}_{w}}+O\left(\frac{1}{N}\right)\nonumber\\
    &=\frac{t(\eta^{(j,0)}_{w}-\eta^{(j-1,0)}_{w})}{2}\Re\Tr{G^{(j-1,0)}_{w}(\eta^{(j,0)}_{w})G^{(j-1,0)}_{w}}+O\left(\frac{1}{N}\right)\nonumber\\
    &=O\left(\frac{1}{N}\right).\label{eq:etaPert}
\end{align}
The last line follows by \eqref{eq:C3.1}. We can then use \eqref{eq:phiBound2} to obtain
\begin{align*}
    \left|\phi^{(j-1,0)}_{w}(\eta^{(j,0)}_{w})-\phi^{(j-1,0)}_{w}\right|&\leq\frac{C}{N^{2}t}.
\end{align*}

We use the resolvent identity and Lemma \ref{lem:minorresolvent} to replace $w$ with $\wt{u}_{j}$ and $\eta^{(j,0)}_{w}$ with $\eta^{(j-1,0)}_{u_{j}}$ in the first factor in \eqref{eq:ratio}, assuming $|w-\wt{u}_{j}|<\frac{\log N}{\sqrt{N}}$:
\begin{align*}
    &\frac{\det^{1/2}\left[(1_{2}\otimes\mbf{v}_{j}^{T})G^{(j-1,0)}_{w}(\eta^{(j,0)}_{w})(1_{2}\otimes\mbf{v}_{j})\right]}{\det^{1/2}\left[(1_{2}\otimes\mbf{v}_{j}^{T})G^{(j-1,0)}_{\wt{u}_{j}}(1_{2}\otimes\mbf{v}_{j})\right]}\\
    &=\det^{1/2}\left[1_{2}-\bigl((1_{2}\otimes\mbf{v}_{j}^{T})G^{(j-1,0)}_{w}(1_{2}\otimes\mbf{v}_{j})\bigr)^{-1}(1_{2}\otimes\mbf{v}_{j}^{T})G^{(j-1,0)}_{w}(\eta^{(j,0)}_{w})BG^{(j-1,0)}_{\wt{u}_{j}}(1_{2}\otimes\mbf{v}_{j})\right],
\end{align*}
where
\begin{align*}
    B&=\begin{pmatrix}i(\eta^{(j,0)}_{w}-\eta^{(j-1,0)}_{\wt{u}_{j}})&w-u_{j}\\\bar{w}-\wt{u}_{j}&i(\eta^{(j,0)}_{w}-\eta^{(j-1,0)}_{\wt{u}_{j}})\end{pmatrix}.
\end{align*}
Since $\|B\|<\frac{\log N}{\sqrt{N}}$, by \ref{lem:minorresolvent} we obtain
\begin{align*}
    \frac{\det^{1/2}\left[(1_{2}\otimes\mbf{v}_{j}^{T})G^{(j-1,0)}_{w}(\eta^{(j,0)}_{w})(1_{2}\otimes\mbf{v}_{j})\right]}{\det^{1/2}\left[(1_{2}\otimes\mbf{v}_{j}^{T})G^{(j-1,0)}_{\wt{u}_{j}}(1_{2}\otimes\mbf{v}_{j})\right]}&=1+O\left(\frac{\log N}{\sqrt{Nt^{2}}}\right).
\end{align*}

Altogether we find that when $|w-\wt{u}_{j}|<\frac{\log N}{\sqrt{N}}$,
\begin{align*}
    \frac{\psi^{(j,0)}_{w}}{\psi^{(j-1,0)}_{w}}&=\left[1+O\left(\frac{\log N}{\sqrt{Nt^{2}}}\right)\right]\left|\det^{1/2}\left[(1_{2}\otimes\mbf{v}_{j}^{T})G^{(j-1,0)}_{\wt{u}_{j}}(1_{2}\otimes\mbf{v}_{j})\right]\right|.
\end{align*}

The same arguments also give
\begin{align}
    \frac{\psi^{(p,k)}_{w}}{\psi^{(p,k-1)}_{w}}&\leq\frac{CN}{t},\label{eq:ratioBound2}
\end{align}
when $|w|\geq1$, and
\begin{align*}
    \frac{\psi^{(p,k)}_{w}}{\psi^{(p,k-1)}_{w}}&=\left[1+O\left(\frac{\log N}{\sqrt{Nt^{2}}}\right)\right]\left|\det^{1/2}\left[(1_{2}\otimes V_{k}^{T})G^{(p,k-1)}_{\wt{z}_{k}}(1_{2}\otimes V_{k})\right]\right|,
\end{align*}
when $|w-\wt{z}_{k}|<\frac{\log N}{\sqrt{N}}$.
\end{proof}

We can now prove Lemmas \ref{lem:xi_r} and \ref{lem:xi_c} using the results from Sections \ref{sec:expectation} and \ref{sec:gaussian}.
\begin{proof}[Proof of Lemma \ref{lem:xi_r}]
Observe that $A^{(j,0)}\in\mc{X}_{N-j}(\gamma,\omega)$ and $A^{(p,k)}\in\mc{X}_{N-p-2k}(\gamma,\omega)$ uniformly in $\{\mbf{v}_{j}\}_{j=1}^{p}$ and $\{V_{k}\}_{k=1}^{q}$ by Lemma \ref{lem:projection}, so that the asymptotics in Lemmas \ref{lem:Fnear}, \ref{lem:KAsymp} and \ref{lem:Lnear} hold uniformly in $\{\mbf{v}_{j}\}$ and $\{V_{j}\}$. The proof amounts to applying these asymptotics and replacing $A^{(j,k)}$ with $A$ using the estimates
\begin{align}
    \frac{\Tr{(H^{(j,k)}_{u})^{2}}}{\Tr{H_{u}^{2}}}&=1+O\left(\frac{1}{Nt^{2}}\right),\label{eq:HEstimate1}\\
    \frac{\sigma^{(j,k)}_{u}}{\sigma_{u}}&=1+O\left(\frac{1}{Nt^{2}}\right),\label{eq:sigmaEstimate1}
\end{align}
which hold uniformly in $\{\mbf{v}_{j}\}_{j=1}^{p}$ and $\{V_{k}\}_{k=1}^{q}$ by Lemma \ref{lem:minorresolvent}, and
\begin{align}
    \frac{\Tr{H_{\lambda_{j}}^{2}}}{\Tr{H_{u}^{2}}}&=1+O\left(\frac{\log N}{\sqrt{Nt^{3}}}\right),\label{eq:HEstimate2}\\
    \frac{\sigma_{\lambda_{j}}}{\sigma_{u}}&=1+O\left(\frac{\log N}{\sqrt{Nt^{3}}}\right),\label{eq:sigmaEstimate2}
\end{align}
uniformly in $|\lambda_{j}-u|<N^{-1/2}\log N$, which follow from the resolvent identity, Cauchy-Schwarz and \eqref{eq:C3.1}.

If $p$ is odd, $|u_{j}|<C,\,j=1,...,p$ and $|z_{j}|<C,\,j=1,...,q$, then by Lemma \ref{lem:Fodd} we have
\begin{align*}
    F_{N-m}\left(\wt{\mbf{u}},\wt{\mbf{z}},A^{(p,q)}\right)&\leq e^{-c\log^{2}N},
\end{align*}
uniformly in $\{\mbf{v}_{j}\}_{j=1}^{p}$ and $\{V_{k}\}_{k=1}^{q}$. Using the bound on $\Psi$ in \eqref{eq:PsiBound} and the asymptotics of $L_{n}$ from Lemma \ref{lem:Lnear}, the integral over $\nu_{q}$ and $\delta_{q}$ is bounded by
\begin{align}
    N^{C}e^{-c\log^{2}N}\int_{0}^{\infty}L_{q}(\delta_{q})\diff\alpha(\delta_{q})&\leq e^{-c\log^{2}N}.
\end{align}
Repeating these bounds for the remaining integrals we obtain \eqref{eq:xi1}.

Now let $p$ be even, $|u_{j}|<C,\,j=1,...,p-1$ and $|z_{j}|<C,\,j=1,...,p$. If $|u_{p}|>\log N$, then by Lemma \ref{lem:Ffar} we have
\begin{align*}
    F_{N-m}\left(\wt{\mbf{u}},\wt{\mbf{z}},A^{(p,q)}\right)&\leq e^{-c\log^{2}N},
\end{align*}
uniformly $\{\mbf{v}_{j}\}_{j=1}^{p}$ and $\{V_{k}\}_{k=1}^{q}$. We integrate over $\delta_{k}$ as before. For $K_{p}$ we use the bounds in Lemma \ref{lem:KAsymp}:
\begin{align*}
    K_{p}&\leq\begin{cases}
    e^{-\frac{cN}{t}u_{p}^{2}}&\quad |u_{p}|>4\sqrt{N\sigma_{u}}\|X\|\\
    \frac{C\|X_{u_{p}}\|^{2}}{\sqrt{N}}&\quad \textup{otherwise}
    \end{cases}.
\end{align*}
This implies \eqref{eq:xi2} and \eqref{eq:xi3}.

If $|u_{p}|<\log N$, then by Lemma \ref{lem:Fnear} and the estimates in \eqref{eq:HEstimate1}--\eqref{eq:sigmaEstimate2} we have
\begin{align*}
    F_{N-m}\left(\wt{\mbf{u}},\wt{\mbf{z}},A^{(p,q)}\right)&=\left[1+O\left(\frac{\log N}{\sqrt{Nt^{3}}}\right)\right]\frac{(N\sigma_{u})^{q/2}}{(t^{2}\Tr{H_{u}^{2}}\sigma_{u})^{m(m-1)/4}}e^{-2\sum_{j=1}^{q}(\Im z_{j})^{2}}I_{m}(\mbf{u},\mbf{z}),
\end{align*}
uniformly in $\{\mbf{v}_{j}\}_{j=1}^{p}$ and $\{V_{k}\}_{k=1}^{q}$. Taking this outside the remaining integrals and using the asymptotics of $\Psi^{(p,q)}$ in \eqref{eq:PsiAsymp} we find
\begin{align*}
    \xi^{(p,q)}_{u}(\mbf{u},\mbf{z})&=\left[1+O\left(\frac{\log N}{\sqrt{Nt^{3}}}\right)\right]\frac{(N\Tr{H_{u}^{2}})^{(p+3q)/2}}{2^{p}(2\pi)^{p+3q}}e^{-2\sum_{j=1}^{q}(\Im z_{j})^{2}}I_{m}(\mbf{u},\mbf{z})\\
    &\times K_{1}\mbb{E}_{\mu_{1}}\Big[D_{1,1}^{1/2}\cdots K_{p}\mbb{E}_{\mu_{p}}\Big[D_{p,1}^{p/2}\\
    &\times\int_{0}^{\infty}L_{1}(\delta_{1})\mbb{E}_{\nu_{1}}\Big[D_{1,2}^{p/2+1}\cdots\int_{0}^{\infty}L_{q}(\delta_{q})\mbb{E}_{\nu_{q}}\Big[D_{q,2}^{p/2+q}\Big]\diff\alpha_{\wt{y}_{q}}(\delta_{q})\cdots\Big]\diff\alpha_{\wt{y}_{1}}(\delta_{1})\cdots\Big]\cdots\Big],
\end{align*}
where
\begin{align*}
    D_{j,1}&=\frac{1}{t^{2}\Tr{H_{u}^{2}}\sigma_{u}}\left|\det\left[(1_{2}\otimes\mbf{v}_{j}^{T})G^{(j-1,0)}_{\wt{u}_{j}}(1_{2}\otimes\mbf{v}_{j})\right]\right|,\\
    D_{k,2}&=\frac{1}{\left(t^{2}\Tr{H_{u}^{2}}\sigma_{u}\right)^{2}}\left|\det\left[(1_{2}\otimes V_{k}^{T})G^{(p,k-1)}_{\wt{z}_{k}}(1_{2}\otimes V_{k})\right]\right|.
\end{align*}
Since $D_{k,2}\leq C/t^{2}$, by \eqref{eq:Lnear1} and \eqref{eq:Lnear2} we have
\begin{align*}
    \int_{\log N}^{\infty}L_{q}(\delta_{q})\mbb{E}_{\nu_{q}}\Big[D_{q,1}^{p/2+q}\Big]\diff\alpha_{\wt{y}_{q}}(\delta_{q})&\leq\frac{C}{t^{p+2q}}e^{-c\log^{2}N}\leq e^{-c\log^{2}N}.
\end{align*}
When $\delta_{q}<\log N$, then by Lemma \ref{lem:det2Near} and \eqref{eq:HEstimate1}--\eqref{eq:sigmaEstimate2} we have
\begin{align*}
    \mbb{E}\left[D^{p/2+q}_{q,2}\right]&=1+O\left(\frac{\log N}{\sqrt{Nt^{3}}}\right),
\end{align*}
and by Lemma \ref{lem:Lnear} and \eqref{eq:HEstimate1}--\eqref{eq:sigmaEstimate2} we have
\begin{align*}
    L_{q}(\delta_{q})&=\left[1+O\left(\frac{\log N}{\sqrt{Nt^{3}}}\right)\right]\frac{2^{5/2}\pi^{3/2}}{N^{3/2}\Tr{H_{u}^{2}}^{3/2}}\cdot\frac{2y_{q}\delta_{q}}{\sqrt{\delta_{q}^{2}+4y_{q}^{2}}}e^{-\frac{1}{2}\delta_{q}^{2}},
\end{align*}
where $y_{q}=\Im z_{q}$. Thus we find
\begin{align*}
    \int_{0}^{\infty}L_{q}(\delta_{q})\mbb{E}_{\nu_{q}}\Big[D_{q,1}^{p/2+q}\Big]\diff\alpha_{\wt{y}_{q}}(\delta_{q})&=\left[1+O\left(\frac{\log N}{\sqrt{Nt^{3}}}\right)\right]\frac{2^{3}\pi^{2}}{N^{3/2}\Tr{H_{u}^{2}}^{3/2}}\cdot y_{q}e^{2y_{q}^{2}}\textup{erfc}\left(\sqrt{2}y_{q}\right),
\end{align*}
uniformly in $\{\mbf{v}_{j}\}_{j=1}^{p}$ and $\{V_{k}\}_{k=1}^{q}$. We can then repeat these calculations for the remaining integrals over $V_{k},\,k=1,...,q-1$. For the integrals over $\mbf{v}_{j}$, we note that by Lemma \ref{lem:det1} and \eqref{eq:HEstimate1}--\eqref{eq:sigmaEstimate2} we have
\begin{align*}
    \mbb{E}_{\mu_{j}}\left[D_{j,1}^{j/2}\right]&=1+O\left(\frac{\log N}{\sqrt{Nt^{3}}}\right),
\end{align*}
and by Lemma \ref{lem:KAsymp} and \eqref{eq:HEstimate1}--\eqref{eq:sigmaEstimate2} we have
\begin{align*}
    K_{j}&=\left[1+O\left(\frac{\log N}{\sqrt{Nt^{3}}}\right)\right]\sqrt{\frac{4\pi}{N\Tr{H_{u}^{2}}}}.
\end{align*}
Combining all the constants together we obtain \eqref{eq:xi4}.
\end{proof}

\begin{proof}[Proof of Lemma \ref{lem:xi_c}]
By \ref{lem:Ffar}, we have
\begin{align*}
    F_{N-q}(\wt{\mbf{z}};A^{(q)})&=\left[1+O\left(\frac{\log^{3}N}{\sqrt{Nt^{3}}}\right)\right]\left(\frac{2\pi^{3}N}{t^{2}\Tr{H^{2}_{z}}}\right)^{q/2}\left(\frac{4y^{2}}{t^{2}\Tr{H^{2}_{z}}\sigma_{z}\tau_{z}}\right)^{q(q-1)/2}\\
    &\times\det\left[\frac{1}{\pi}e^{-\frac{1}{2}(|z_{j}|^{2}+|z_{k}|^{2})+\bar{z}_{j}z_{k}}\right]_{j,k=1}^{q},
\end{align*}
where we have used \eqref{eq:HEstimate1}--\eqref{eq:sigmaEstimate2} with $z$ in place of $u$ to remove the superscript $(q)$ from the factor in front of the determinant. We are left with the integrals over $\nu_{k}$ and $\delta_{k}$:
\begin{align*}
    \int_{0}^{\infty}L_{1}(\delta_{1})\mbb{E}_{\nu_{1}}\left[D_{1,2}\cdots\int_{0}^{\infty}L_{q}(\delta_{q})\mbb{E}_{\nu_{q}}\left[D_{q,2}^{q}\right]\diff\alpha_{\wt{y}_{q}}(\delta_{q})\cdots\right]\diff\alpha_{\wt{y}_{1}}(\delta_{1}),
\end{align*}
where
\begin{align*}
    D_{k,2}&=\frac{4y^{2}}{t^{2}\Tr{H_{z}^{2}}\sigma_{z}\tau_{z}}\left|\det\left[(1_{2}\otimes V_{k}^{T})G^{(k-1)}_{\wt{z}_{k}}(1_{2}\otimes V_{k})\right]\right|.
\end{align*}
By Lemma \ref{lem:Lfar} we can truncate the integral over $\delta_{q}$ to the region $\delta_{q}<\log N$, in which we have the asymptotics
\begin{align*}
    L_{q}(\delta_{q})&=\left[1+O\left(\frac{\log^{2}N}{\sqrt{Nt}}\right)\right]\frac{2^{5/2}\pi^{3/2}}{N^{3/2}\Tr{H^{2}_{z}}^{1/2}\Tr{H_{z}H_{\bar{z}}}}\\
    &\times\left[\exp\left\{-\frac{\tau_{z}}{8y^{2}t^{2}\Tr{H_{z}H_{\bar{z}}}\sigma_{z}}\delta_{q}^{2}\right\}+O(N^{-D})\right].
\end{align*}
By Lemma \ref{lem:det2Far}, when $\delta_{q}<\log N$ we have
\begin{align*}
    \mbb{E}_{\nu_{q}}\left[D_{k,2}^{k}\right]&=1+O\left(\frac{\log N}{\sqrt{Nt^{3}}}\right).
\end{align*}
Thus we have
\begin{align*}
    \int_{0}^{\infty}L_{q}(\delta_{q})\mbb{E}_{\nu_{q}}\left[D_{q,2}^{q}\right]\diff\alpha_{\wt{y}_{q}}(\delta_{q})&=\left[1+O\left(\frac{\log N}{\sqrt{Nt^{3}}}\right)\right]\frac{2^{5/2}\pi^{3/2}}{N^{3/2}\Tr{H^{2}_{z}}^{1/2}}\cdot\frac{4y^{2}t^{2}\sigma_{z}}{\tau_{z}}.
\end{align*}
Repeating these calculations for the remaining integrals over $\{V_{k},\delta_{k}\}_{k=1}^{q-1}$ and combining all the constants we arrive at \eqref{eq:xi_c}.
\end{proof}

\section{Expectation value of products of characteristic polynomials}\label{sec:expectation}
Fix $\gamma\in(0,1/3),\,\epsilon\in(0,\gamma/2),\,\omega\in(0,1),\,N^{-\gamma+\epsilon}\leq t\leq N^{-\epsilon}$ and $X\in\mc{X}_{n}(\gamma,\omega)$. Let $p,q\in\mbb{N}$ and $m=p+2q$. For $\mbf{u}\in\mbb{R}^{p}$ and $\mbf{z}\in\mbb{C}^{q}$ we recall the definition
\begin{align*}
    F_{n}(\mbf{u},\mbf{z},X)&=\frac{\wt{\Delta}(\mbf{u},\mbf{z})}{\prod_{j=1}^{m}\psi_{\lambda_{j}}}\mbb{E}_{Y}\left[\prod_{j=1}^{m}\det\left(X+\sqrt{\frac{Nt}{n}}Y-\lambda_{j}\right)\right],
\end{align*}
where $\psi_{\lambda}:=\psi^{X}_{\lambda}$ is defined in \eqref{eq:psi}, $Y$ is a GinOE(n) matrix and \[\bs\lambda=(u_{1},...,u_{p},z_{1},...,z_{q},\bar{z}_{1},...,\bar{z}_{q}).\] By the diffusion method (e.g. \cite{grela_diffusion_2016,liu_phase_2022}) or anti-commuting variables (e.g. \cite{tribe_ginibre_2014-1}) one can derive the following duality formula, which is eq. (117) in \cite{tribe_ginibre_2014-1} and eq. (2.8) in \cite{liu_phase_2022} (up to the factor $\wt{\Delta}$):
\begin{align}
    F_{n}(\mbf{u},\mbf{z},X)&=\left(\frac{N}{\pi t}\right)^{m(m-1)/2}\wt{\Delta}(\mbf{u},\mbf{z})\int_{\mbb{M}^{skew}_{m}(\mbb{C})}e^{-\frac{N}{2t}\tr S^{*}S}\pf \mbf{M}_{\bs\lambda}(S)\diff S,
\end{align}
where
\begin{align}
    \mbf{M}_{\bs\lambda}(S)&=\begin{pmatrix}S\otimes1_{n}&1_{m}\otimes X-\bs\lambda\otimes1_{n}\\-1_{m}\otimes X^{T}+\bs\lambda\otimes1_{n}&S^{*}\otimes1_{n}\end{pmatrix}.
\end{align}
{Here $\pf$ denotes the pfaffian,
\begin{align*}
    \pf M&=\frac{1}{2^{n}n!}\sum_{\sigma}\text{sgn}(\sigma) \prod_{j=1}^{2n}M_{\sigma_{2j-1},\sigma_{2j}},
\end{align*}
where the sum is over all permutations of $[1,...,2n]$.} and the integral is over $m\times m$ complex skew-symmetric matrices $S$.

We are interested in the asymptotics of $F_{n}$ in four regimes:
\begin{itemize}
\item $p>0$, $m$ is odd and $\|\bs\lambda-u\|_{\infty}<CN^{-1/2}$;
\item $p>0$, $m$ is even and $|u_{j}-u|>N^{-1/2}\log N$ for exactly one $j\in[1,...,p]$;
\item $p>0$, $m$ is even and $\|\bs\lambda-u\|_{\infty}<CN^{-1/2}$;
\item $p=0$ and $\|\mbf{z}-z\|_{\infty}<CN^{-1/2}$ for some fixed $z\in\mbb{D}_{+}$.
\end{itemize}

In the first regime, when there are an odd number of arguments which are all close to some given point $u\in\mbb{R}$, we expect that $F_{n}$ is small. {To see why this might be the case, we note that normalising the expectation by $\psi_{\lambda}$ is effectively the same as normalising by
\begin{align*}
    \mbb{E}_{Y}\left|\det\left(X+\sqrt{\frac{Nt}{n}}Y-\lambda\right)\right|.
\end{align*}
Now consider 
\begin{align*}
    \frac{\mbb{E}_{Y}\det\left(X+\sqrt{t}Y-u\right)}{\mbb{E}_{Y}|\det\left(X+\sqrt{t}Y-u\right)|}&=\frac{\mbb{E}\left[\prod_{n=1}^{N_{\mbb{R}}}(u_{n}-u)\prod_{n=1}^{N_{\mbb{C}}}|z_{n}-u|^{2}\right]}{\mbb{E}\left[\prod_{n=1}^{N_{\mbb{R}}}|u_{n}-u|\prod_{n=1}^{N_{\mbb{C}}}|z_{n}-u|^{2}\right]}.
\end{align*}
Since the real spectrum is approximately symmetric about the origin, the fluctuations in the numerator act to suppress the expectation value.} These rough heuristics are confirmed in the following.
\begin{lemma}\label{lem:Fodd}
Let $p$ be odd, $|u_{j}-u|<CN^{-1/2},\,j=1,...,p$ and $|z_{j}-u|<CN^{-1/2},\,j=1,...,q$ for some $u\in(-1+\omega,1-\omega)$. Then
\begin{align}
    |F_{n}(\mbf{u},\mbf{z},A)|&\leq e^{-c\log^{2}N}.\label{eq:Fodd}
\end{align}
\end{lemma}
\begin{proof}
Let $P=|S|^{2}$ and $Q=|S^{*}|^{2}$. Define the following subsets of the integration domain.
\begin{align*}
    \Omega_{1}&=\bigcap_{j=1}^{m}\left\{|P_{jj}-(\eta_{\lambda_{j}})^{2}|<\sqrt{\frac{t^{3}}{N}}\log N\right\},\\
    \Omega_{2}&=\bigcap_{1\leq j\neq k\leq m}\left\{|P_{jk}|<\sqrt{\frac{t^{3}}{N}}\log N\right\}.
\end{align*}

Taking the absolute value of the Pfaffian, we have
\begin{align*}
    |\pf \mbf{M}_{\bs\lambda}|&=|\det^{1/2}\mbf{M}_{\bs\lambda}|\\
    &=\det^{1/4}\begin{pmatrix}Q\otimes1_{n}+|1_{m}\otimes X^{T}-\bs\lambda^{*}\otimes1_{n}|^{2}&(S\bs\lambda^{*}-\bs\lambda S)\otimes1_{n}\\(\bs\lambda S^{*}-S^{*}\bs\lambda^{*})\otimes1_{n}&P\otimes1_{n}+|1_{m}\otimes X-\bs\lambda^{*}\otimes1_{n}|^{2}\end{pmatrix}.
\end{align*}
Extracting the diagonal blocks by Fischer's inequality we have
\begin{align*}
    |\pf \mbf{M}_{\bs\lambda}|&\leq\prod_{j=1}^{m}\det^{1/4}(\mbf{M}_{\bs\lambda}(\mbf{M}_{\bs\lambda})^{*})_{jj}\\
    &=\prod_{j=1}^{m}\det^{1/4}\left(P_{jj}+|X_{\lambda_{j}}|^{2}\right)\det^{1/4}\left(Q_{jj}+|X_{\lambda_{j}}|^{2}\right).
\end{align*}
The integrand is thus bounded by
\begin{align*}
    \frac{e^{-\frac{N}{4t}\sum_{j}(P_{jj}+Q_{jj})}|\pf\mbf{M}_{\bs\lambda}|}{\prod_{j=1}^{m}\psi_{\lambda_{j}}}&\leq \exp\left\{-\frac{N}{4}\sum_{j=1}^{n}\left[\phi_{\lambda_{j}}(P_{jj})+\phi_{\lambda_{j}}(Q_{jj})-2\phi_{\lambda_{j}}(\eta_{\lambda_{j}})\right]\right\}.
\end{align*}
Since $|\lambda_{j}-u|<\frac{C}{\sqrt{N}}$ for $j=1,...,m$ and $u\in(-1+\omega,1-\omega)$, from the bound in \eqref{eq:phiBound2} we have
\begin{align*}
    \phi_{\lambda_{j}}(\eta)-\phi_{\lambda_{j}}(\eta_{\lambda_{j}})&\geq\frac{C(\eta-\eta_{\lambda_{j}})^{2}}{t},\quad j=1,...,m,
\end{align*}
which implies that the integral over the complement of $\Omega_{1}$ is $O(e^{-\log^{2}N})$.

Fix $S\in\Omega_{1}$. We use Fischer's inequality again {but this time we isolate a $2\times2$ principal block minor:
\begin{align*}
    \det M&\leq\det\begin{pmatrix}M_{jj}&M_{jk}\\M_{kj}&M_{kk}\end{pmatrix}\prod_{l\neq j,k}\det M_{ll}\\
    &=\det\left(1-M_{jj}^{-1}M_{jk}M_{kk}^{-1}M_{kj}\right)\prod_{l}\det M_{ll}.
\end{align*}}
With $M_{jj}=(\mbf{M}_{\bs\lambda}(\mbf{M}_{\bs\lambda})^{*})_{jj},\,M_{jk}=(\mbf{M}_{\bs\lambda}(\mbf{M}_{\bs\lambda})^{*})_{jk}$ and $M_{kk}=(\mbf{M}_{\bs\lambda}(\mbf{M}_{\bs\lambda})^{*})_{kk}$ for a distinct pair $j,k\in[m+1,...,2m]$ we have
\begin{align*}
    |\pf\mbf{M}_{\bs\lambda}|&\leq\left(\prod_{j=1}^{m}\det^{1/4}\left(P_{jj}+|X_{\lambda_{j}}|^{2}\right)\det^{1/4}\left(Q_{jj}+|X_{\lambda_{j}}|^{2}\right)\right)\\
    &\times\det^{1/4}\left(1-|P_{jk}|^{2}H_{\lambda_{j}}(\sqrt{P_{jj}})H_{\lambda_{k}}(\sqrt{P_{kk}})\right).
\end{align*}
Since $|\lambda_{j}-\lambda_{k}|<CN^{-1/2}$ for $j,k=1,...,m$ and $P_{jj},P_{kk}=O(t)$ by our assumption $S\in\Omega_{1}$, we have $\Tr{H_{\lambda_{j}}(\sqrt{P_{jj}})H_{\lambda_{k}}(\sqrt{P}_{kk})}\geq C/t^{3}$ by \eqref{eq:traceEstimate1} and so
\begin{align*}
    \det^{1/4}\left(1-|P_{jk}|^{2}H_{\lambda_{j}}(\sqrt{P_{jj}})H_{\lambda_{k}}(\sqrt{P_{kk}})\right)&\leq\exp\left\{-\frac{CN|P_{jk}|^{2}}{t^{3}}\right\}.
\end{align*}
Repeating this for each pair of distinct indices $j,k=m,...,2m$, we conclude that the integral over the complement of $\Omega_{1}\cap\Omega_{2}$ is $O(e^{-\log^{2}N})$.

Since $m=p+2q$ is odd by assumption and $S$ is skew-symmetric, $S$ must have a zero singular value. On the other hand, by Gershgorin's theorem, for $S\in\Omega_{1}\cap\Omega_{2}$, the singular values are contained in the discs of squared radius
\begin{align*}
    P_{jj}\pm\sum_{k\neq j}|P_{jk}|&=(\eta_{\lambda_{j}})^{2}+O\left(\sqrt{\frac{t^{3}}{N}}\log N\right)\\
    &=\left[1+O\left(\frac{\log N}{\sqrt{Nt}}\right)\right](\eta_{\lambda_{j}})^{2}.
\end{align*}
Since $\eta_{\lambda_{j}}\geq Ct$ by \eqref{eq:eta_zB1}, these discs are separated from zero by a distance $Ct$ and so $\Omega_{1}\cap\Omega_{2}=\emptyset$.
\end{proof}

In the second regime, when there are an even number of arguments and one of them is far from the rest, $F_{n}$ is also small. {Here again we might expect this behaviour by considering
\begin{align*}
    \frac{\mbb{E}\prod_{n=1}^{N_{\mbb{R}}}(u_{n}-u_{1})(u_{n}-u_{2})}{\mbb{E}\prod_{n=1}^{N_{\mbb{R}}}|u_{n}-u_{1}|\mbb{E}\prod_{n=1}^{N_{\mbb{R}}}|u_{n}-u_{2}|},
\end{align*}
and noting that correlations between eigenvalues decay in the separation.}
\begin{lemma}\label{lem:Ffar}
Let $p$ be even, $|u_{j}-u|<CN^{-1/2},\,j=2,...,p$ and $|z_{j}-u|<CN^{-1/2},\,j=1,...,q$ for some $u\in(-1+\omega,1-\omega)$. Then for $|u_{1}-u|>N^{-1/2}\log N$ we have
\begin{align}
    |F_{n}(\mbf{u},\mbf{z},X)|&\leq e^{-c\log^{2}N}.\label{eq:Ffar}
\end{align}
\end{lemma}
\begin{proof}
Define the following subsets of the integration domain.
\begin{align*}
    \Omega_{1}&=\bigcap_{j=2}^{m}\left\{|P_{jj}-(\eta_{\lambda_{j}})^{2}|<\sqrt{\frac{t^{3}}{N}}\log N\right\},\\
    \Omega_{2}&=\bigcap_{2\leq j\neq k\leq m}\left\{|P_{jk}|<\sqrt{\frac{t^{3}}{N}}\log N\right\},\\
    \Omega_{3}&=\left\{P_{11}<\delta t^{2}\right\}.
\end{align*}

Arguing as in the proof of Lemma \ref{lem:Fodd}, we conclude that the integral over the complement of $\Omega_{1}\cap\Omega_{2}$ is $O(e^{-\log^{2}N})$. For this we also need the bound in \eqref{eq:phiBound1} which gives
\begin{align*}
    \frac{e^{-\frac{N}{4t}P_{11}}\det^{1/4}\left(P_{11}+|X_{u_{1}}|^{2}\right)}{(\psi_{u_{1}})^{1/2}}&=\exp\left\{-\frac{N}{4}\left[\phi_{u_{1}}(P_{11})-\phi_{u_{1}}(\eta_{u_{1}})\right]\right\}\\
    &\leq\begin{cases}e^{-CNP_{11}/t}&\quad P_{11}>C\|X_{u_{1}}\|^{2}\\
    C&\quad 0\leq P_{11}\leq C\|X_{u_{1}}\|
    \end{cases}.
\end{align*}

Fix $S\in\Omega_{1}$ and assume that $P_{11}>\delta t^{2}$. The latter assumption implies that $P_{11}+|X_{\lambda_{1}}|^{2}$ is invertible so we can use Fischer's inequality with the $2\times2$ block \[\begin{pmatrix}(\mbf{M}_{\bs\lambda}(\mbf{M}_{\bs\lambda})^{*})_{m+1,m+1}&(\mbf{M}_{\bs\lambda}(\mbf{M}_{\bs\lambda})^{*})_{m+1,j}\\(\mbf{M}_{\bs\lambda}(\mbf{M}_{\bs\lambda})^{*})_{j,m+1}&(\mbf{M}_{\bs\lambda}(\mbf{M}_{\bs\lambda})^{*})_{jj}\end{pmatrix}\] for $j=1,...,m$:
\begin{align*}
    |\pf \mbf{M}_{\bs\lambda}|&\leq\left(\prod_{j=1}^{m}\det^{1/4}\left(P_{jj}+|X_{\lambda_{j}}|^{2}\right)\det^{1/4}\left(Q_{jj}+|X_{\lambda_{j}}|^{2}\right)\right)\\
    &\times\det^{1/4}\left(1-|S_{1j}|^{2}|\lambda_{1}-\lambda_{j}|^{2}\wt{H}_{\lambda_{j}}(\sqrt{Q_{jj}})H_{\lambda_{1}}(\sqrt{P_{11}})\right).
\end{align*}
By \eqref{eq:C2.1}, the fact that $t^{2}/C\leq Q_{jj}\leq Ct^{2}$ in $\Omega_{1}$ and our assumption $P_{11}>\delta t^{2}$, we have
\begin{align*}
    \Tr{\wt{H}_{\lambda_{j}}(\sqrt{Q_{jj}})H_{\lambda_{1}}(\sqrt{P_{11}})}&\geq\frac{C\sqrt{P_{11}}\wedge t}{t\sqrt{P_{11}}(|\lambda_{1}-\lambda_{j}|^{2}+\sqrt{P_{11}}\vee t)}\\
    &\geq\frac{C}{\sqrt{P_{11}}(|\lambda_{1}-\lambda_{j}|^{2}+\sqrt{P_{11}})}
\end{align*}
and hence
\begin{align*}
    \det\left(1-|S_{1j}|^{2}|\lambda_{1}-\lambda_{j}|^{2}\wt{H}_{\lambda_{j}}(\sqrt{Q_{jj}})H_{\lambda_{1}}(\sqrt{P_{11}})\right)&\leq\exp\left\{-\frac{CN|S_{1j}|^{2}|\lambda_{1}-\lambda_{j}|^{2}}{\sqrt{P_{11}}(|\lambda_{1}-\lambda_{j}|^{2}+\sqrt{P_{11}})}\right\}.
\end{align*}
Writing $|\pf\mbf{M}_{\bs\lambda}|=\left(|\pf\mbf{M}_{\bs\lambda}|^{\frac{1}{m-1}}\right)^{m-1}$ and applying this bound $m-1$ times for $j=2,...,m$, we obtain
\begin{align*}
    |\pf\mbf{M}_{\bs\lambda}|&\leq\left(\prod_{j=1}^{m}\det^{1/4}\left(P_{jj}+|X_{\lambda_{j}}|^{2}\right)\det^{1/4}\left(Q_{jj}+|X_{\lambda_{j}}|^{2}\right)\right)\\
    &\times\exp\left\{-\frac{CN}{\sqrt{P_{11}}}\sum_{j=2}^{m}\frac{|S_{1j}|^{2}|\lambda_{1}-\lambda_{j}|^{2}}{|\lambda_{1}-\lambda_{j}|^{2}+\sqrt{P_{11}}}\right\}.
\end{align*}
By the assumptions $|\lambda_{j}-u|<CN^{-1/2},\,j=2,...,m$ and $|u_{1}-u|>N^{-1/2}\log N$, we have $|\lambda_{1}-\lambda_{j}|>C|\lambda_{1}-\lambda_{2}|>CN^{-1/2}\log N,\,j=2,...,m$, for sufficiently large $N$. Since we also have $P_{11}=\sum_{j=2}^{m}|S_{1j}|^{2}$, the term on the second line is bounded by
\begin{align*}
    \exp\left\{-\frac{CN}{\sqrt{P_{11}}}\sum_{j=2}^{m}\frac{|S_{1j}|^{2}|\lambda_{1}-\lambda_{j}|^{2}}{|\lambda_{1}-\lambda_{j}|^{2}+\sqrt{P_{11}}}\right\}&\leq\exp\left\{-\frac{CN\sqrt{P_{11}}|\lambda_{1}-\lambda_{2}|^{2}}{|\lambda_{1}-\lambda_{2}|^{2}+\sqrt{P_{11}}}\right\}\\
    &\leq e^{-c\log^{2}N},
\end{align*}
from which we conclude that the integral over the complement of $\Omega_{1}\cap\Omega_{3}$ is $O(e^{-\log^{2}N})$.

By Gershgorin's theorem, we conclude that for $S\in\Omega_{1}\cap\Omega_{2}$, the eigenvalues of the submatrix $\wt{P}:=(P_{jk})_{j,k=2}^{m}$ must lie in the discs centred at $(\eta_{\lambda_{j}})^{2}$ of radius $O\left(\sqrt{\frac{t^{3}}{N}}\log N\right)$. In particular, since $|\lambda_{j}|<1$ for $j=2,...,m$, by \eqref{eq:eta_zB1} we have $\eta_{\lambda_{j}}>Ct$ for $j=2,...,m$ and so these discs are separated from 0 by a distance $Ct^{2}$.

Now consider the matrix $P^{(1)}=|S^{(1)}|^{2}$, where $S^{(1)}$ is the submatrix of $S$ formed by removing the first row and column. For $S\in\Omega_{3}$, we have
\begin{align*}
    \|\wt{P}-P^{(1)}\|&\leq\sqrt{\sum_{j,k=2}^{m}|(P^{(1)}-\wt{P})_{jk}|^{2}}\\
    &=\sqrt{\sum_{j,k=2}^{m}|S_{1j}\bar{S}_{1k}|^{2}}\\
    &< \delta t^{2}.
\end{align*}
Since $S^{(1)}$ is a complex skew-symmetric matrix with odd dimension $m-1$ (recall that we assumed $m$ is even), it must have a zero singular value, and so $P^{(1)}$ has a zero eigenvalue. By Weyl's inequality and the above bound on $\|\wt{P}-P^{(1)}\|$, $\wt{P}$ has an eigenvalue smaller than $\delta t^{2}$. Combined with the previous paragraph, we conclude that $\Omega_{1}\cap\Omega_{2}\cap\Omega_{3}=\emptyset$ for sufficiently small $\delta$.
\end{proof}

In the third regime, when the number of arguments is even and each is near a point $u\in(-1+\omega,1-\omega)$, we have a non-trivial asymptotic behaviour, which can be expressed in terms of the group integral
\begin{align}
    I_{m}(\mbf{u},\mbf{z})&={\frac{1}{(2\pi)^{m(m-1)/4}}}\frac{\wt{\Delta}(\mbf{u},\mbf{z})}{\textup{Vol}(USp(m))}\int_{G}\exp\left\{-\frac{1}{2}\tr(\bs\lambda^{2}+JU^{*}\bs\lambda UJU^{T}\bs\lambda\bar{U})\right\}\,\mathrm{d}_{H}U,\label{eq:Im}
\end{align}
where $G=U(m)/(USp(2))^{m/2}$ and
\begin{align}
    {J}&{=\begin{pmatrix}0&1_{m}\\-1_{m}&0\end{pmatrix}},\label{eq:J}
\end{align}
is the canonical skew-symmetric form of dimension $2m$.
\begin{lemma}\label{lem:Fnear}
Let $p$ be even, $|u_{j}|<C,\,j=2,...,p$, $|z_{j}|<C,\,j=1,...,q$ and $|u_{1}|<\log N$. Then for any $u\in(-1+\omega,1-\omega)$ we have
\begin{align}
    &F_{n}\left(u+\frac{\mbf{u}}{\sqrt{N\sigma_{u}}},u+\frac{\mbf{z}}{\sqrt{N\sigma_{u}}},X\right)\nonumber\\
    &=\left[1+O\left(\frac{\log^{3}N}{\sqrt{Nt^{3}}}\right)\right]\frac{(N\sigma_{u})^{q/2}}{(t^{2}\Tr{H_{u}^{2}}\sigma_{u})^{m(m-1)/4}}\exp\left\{-2\sum_{j=1}^{q}(\Im z_{j})^{2}\right\}I_{m}(\mbf{u},\mbf{z}).\label{eq:Fnear}
\end{align}
\end{lemma}
\begin{proof}
Let $\wt{\lambda}_{j}=u+\frac{\lambda_{j}}{\sqrt{N\sigma_{u}}}$. When $|u_{1}|<\log N$, the estimates in the proof of Lemma \ref{lem:Ffar} imply that we can restrict to the region in which the singular values $\sigma_{j}$ of $S$ are in $O\left(\sqrt{\frac{t}{N}}\log N\right)$ neighbourhoods of $\eta_{\wt{\lambda}_{j}}$. By \eqref{eq:eta_zB2}, we have
\begin{align*}
    |\eta_{\wt{\lambda}_{j}}-\eta_{u}|&\leq Ct|\lambda_{j}|<\frac{Ct\log N}{\sqrt{N}},
\end{align*}
and so all the singular values $\sigma_{j}$ are in an $O\left(\sqrt{\frac{t}{N}}\log N\right)$ neighbourhood of $\eta_{u}$. Let \[\mbb{R}^{m/2}_{>,+}=\{\bs\sigma\in\mbb{R}^{m/2}:0\leq\sigma_{1}<\cdots<\sigma_{m/2}\}.\] We make the change of variables $S=U\Sigma U^{T}$, where
\begin{align}
    \Sigma&=\begin{pmatrix}0&\bs\sigma\\-\bs\sigma&0\end{pmatrix},
\end{align}
$\bs\sigma\in\mbb{R}^{m/2}_{>,+}$ and $U\in G$ (any even-dimensional skew-symmetric matrix can be written in this form, see e.g. \cite[Corollary 2.6.6]{horn_matrix_2012}). The Lebesgue measure transforms as (see Appendix \ref{sec:jacobian} for the short proof)
\begin{align}
    \diff S&={2^{m/2}}\Delta^{4}(\bs\sigma^{2})\left(\prod_{j=1}^{m/2}\sigma_{j}\diff\sigma_{j}\right)\,\mathrm{d}_{H}U,\label{eq:Jacobian}
\end{align}
where \[\Delta^{4}(\bs\sigma^{2}):=\prod_{j<k}(\sigma_{j}^{2}-\sigma_{k}^{2})^{4}.\] Using $\pf USU^{T}=\det U\pf S$ we have
\begin{align*}
    \pf \mbf{M}_{\bs\lambda}(S)&=\pf\begin{pmatrix}\Sigma\otimes1_{n}&1_{m}\otimes X_{u}-U^{*}\bs\lambda U\otimes1_{n}\\
    -1_{m}\otimes X^{T}_{u}+U^{T}\bs\lambda\bar{U}\otimes1_{n}&-\Sigma\otimes1_{n}\end{pmatrix}.
\end{align*}
Let
\begin{align*}
    \mbf{M}_{u}&=\begin{pmatrix}\Sigma\otimes1_{n}&1_{m}\otimes X_{u}\\
    -1_{m}\otimes X_{u}^{T}&-\Sigma\otimes1_{n}\end{pmatrix},
\end{align*}
and
\begin{align*}
    \bs\Lambda&=\begin{pmatrix}0&-U^{*}\bs\lambda U\otimes1_{n}\\U^{T}\bs\lambda\bar{U}\otimes1_{n}&0\end{pmatrix}.
\end{align*}
Since the blocks of $\mbf{M}_{u}$ commute we have
\begin{align*}
    (\mbf{M}_{u})^{-1}&=\begin{pmatrix}-(\Sigma\otimes1_{n})\wt{\mbf{H}}_{u}(\Sigma)&-(1_{m}\otimes X_{u})\mbf{H}_{u}(\Sigma)\\
    \mbf{H}_{u}(\Sigma)(1_{m}\otimes X_{u}^{T})&(\Sigma\otimes1_{n})\mbf{H}_{u}(\Sigma)\end{pmatrix},
\end{align*}
where we have defined 
\begin{align*}
    \mbf{H}_{u}(\Sigma)&=\left(-\Sigma^{2}\otimes1_{n}+1_{m}\otimes|X_{u}|^{2}\right)^{-1},\\
    \wt{\mbf{H}}_{u}(\Sigma)&=\left(-\Sigma^{2}\otimes1_{n}+1_{m}\otimes|X^{T}_{u}|^{2}\right)^{-1}.
\end{align*}
Since $\sigma_{j}=O(t)$, we have $\|(\mbf{M}_{u})^{-1}\|\leq C/t$ and so $\left\|\left(1+sN^{-1/2}(\mbf{M}_{u})^{-1}\mbf{\Lambda}\right)^{-1}\right\|<C$ for $s\in[0,1]$. We use the identity
\begin{align}
    \pf\left(\mbf{M}_{u}+\frac{1}{\sqrt{N}}\mbf{\Lambda}\right)&=\pf \mbf{M}_{u} \exp\left\{\frac{1}{2\sqrt{N}}\tr (\mbf{M}_{u})^{-1}\mbf{\Lambda}-\frac{1}{4N}\tr((\mbf{M}_{u})^{-1}\mbf{\Lambda})^{2}\right.\nonumber\\&\left.+\frac{1}{2N^{3/2}}\int_{0}^{1}s^{3}\tr(1+sN^{-1/2}(\mbf{M}_{u})^{-1}\mbf{\Lambda})^{-1}((\mbf{M}_{u})^{-1}\mbf{\Lambda})^{3}\diff s\right\}\label{eq:pfPerturbation}
\end{align}
and estimate each term using the conditions on $X\in\mc{X}_{n}(\gamma,\omega)$.

The pfaffian in front of the exponential is
\begin{align*}
    \pf\mbf{M}_{u}&=\prod_{j=1}^{m/2}\det\left(\sigma_{j}^{2}+|X_{u}|^{2}\right).
\end{align*}
Writing
\begin{align*}
    \det\left(\sigma_{j}^{2}+|X_{u}|^{2}\right)&=\det\left[\eta^{2}_{u}+|X_{u}|^{2}\right]\cdot\det\left[1-i(\sigma_{j}-\eta_{u})G_{u}\right],
\end{align*}
Taylor expanding the second determinant to second order and estimating the remainder using the bound $|\sigma_{j}-\eta_{u}|<\sqrt{\frac{t}{N}}\log N$, we obtain
\begin{align*}
    e^{-\frac{N}{t}\sum_{j=1}^{m/2}\sigma_{j}^{2}}\pf\mbf{M}_{u}&=\left[1+O\left(\frac{\log^{3}N}{\sqrt{Nt}}\right)\right](\psi_{u})^{m}\exp\left\{-2N\Tr{(\eta_{u}H_{u})^{2}}\sum_{j=1}^{m/2}(\sigma_{j}-\eta_{u})^{2}\right\}
\end{align*}

The first term in the exponent of \eqref{eq:pfPerturbation} is
\begin{align*}
    \frac{1}{2\sqrt{N}}\tr(\mbf{M}_{u})^{-1}\mbf{\Lambda}&=-\sqrt{N}\sum_{j=1}^{m/2}\Tr{X_{u}H_{u}(\sigma_{j})}\left[(U^{*}\bs\lambda U)_{jj}+(U^{*}\bs\lambda U)_{m/2+j,m/2+j}\right].
\end{align*}
In terms of $G_{u}$ we have
\begin{align*}
    \frac{1}{2\sqrt{N}}\tr(\mbf{M}_{u})^{-1}\mbf{\Lambda}&=-\sqrt{N}\sum_{j=1}^{m/2}\Tr{G_{u}(\sigma_{j})\Re F}\left[(U^{*}\bs\lambda U)_{jj}+(U^{*}\bs\lambda U)_{m/2+j,m/2+j}\right].
\end{align*}
Using the resolvent identity three times we find
\begin{align*}
    \frac{1}{2\sqrt{N}}\tr(\mbf{M}_{u})^{-1}\mbf{\Lambda}&=-\sqrt{N}\Tr{G_{u}\Re F}\sum_{j=1}^{m}\lambda_{j}\\&+\sqrt{N}\Tr{(G_{u})^{2}\Re F}\sum_{j=1}^{m/2}(\sigma_{j}-\eta_{u})\left[(U^{*}\bs\lambda U)_{jj}+(U^{*}\bs\lambda U)_{m/2+j,m/2+j}\right]\\
    &+\sqrt{N}\sum_{j=1}^{m/2}\Tr{G_{u}(\sigma_{j})(G_{u})^{2}\Re F}(\sigma_{j}-\eta_{u})^{2}\left[(U^{*}\bs\lambda U)_{jj}+(U^{*}\bs\lambda U)_{m/2+j,m/2+j}\right].
\end{align*}
The last term can be estimated by Cauchy-Schwarz and \eqref{eq:C3.1} to give
\begin{align*}
    \frac{1}{2\sqrt{N}}\tr(\mbf{M}_{u})^{-1}\mbf{\Lambda}&=-\sqrt{N}\Tr{G_{u}\Re F}\sum_{j=1}^{m}\lambda_{j}\\&+\sqrt{N}\Tr{(G_{u})^{2}\Re F}\sum_{j=1}^{m/2}(\sigma_{j}-\eta_{u})\left[(U^{*}\bs\lambda U)_{jj}+(U^{*}\bs\lambda U)_{m/2+j,m/2+j}\right]\\&+O\left(\frac{\log^{2}N}{\sqrt{Nt}}\right)
\end{align*}

By similar estimates, the second term in the exponent of \eqref{eq:pfPerturbation} is
\begin{align*}
    \frac{1}{4N}\tr((\mbf{M}_{u})^{-1}\mbf{\Lambda})^{2}&=\Tr{(G_{u}\Im F)^{2}}\sum_{j=1}^{m}\lambda_{j}^{2}+\frac{1}{2}\eta^{2}_{u}\Tr{H_{u}\wt{H}_{u}}\tr\left(\bs\lambda^{2}+JU^{*}\bs\lambda UJU^{T}\bs\lambda \bar{U}\right)\\
    &+O\left(\frac{\log^{2}N}{\sqrt{Nt}}\right).
\end{align*}

The third term in the exponent of \eqref{eq:pfPerturbation} can be bounded by Cauchy-Schwarz and \eqref{eq:C3.1}:
\begin{align*}
    &\left|\tr(1+sN^{-1/2}(\mbf{M}_{u})^{-1}\mbf{\Lambda})^{-1}((\mbf{M}_{u})^{-1}\mbf{\Lambda})^{3}\right|\\&\leq\frac{1}{1-C/\sqrt{Nt^{2}}}\left(\tr(\mbf{M}_{u})^{-1}\mbf{\Lambda}(\mbf{M}_{u})^{-1}(\mbf{M}^{*}_{u})^{-1}\mbf{\Lambda}^{*}(\mbf{M}^{*}_{u})^{-1}\right)^{1/2}\\
    &\times\left(\tr\mbf{\Lambda}(\mbf{M}_{u})^{-1}\mbf{\Lambda}\mbf{\Lambda}^{*}(\mbf{M}^{*}_{u})^{-1}\mbf{\Lambda}^{*}\right)^{1/2}\\
    &\leq\frac{CN\log^{3}N}{t^{3/2}}.
\end{align*}

We can take the error outside the integral, change variables $\sigma_{j}\mapsto\eta_{u}+\sigma_{j}/\sqrt{2N\Tr{(\eta_{u}H_{u})^{2}}}$ and extend $\sigma_{j}$ to $\mbb{R}$ to obtain
\begin{align}
    F_{n}(\mbf{u},\mbf{z},X)&=\left[1+O\left(\frac{\log^{3}N}{\sqrt{Nt^{3}}}\right)\right]\frac{1}{{2^{m/4}\pi^{m(m-1)/2}}(t^{2}\Tr{H_{u}^{2}}\sigma_{u})^{m(m-1)/4}}\nonumber\\
    &\times\prod_{j=1}^{m}\left(\frac{\psi_{u}}{\psi_{\wt{\lambda}_{j}}}\exp\left\{-\sqrt{N}\Tr{G_{u}\Re F}\lambda_{j}-\Tr{(G_{u}\Re F)^{2}}\lambda_{j}^{2}\right\}\right)\nonumber\\
    &\times\wt{\Delta}(\mbf{u},\mbf{z})\int_{G}\exp\left\{-\frac{1}{2}\eta^{2}_{u}\Tr{H_{u}\wt{H}_{u}}\tr\left(\bs\lambda^{2}+JU^{*}\bs\lambda UJU^{T}\bs\lambda\bar{U}\right)\right\}f(U)\,\mathrm{d}_{H}U,\label{eq:Fapprox}
\end{align}
where
\begin{align*}
    f(U)&=\int_{\mbb{R}^{m/2}_{>,+}}\exp\left\{-\sum_{j=1}^{m/2}\sigma^{2}_{j}+\frac{\Tr{(G_{u})^{2}\Re F}}{\sqrt{2\Tr{(\eta_{u}H_{u})^{2}}}}\tr\begin{pmatrix}\bs\sigma&0\\0&\bs\sigma\end{pmatrix}U^{*}\bs\lambda U\right\}\Delta^{4}(\bs\sigma)\diff\bs\sigma
\end{align*}
The ratio $\psi_{u}/\psi_{\wt{\lambda}_{j}}$ can be estimated by Lemma \ref{lem:psi}. To evaluate the integral over $\bs\sigma$, we note that the rest of the integrand and the measure are invariant under the transformation $U\mapsto UV$ where $V\in USp(m)$ and so we can replace $f(U)$ with its average over $USp(m)$:
\begin{align*}
    f_{avg}(U)&=\frac{1}{\text{Vol}(USp(m))}\int_{USp(m)}f(UV)\,\mathrm{d}_{H}V\\
    &=\frac{1}{\text{Vol}(USp(m))}\int\exp\left\{-\frac{1}{2}\tr M^{2}+\frac{\Tr{(G_{u})^{2}\Re F}}{\sqrt{2\Tr{(\eta_{u}H_{u})^{2}}}}\tr MU^{*}\bs\lambda U\right\}\Delta^{4}(\bs\sigma)\diff\bs\sigma \,\mathrm{d}_{H}V.
\end{align*}
The latter integral is over the set of matrices
\begin{align*}
    M&=V\begin{pmatrix}\bs\sigma&0\\0&\bs\sigma\end{pmatrix}V^{*},\quad V\in USp(m),
\end{align*}
which are precisely the Hermitian, self-dual matrices $M=M^{*}=M^{R}$, where the involution $R$ is defined by
\begin{align*}
    M^{R}&=J^{-1}M^{T}J.
\end{align*}
We then observe that that the Lebesgue measure on this space is $\diff M=\Delta^{4}(\bs\sigma)\diff\bs\sigma \,\mathrm{d}_{H}V$ and
\begin{align*}
    \tr MU^{*}\bs\lambda U&=\tr M\frac{U^{*}\bs\lambda U+(U^{*}\bs\lambda U)^{R}}{2}.
\end{align*}
Assume for the moment that $\bs\lambda\in\mbb{R}^{m}$; then the following change variables preserves the integration domain and measure:
\begin{align*}
    M&\mapsto M+\frac{\Tr{(G_{u})^{2}\Re F}}{2\sqrt{2\Tr{(\eta_{u}H_{u})^{2}}}}\left(U^{*}\bs\lambda U+(U^{*}\bs\lambda U)^{R}\right).
\end{align*}
After this change of variables we obtain
\begin{align*}
    f_{avg}(U)&=\frac{\pi^{m(m-1)/4}}{2^{m(m-2)/4}}\exp\left\{\frac{\Tr{(G_{u})^{2}\Re F}^{2}}{8\Tr{(\eta_{u}H_{u})^{2}}}\tr(\bs\lambda^{2}-JU^{*}\bs\lambda UJU^{T}\bs\lambda\bar{U})\right\}.
\end{align*}
Extending this to $\bs\lambda\in\mbb{C}^{m}$ by analytic continuation and inserting it into the expression for $F_{n}(\bs\lambda,X)$ in \eqref{eq:Fapprox} we obtain \eqref{eq:Fnear}.
\end{proof}
As shown in \cite{tribe_averages_2023}, the group integral $I_{m}$ can be evaluated explicitly in terms of a Pfaffian. For our purposes the explicit evaluation is not necessary and so we leave it in this form. The main property to note is that it decays exponentially in $|\lambda_{1}-\lambda_{j}|^{2}$, as can be seen by writing the exponent as a quadratic form in the skew-symmetric unitary matrix $W=UJU^{T}$:
\begin{align*}
    \tr\left(\bs\lambda^{2}+JU^{*}\bs\lambda UJU^{T}\bs\lambda\bar{U}\right)&=\sum_{j=1}^{m}\lambda_{j}^{2}-2\sum_{j<k}^{m}\lambda_{j}\lambda_{k}|W_{jk}|^{2}.
\end{align*}

Finally, in the fourth regime the arguments are all complex and separated from the real axis by a fixed distance. In this case we have a different asymptotic behaviour.
\begin{lemma}\label{lem:Fcomplex}
Let $|z_{j}|<C,\,j=1,...,q$. For any $z\in\mbb{D}_{+,\omega}$ we have
\begin{align}
    F_{n}\left(z+\frac{\mbf{z}}{\sqrt{N\sigma_{z}}},X\right)&=\left[1+O\left(\frac{\log^{3}N}{\sqrt{Nt^{3}}}\right)\right]\left(\frac{2\pi^{3}N}{t^{2}\Tr{H_{z}^{2}}}\right)^{q/2}\left(\frac{4(\Im z)^{2}}{t^{2}\Tr{H_{z}^{2}}\sigma_{z}\tau_{z}}\right)^{q(q-1)/2}\nonumber\\&\times\det\left[\frac{1}{\pi}e^{-\frac{1}{2}(|z_{j}|^{2}+|z_{k}|^{2})+\bar{z}_{j}z_{k}}\right]_{j,k=1}^{q},\label{eq:Fcomplex}
\end{align}
where
\begin{align}
    \tau_{z}&=\left|1-t\Tr{H_{z}X^{*}_{z}X_{\bar{z}}H_{\bar{z}}}\right|^{2}-t^{2}\eta_{z}^{4}\Tr{H_{z}H_{\bar{z}}}\Tr{\wt{H}_{z}\wt{H}_{\bar{z}}}.\label{eq:tau}
\end{align}
\end{lemma}
\begin{proof}
Let $w_{j}=z+z_{j}/\sqrt{N\sigma_{z}}$. Since $z\in\mbb{D}_{+,\omega}$, there is a constant $C$ such that $\Im z>C>0$. Note that now $S$ is a $2q\times2q$ matrix. Let $P=|S|^{2}$ and $Q=|S^{*}|^{2}$. The argument in the proof of Lemma \ref{lem:Ffar} allows us to restrict to the region in which $\sqrt{P_{jj}}$ and $\sqrt{Q_{jj}}$ are in $O\left(\sqrt{\frac{t}{N}}\log N\right)$ neighbourhoods of $\eta_{w_{j}}$, which are themselves contained in an $O\left(\frac{t\log N}{\sqrt{N}}\right)$ neighbourhood of $\eta_{z}$ by \eqref{eq:eta_zB2}. Now we make a block decomposition of $S$:
\begin{align*}
    S&=\begin{pmatrix}S_{1}&T\\-T^{T}&S_{2}^{*}\end{pmatrix}.
\end{align*}
Using Fischer's inequality and extracting the $2\times2$ block \[\begin{pmatrix}(\mbf{M}_{\bs\lambda}(\mbf{M}_{\bs\lambda})^{*})_{jj}&(\mbf{M}_{\bs\lambda}(\mbf{M}_{\bs\lambda})^{*})_{jk}\\(\mbf{M}_{\bs\lambda}(\mbf{M}_{\bs\lambda})^{*})_{kj}&(\mbf{M}_{\bs\lambda}(\mbf{M}_{\bs\lambda})^{*})_{kk}\end{pmatrix}\] for $j=2m+1,,..,3m$ and $k=1,...,m$ gives
\begin{align*}
    |\pf \mbf{M}_{\bs\lambda}|&\leq\prod_{j=1}^{m}\det^{1/2}\left(P_{jj}+|X_{w_{j}}|^{2}\right)\det^{1/2}\left(Q_{jj}+|X_{w_{j}}|^{2}\right)\\
    &\times\det\left(1-|S_{1,jk}|^{2}|\bar{w}_{j}-w_{k}|^{2}\wt{H}_{w_{j}}(\sqrt{P_{jj}})H_{\bar{w}_{k}}(\sqrt{Q_{kk}})\right).
\end{align*}
By \eqref{eq:C2.1}, the fact that $P_{jj},Q_{jj}=O(t)$ and the assumption $\Im z>C>0$, we have
\begin{align*}
    |w_{j}-\bar{w}_{k}|^{2}\Tr{\wt{H}_{w_{j}}(\sqrt{P_{jj}})H_{\bar{w}_{k}}(\sqrt{Q_{jj}})}&\geq\frac{C|w_{j}-\bar{w}_{k}|^{2}}{t(|w_{j}-\bar{w}_{k}|^{2}+t)}\\
    &\geq\frac{C}{t},
\end{align*}
from which we conclude that we can restrict to the region $|S_{1,jk}|<\sqrt{\frac{t}{N}}\log N$ for $1\leq j,k\leq m$. Doing the same for $j=3m+1,...,4m$ and $k=m+1,...,2m$ we conclude that we can restrict to the region in which $\|S_{1}\|,\,\|S_{2}\|\leq\sqrt{\frac{t}{N}}\log N$.

Define
\begin{align*}
    \mbf{M}_{0}&=\begin{pmatrix}T\otimes1_{n}&1_{m}\otimes X-\mbf{w}\otimes1_{n}\\-1_{m}\otimes X^{*}+\bar{\mbf{w}}\otimes1_{n}&T^{*}\otimes1_{n}\end{pmatrix},\\
    \mbf{M}_{1}&=\begin{pmatrix}0&\mbf{M}_{0}\\-(\mbf{M}_{0})^{T}&0\end{pmatrix},
\end{align*}
and
\begin{align*}
    \mbf{M}_{2}&=\begin{pmatrix}S_{1}&&&\\&S_{2}&&\\&&S_{2}^{*}&\\&&&S_{1}^{*}\end{pmatrix}\otimes1_{n}.
\end{align*}
Then we have
\begin{align*}
    \pf \mbf{M}_{\bs\lambda}&=\det(\mbf{M}_{0})\exp\left\{-\frac{1}{4}\tr((\mbf{M}_{1})^{-1}\mbf{M}_{2})^{2}\right.\\
    &\left.+\frac{1}{2}\int_{0}^{1}s^{3}\tr\left(1+s(\mbf{M}_{1})^{-1}\mbf{M}_{2}\right)^{-1}((\mbf{M}_{1})^{-1}\mbf{M}_{2})^{4}\right\}\\
    &=\det(\mbf{M}_{0}+\mbf{Z})\exp\left\{\frac{N}{2}\left(\Tr{H_{z}X^{*}_{z}X_{\bar{z}}H_{\bar{z}}}\tr|S_{1}|^{2}+\Tr{X_{z}H_{z}H_{\bar{z}}X^{*}_{\bar{z}}}\tr|S_{2}|^{2}\right)\right.\\
    &\left.+\frac{N\eta_{z}^{2}}{2}\left(\Tr{\wt{H}_{z}\wt{H}_{\bar{z}}}\tr S_{1}S_{2}^{*}+\Tr{H_{z}H_{\bar{z}}}\tr S_{1}^{*}S_{2}\right)+O\left(\frac{\log^{3}N}{\sqrt{Nt^{3}}}\right)\right\}.
\end{align*}

Integrating first over $S_{1}$ and $S_{2}$ we find
\begin{align*}
    F_{n}(\mbf{w},X)&=\left(\frac{4(\Im z)^{2}}{\tau_{z}}\right)^{q(q-1)/2}\\
    &\times |\Delta(\mbf{w})|^{2}\left(\frac{N}{\pi t}\right)^{q^{2}}\int_{\mbb{M}_{q}(\mbb{C})}\left[1+O\left(\frac{\log^{3} N}{\sqrt{Nt^{3}}}\right)\right]e^{-\frac{N}{t}\tr T^{*}T}\det\mbf{M}_{0}\diff T,
\end{align*}
{where the factor $(2\Im z)^{q(q-1)}$ comes from
\begin{align*}
    \prod_{j<k}|w_{j}-\bar{w}_{k}|^{2}&=\left[1+O\left(\frac{1}{\sqrt{N}}\right)\right](2\Im z)^{q(q-1)}.
\end{align*}}
This is exactly the same as the integral that arose in the case of complex matrices treated in \cite[Section 6]{maltsev_bulk_2023} and so we can simply borrow the end result \cite[Lemma 5.1]{maltsev_bulk_2023}.
\end{proof}

\section{Properties of Gaussian measures on Stiefel Manifolds}\label{sec:gaussian}
Fix $\gamma\in(0,1/3)$, $\epsilon\in(0,\gamma)$, $\omega\in(0,1)$, $X\in\mc{X}_{n}(\gamma,\omega)$ and $t\geq N^{-\gamma+\epsilon}$. It is important to keep in mind throughout this section that if $z\in\mbb{D}_{\omega}$ then $\eta^{X}_{z}$ defined by \eqref{eq:eta_z} is $O(t)$. {We will also drop the superscript $X$ from $\eta^{X}_{z},\,\phi^{X}_{z},\,$ etc., it being understood that these are all functions of the given matrix $X$.} We will study the measures \eqref{eq:mu_n} and \eqref{eq:nu_n} on $O(n,1)=S^{n-1}$ and $O(n,2)=O(n)/O(n-2)$, which we call Gaussian because their density with respect to the Haar measure is the exponential of a quadratic form. We begin with the sphere $S^{n-1}$.

\paragraph{$S^{n-1}$:}
For $X\in\mbb{M}_{n}$ and $u\in\mbb{R}$ we recall the definition of the probability measure $\mu_{n}$ on $S^{n-1}$:
\begin{align*}
    \diff\mu_{n}(\mbf{v};u,X)&=\frac{{\psi_{u}}}{K_{n}(u,X)}{\left(\frac{N}{2\pi t}\right)^{n/2-1}}e^{-\frac{N}{2t}\left\|X_{u}\mbf{v}\right\|^{2}}\,\mathrm{d}_{H}\mbf{v},
\end{align*}
where $\psi_{u}:=\psi^{X}_{u}$ is defined in \eqref{eq:psi}, and the normalisation is given by
\begin{align*}
    K_{n}(u,X)&={\psi_{u}\left(\frac{N}{2\pi t}\right)^{n/2-1}}\int_{S^{n-1}}e^{-\frac{N}{2t}\|X_{u}\mbf{v}\|^{2}}\,\mathrm{d}_{H}\mbf{v}.
\end{align*}

The analysis of $\mu_{n}$ will be based on the following identity for expectation values of the exponential of a quadratic form. To state it in a more compact form we define the function $g_{u}:\mbb{M}_{n}(\mbb{R})\to\mbb{R}$ by
\begin{align}
    g_{u}(M)&=\int_{-\infty}^{\infty}e^{\frac{iNp}{2t}}\det^{-1/2}\left[1+ip\left(\eta^{2}_{u}+|X_{u}|^{2}-\frac{2t}{N}M\right)^{-1}\right]\diff p.\label{eq:gF}
\end{align}
\begin{lemma}\label{lem:Kduality}
Let $M\in\mbb{M}_{n}$ such that $2t\|\sqrt{H_{u}}M\sqrt{H_{u}}\|/N<1$. Then
\begin{align}
    \mbb{E}_{\mu}\left[e^{\mbf{v}^{T}M\mbf{v}}\right]&=\frac{g_{u}(M)}{K_{n}(u,X)}\det^{-1/2}\left(1-\frac{2t}{N}\sqrt{H_{u}}M\sqrt{H_{u}}\right).\label{eq:Kduality}
\end{align}
\end{lemma}
\begin{proof}
This is simply an application of Lemma \ref{lem:sphericalint}. Let
\begin{align*}
    f(\mbf{v})&=\exp\left\{-\frac{N}{2t}\mbf{v}^{T}\left(\eta^{2}_{u}+|X_{u}|^{2}-\frac{2t}{N}M\right)\mbf{v}\right\};
\end{align*}
then the assumption $2t\|\sqrt{H_{u}}M\sqrt{H_{u}}\|/N<1$ implies that $f\in L^{1}$. We can calculate the transform $\hat{f}$ explicitly:
\begin{align*}
    \hat{f}(p)&=\frac{1}{\pi}\int_{\mbb{R}^{n}}e^{-\frac{N}{2t}\mbf{x}^{T}\left(\eta_{u}^{2}+|X_{u}|^{2}+ip-\frac{2t}{N}M\right)\mbf{x}}d\mbf{x}\\
    &=\frac{1}{\pi}\left(\frac{2\pi t}{N}\right)^{n/2}\det^{-1/2}\left(\eta_{u}^{2}+|X_{u}|^{2}+ip-\frac{2t}{N}M\right).
\end{align*}
We have $\hat{f}\in L^{1}$ since the real part of the matrix in the determinant is positive definite and the determinant decays as $e^{-Np^{2}}$. Appplying Lemma \ref{lem:sphericalint} we have
\begin{align*}
    \int_{S^{n-1}}f(\mbf{v})\,\mathrm{d}_{H}\mbf{v}&=\int_{-\infty}^{\infty}e^{ip}\hat{f}(p)\diff p.
\end{align*}
Changing variable $p\mapsto\frac{Np}{2t}$ we obtain \eqref{eq:Kduality}.
\end{proof}

The asymptotics of the normalisation are given by the following lemma.
\begin{lemma}\label{lem:KAsymp}
For $|u|>4\|X\|$ we have
\begin{align}
    K_{n}(u,X)&\leq e^{-\frac{cN}{t}u^{2}};\label{eq:KBound1}
\end{align}
for any $u\in\mbb{R}$ we have
\begin{align}
    K_{n}(u,X)&\leq \frac{C\|X_{u}\|^{2}}{\sqrt{n}};\label{eq:KBound2}
\end{align}
for $u\in(-1+\omega,1-\omega)$ we have
\begin{align}
    K_{n}(u,X)&=\left[1+O\left(\frac{\log^{3}N}{\sqrt{Nt}}\right)\right]\sqrt{\frac{4\pi}{N\Tr{H_{u}^{2}}}}.\label{eq:KAsymp}
\end{align}
\end{lemma}
\begin{proof}
When $|u|>4\|X\|$, we have $|1-X/u|^{2}>1/2$ and so
\begin{align*}
    K_{n}(u;X)&=\psi_{u}\left(\frac{N}{2\pi t}\right)^{n/2-1}\int_{S^{n-1}}\exp\left\{-\frac{Nu^{2}}{2t}\mbf{v}^{T}\left|1-\frac{X}{u}\right|^{2}\mbf{v}\right\}\,\mathrm{d}_{H}\mbf{v}\\
    &\leq \psi_{u}e^{-\frac{N}{8t}u^{2}}\cdot\left(\frac{N}{2\pi t}\right)^{n/2-1}\int_{S^{n-1}}e^{-\frac{N}{4t}\|X_{u}\mbf{v}\|^{2}}\,\mathrm{d}_{H}\mbf{v}.
\end{align*}
We can evaluate the integral using \eqref{eq:Kduality} with $M=0$ after replacing $t$ with $2t$:
\begin{align*}
    K_{n}(u;X)&=2^{n/2-1}e^{-\frac{N}{8t}u^{2}-\frac{N}{4t}\eta_{u}^{2}}\int_{-\infty}^{\infty}e^{\frac{iNp}{4t}}\det^{-1/2}\left(1+ipH_{u}\right)\diff p\\
    &\leq 2^{n/2-1}e^{-\frac{N}{8t}u^{2}}\int_{-\infty}^{\infty}\det^{-n/4}\left[1+\left(\frac{p}{\eta_{u}^{2}+\|X_{u}\|^{2}}\right)^{2}\right]\diff p\\
    &\leq \frac{C2^{n/2-1}\|X_{u}\|^{2}}{\sqrt{n}}e^{-\frac{N}{8t}u^{2}}\\
    &\leq e^{-\frac{cN}{t}u^{2}}.
\end{align*}
{In the last line we have used the assumptions $n\leq CN$ and $|u|>4\|X\|$ to conclude that
\begin{align*}
    2^{n}e^{-\frac{cN}{t}u^{2}}&\leq e^{-\frac{cN}{t}u^{2}+CN\log 2}\leq e^{-\frac{c'N}{t}u^{2}}.
\end{align*}}

Using \eqref{eq:Kduality} for $M=0$ we have
\begin{align}
    K_{n}(u;X)&=g_{u}(0)=\int_{-\infty}^{\infty}e^{\frac{iNp}{2t}}\det^{-1/2}\left(1+ipH_{u}\right)\diff p.
\end{align}
For general $u\in\mbb{R}$, we have already obtained a bound for the integral:
\begin{align*}
    \int_{-\infty}^{\infty}e^{\frac{iNp}{2t}}\det^{-1/2}\left(1+ipH_{u}\right)\diff p&\leq\frac{C\|X_{u}\|^{2}}{\sqrt{n}}.
\end{align*}

Let $u\in(-1+\omega,1-\omega)$; then by Lemma \ref{lem:eta} we have $\eta_{u}=O(t)$. First, note that the determinant decays as $e^{-np^{2}/\|X\|^{4}}$ and so we truncate the integral to $|p|<C\|X\|^{2}$. In this region we have
\begin{align*}
    \left|\det^{-1/2}\left(1+ipH_{u}\right)\right|&=\det^{-1/4}\left(1+p^{2}H_{u}^{2}\right)\\
    &\leq\exp\left\{-\frac{1}{4}\left(1+\frac{p^{2}}{\eta_{u}^{4}}\right)^{-1}p^{2}\tr H_{u}^{2}\right\}\\
    &\leq\exp\left\{-\frac{CNp^{2}}{\eta_{u}^{3}(1+p^{2}/\eta_{u}^{4})}\right\},
\end{align*}
where the last inequality follows from \eqref{eq:C1.2}. This allows us to truncate to $|p|<\sqrt{\frac{t^{3}}{N}}\log N$. Now we can Taylor expand the determinant to find
\begin{align*}
    \det^{-1/2}\left(1+ipH_{u}\right)&=\left[1+O\left(\frac{\log^{3}N}{\sqrt{Nt}}\right)\right]\exp\left\{-\frac{iNp}{2t}-\frac{1}{4}p^{2}\tr H_{u}^{2}\right\}.
\end{align*}
Taking the error outside and extending the integral to the whole real axis we obtain \eqref{eq:KAsymp}.
\end{proof}

Next we show that quadratic forms in $\mbf{v}\sim\mu_{n}$ are subexponential.
\begin{lemma}\label{lem:muConc}
Let $u\in(-1+\omega,1-\omega),\,M\in\mbb{M}^{sym}_{n}(\mbb{R})$ and
\begin{align}
    \beta_{\mu}(M)&=\frac{N\eta^{2}_{u}(X)}{4t\|M\|},\label{eq:p_F}\\
    \mc{K}_{\mu}(M)&=\frac{2t^{2}\Tr{(H_{u}M)^{2}}}{N}.\label{eq:mcK_F}
\end{align}
Then for any $r>0$ and $|\beta|<\beta_{\mu}(M)$, we have
\begin{align}
    \mu_{n}\left(\left\{\left|\mbf{v}^{T}M\mbf{v}-t\Tr{H_{u}M}\right|>r\right\}\right)&\leq C\exp\left\{-\beta r+\mc{K}_{\mu}(M)\beta^{2}\right\}.\label{eq:muConc}
\end{align}
\end{lemma}
\begin{proof}
The probability bound will follow by Markov's inequality from the bound
\begin{align*}
    m(\beta,M)&:=e^{-\beta t\Tr{H_{u}M}}\mbb{E}_{\mu}\left[e^{\beta\mbf{v}^{T}M\mbf{v}}\right]\leq C\exp\left\{\mc{K}_{\mu}(M)\beta^{2}\right\},\quad |\beta|<\beta_{\mu}(M),
\end{align*}
on the moment generating function. Using \eqref{eq:Kduality} we have
\begin{align*}
    m(\beta,M)&=\frac{g_{u}(M)}{K_{n}(u,X)}\cdot\frac{e^{-\beta t\Tr{H_{u}M}}}{\det^{1/2}\left(1-\frac{2\beta t}{N}\sqrt{H_{u}}M\sqrt{H_{u}}\right)}.
\end{align*}
If $|\beta|<\beta_{\mu}(M)$, then since $M$ is symmetric and $\frac{2|\beta|t}{N}\left\|\sqrt{H_{u}}M\sqrt{H_{u}}\right\|<1/2$ we have
\begin{align*}
    \tr\left(\eta_{u}^{2}+|X_{u}|^{2}-\frac{2\beta t}{N}M\right)^{-2}&\geq\frac{1}{4}\tr H_{u}^{2}.
\end{align*}
Arguing as in the proof of Lemma \ref{lem:KAsymp}, we obtain
\begin{align*}
    \frac{g_{u}(M)}{K_{n}}&\leq C.
\end{align*}
The fact that $M$ is symmetric, the norm bound $\frac{2|\beta|t}{N}\left\|\sqrt{H_{u}}M\sqrt{H_{u}}\right\|<1/2$ and the identity $\log(1+x)=x-\int_{0}^{1}sx(1+sx)^{-1}ds$ imply that
\begin{align*}
    \frac{e^{-\beta t\Tr{H_{u}M}}}{\det^{1/2}\left(1-\frac{2\beta t}{N}\sqrt{H_{u}}M\sqrt{H_{u}}\right)}&\leq\exp\left\{\mc{K}_{\mu}(M)\beta^{2}\right\}.
\end{align*}
\end{proof}

From this we obtain an approximation for $\det\left[(1_{2}\otimes\mbf{v}^{T})G_{u}(1_{2}\otimes\mbf{v})\right]$.
\begin{lemma}\label{lem:det1}
Let $u\in(-1+\omega,1-\omega)$; then
\begin{align}
    \left|\det\left[(1_{2}\otimes\mbf{v}^{T})G_{u}(1_{2}\otimes\mbf{v})\right]\right|&=\left[1+O\left(\frac{\log N}{\sqrt{Nt^{2}}}\right)\right]t^{2}\Tr{H_{u}^{2}}\sigma_{u},\label{eq:det1}
\end{align}
with probability $1-e^{-c\log^{2}N}$.
\end{lemma}
\begin{proof}
In view of the previous lemma we need only estimate $\Tr{(H_{u}M)^{2}}$ for 
\begin{align*}
    M&\in\left\{\eta_{u}H_{u},\eta_{u}\wt{H}_{u},X_{u}H_{u}+H_{u}X^{T}_{u}\right\}.
\end{align*}
Consider $M=X_{u}H_{u}+H_{u}X^{T}_{u}$; then we find
\begin{align*}
    \Tr{(H_{u}M)^{2}}&=2\Tr{(H_{u}X_{u})^{2}}+2\Tr{H_{u}X_{u}H_{u}^{3}X^{T}_{u}}.
\end{align*}
The first term is $O(1)$ by \eqref{eq:C3.1}; the second term is
\begin{align*}
    \Tr{H_{u}X_{u}H_{u}^{3}X^{T}_{u}}&=\Tr{H_{u}(\wt{H}_{u})^{2}\left[1-\eta_{u}^{2}\wt{H}_{u}\right]}\\
    &\leq\frac{1}{\eta_{u}^{2}}\Tr{H_{u}\wt{H}_{u}}\\
    &\leq\frac{C}{t^{4}},
\end{align*}
where the last line follows by \eqref{eq:C3.1}. Hence we have
\begin{align*}
    \mc{K}_{\mu}(u,X,M)&\leq\frac{C}{Nt^{2}}.
\end{align*}
Choosing $r=\frac{C\log N}{\sqrt{Nt^{2}}}$ and $\beta=\sqrt{Nt^{2}}\log N$ in \ref{eq:muConc} we have
\begin{align*}
    \mu_{n}\left(\left\{\left|\mbf{v}^{T}M\mbf{v}-t\Tr{H_{u}M}\right|>\frac{C\log N}{\sqrt{Nt^{2}}}\right\}\right)&\leq e^{-c\log^{2}N},
\end{align*}
for sufficiently small $C$. Repeating these steps for $M\in\{\eta_{u}H_{u},\eta_{u}\wt{H}_{u}\}$, we find
\begin{align*}
    \left|\det\left[(1_{2}\otimes\mbf{v}^{T})G_{u}(1_{2}\otimes\mbf{v})\right]\right|&=\left[t\eta_{u}\Tr{H_{u}^{2}}+O\left(\frac{\log N}{\sqrt{Nt^{3}}}\right)\right]\left[t\eta_{u}\Tr{H_{u}\wt{H}_{u}}+O\left(\frac{\log N}{\sqrt{Nt^{2}}}\right)\right]\\
    &+\left[t\Tr{H_{u}^{2}X_{u}}+O\left(\frac{\log N}{\sqrt{Nt^{2}}}\right)\right]^{2},
\end{align*}
with probability $1-e^{-c\log^{2}N}$. Recalling the definition of $\sigma_{u}$ from \eqref{eq:sigma}, we obtain \eqref{eq:det1} after noting that $\Tr{H_{u}^{2}}\geq Ct^{-3}$ and $\Tr{H_{u}\wt{H}_{u}}\geq Ct^{-2}$.
\end{proof}

Now we come to the second Stiefel manifold $O(n,2)=O(n)/O(n-2)$.
\paragraph{$O(n,2)$:}
Let $z=x+iy$, \[Y=\begin{pmatrix}0&b\\-c&0\end{pmatrix}\] and $Z=x+Y$, where $y=\sqrt{bc}>0$ and $\delta=b-c\geq0$. We recall the definition of the probability measure $\nu_{n}$ on $O(n,2)$:
\begin{align*}
    d\nu_{n}(V;\delta,z,X)&=\frac{{(\psi_{z})^{2}}}{L_{n}(\delta,z,X)}{\left(\frac{N}{2\pi t}\right)^{n-3}}\nonumber\\&\times\exp\left\{-\frac{N}{2t}\tr\left(V^{T}X^{T}XV-2Z^{T}V^{T}XV+Z^{T}Z\right)\right\}\,\mathrm{d}_{H}V,
\end{align*}
where the normalisation is given by
\begin{align*}
    L_{n}(\delta,z,X)&=(\psi_{z})^{2}\left(\frac{N}{2\pi t}\right)^{n-3}\nonumber\\&\times\int_{O(n,2)}\exp\left\{-\frac{N}{2t}\tr\left(V^{T}X^{T}XV-2Z^{T}V^{T}XV+Z^{T}Z\right)\right\}\,\mathrm{d}_{H}V.
\end{align*}
{If $V\in O(n,2)$ has columns $\mbf{v}_{1}$ and $\mbf{v}_{2}$ then we define \[\text{vec}(V)=\begin{pmatrix}\mbf{v}_{1}\\\mbf{v}_{2}\end{pmatrix}\in\mbb{R}^{2n}.\]}

We follow the same program of obtaining asymptotics for the normalisation $L_{n}$ and concentration for quadratic forms. Defining
\begin{align}
    \mc{M}_{z}(\eta)&=\eta^{2}+|1_{2}\otimes X-Z^{T}\otimes1_{n}|^{2},\label{eq:mcM}
\end{align}
and
\begin{align}
    h_{z}(\eta,B)&=\int_{\mbb{M}^{sym}_{2}(\mbb{R})}e^{\frac{iN}{2t}\tr P}\det^{-1/2}\left[1+(iP\otimes1_{n})\left(\mc{M}_{z}(\eta)-\frac{2t}{N}B\right)^{-1}\right]\diff P,\label{eq:hF}
\end{align}
we have the following analogue of Lemma \ref{lem:Kduality}.
\begin{lemma}\label{lem:Lduality}
Let $\eta>0$ and $B\in\mbb{M}^{h}_{2n}(\mbb{C})$ such that $\|(\mc{M}_{z}(\eta))^{-1/2}B(\mc{M}_{z}(\eta))^{-1/2}\|<N/2t$. Then we have
\begin{align}
    \mbb{E}_{\nu}\left[e^{\textup{vec}(V)^{T}B\textup{vec}(V)}\right]&=e^{-\frac{N}{t}\left(\eta_{z}^{2}-\eta^{2}\right)}\cdot\frac{\det\left(\eta_{z}^{2}+|X_{z}|^{2}\right)}{\det^{1/2}\mc{M}_{z}(\eta)}\cdot\frac{h_{z}(\eta,B)}{L_{n}(\delta,z,X)}\nonumber\\&\cdot\det^{-1/2}\left[1-\frac{2t}{N}(\mc{M}_{z}(\eta))^{-1/2}B(\mc{M}_{z}(\eta))^{-1/2}\right].\label{eq:Lduality}
\end{align}
\end{lemma}
\begin{proof}
The exponent in \eqref{eq:nu_n} can be written in the form
\begin{align*}
    \tr\left(V^{T}X^{T}XV-2Z^{T}V^{T}XV+Z^{T}Z\right)&=-\eta^{2}+\text{vec}(V)^{T}\mc{M}_{z}(\eta)\text{vec}(V),
\end{align*}
so that
\begin{align*}
    \mbb{E}_{\nu}\left[e^{\text{vec}(V)^{T}B\text{vec}(V)}\right]&=\frac{e^{\frac{N}{t}\eta^{2}}}{L_{n}}\int_{O(n,2)}\exp\left\{-\frac{N}{2t}\text{vec}(V)^{T}\left(\mc{M}_{z}(\eta)-\frac{2t}{N}B\right)\text{vec}(V)\right\}\,\mathrm{d}_{H}V.
\end{align*}
The norm condition on $B$ ensures that the quadratic form in the above exponent is positive definite so we can apply Lemma \ref{lem:sphericalint} with the function $f:\mbb{M}_{2}(\mbb{R})\mapsto\mbb{R}$ given by
\begin{align*}
    f(M)&=\exp\left\{-\frac{N}{2t}\text{vec}(V)^{T}\left(\mc{M}_{z}(\eta)-\frac{2t}{N}B\right)\text{vec}(V)\right\}.
\end{align*}
\end{proof}
We have left $\eta$ unspecified because {we will choose $\eta$ differently depending} on whether $z$ is near or far from the real axis. In the former case we set $\mc{M}_{z,r}=\mc{M}_{z}(\eta_{x})$ and $h_{z,r}(B)=h_{z}(\eta_{x},B)$.

We have the following asymptotics for the normalisation.
\begin{lemma}\label{lem:Lnear}
Let $z=x+iy\in\mbb{D}_{+,\omega}$ such that $|y|<\frac{C}{\sqrt{N}}$. Then for $\delta>C\|X\|$ we have
\begin{align}
    L_{n}(\delta,z,X)&\leq e^{-\frac{cN}{t}\delta^{2}},\label{eq:Lnear1}
\end{align}
for $\frac{\log N}{\sqrt{N}}<\delta<C\|X\|$ we have
\begin{align}
    L_{n}(\delta,z,X)&\leq e^{-c\log^{2}N},\label{eq:Lnear2}
\end{align}
and for $\delta<\frac{\log N}{\sqrt{N}}$ we have
\begin{align}
    L_{n}(\delta,z,X)&=\left[1+O\left(\frac{\log^{3}N}{\sqrt{Nt^{3}}}\right)\right]\frac{2^{5/2}\pi^{3/2}}{N^{3/2}\Tr{H_{x}^{2}}^{3/2}}\exp\left\{-\frac{N\sigma_{x}}{2}\delta^{2}\right\},\label{eq:Lnear3}
\end{align}
where $\sigma_{x}$ is defined in \eqref{eq:sigma}.
\end{lemma}
\begin{proof}
Let $\delta>C\|X\|$; using the fact that $\mbf{v}_{j}^{T}X\mbf{v}_{j}=\mbf{v}_{j}^{T}X^{T}\mbf{v}_{j}$ and $\mbf{v}_{j}^{T}\mbf{v}_{k}=\delta_{jk}$, we write the exponent in \eqref{eq:nu_n} as follows
\begin{align*}
    \tr\left(V^{T}X^{T}XV-2Z^{T}V^{T}XV+Z^{T}Z\right)&=-\eta_{z}^{2}\\
    &+\text{vec}(V)^{T}\begin{pmatrix}\eta_{z}^{2}+|X_{z}|^{2}&0\\0&\eta_{z}^{2}+|X_{\bar{z}}|^{2}\end{pmatrix}\text{vec}(V)\\
    &+\delta^{2}-\mbf{v}_{1}^{T}\left[\sqrt{\delta^{2}+4y^{2}}(X-X^{T})+\delta(X+X^{T})\right]\mbf{v}_{2}.
\end{align*}
For the term on the second line we use the lower bound
\begin{align*}
    \delta^{2}-\mbf{v}_{1}^{T}\left[\sqrt{\delta^{2}+4y^{2}}(X-X^{T})+\delta(X+X^{T})\right]\mbf{v}_{2}&\geq\left(\delta-C\|X\|\right)\delta,\quad \delta>C|y|.
\end{align*}
By the argument of Lemma \ref{lem:Lduality} we have
\begin{align*}
    L_{n}&\leq e^{-\frac{N}{t}\left(\delta-C\|X\|\right)\delta}\int_{\mbb{M}^{sym}_{2}(\mbb{R})}e^{\frac{iN}{2t}P}\det^{-1/2}\left[1+iP\otimes1_{n}\begin{pmatrix}H_{z}&0\\0&H_{\bar{z}}\end{pmatrix}\right]\diff P\\
    &\leq e^{-\frac{cN}{t}\delta^{2}}\int_{\mbb{M}^{sym}_{2}(\mbb{R})}\det^{-n/4}\left[1+\left(\frac{P}{\eta_{z}^{2}+\|X_{z}\|^{2}}\right)^{2}\right]\diff P\\
    &\leq \frac{C\|X_{u}\|^{6}}{n^{3/2}}e^{-\frac{cN}{t}\delta^{2}}\\
    &\leq e^{-\frac{cN}{t}\delta^{2}}.
\end{align*}
In the second line we have used Lemma \ref{lem:detLower}.

Now let $\frac{\log N}{\sqrt{N}}<\delta<C\|X\|$. Using \eqref{eq:Lduality} for $B=0$ and $\eta=\eta_{z}$ we have
\begin{align}
    L_{n}(\delta,z,X)&=\left(\frac{\psi_{z}}{\psi_{x}}\right)^{2}\cdot\frac{\det\left(\eta_{x}^{2}+|X_{x}|^{2}\right)}{\det^{1/2}\mc{M}_{z,r}}\cdot h_{z,r}(0).\label{eq:Lnear}
\end{align}
Let us estimate each factor in turn. The first factor can be directly estimated by Lemma \ref{lem:psi}:
\begin{align*}
    \left(\frac{\psi_{z}}{\psi_{x}}\right)^{2}&=\left[1+O\left(\frac{1}{\sqrt{Nt^{3}}}\right)\right]\exp\left\{-2Ny^{2}\Tr{(G_{x}\Im F)^{2}}\right\}.
\end{align*}

In the following we assume $b\geq c\geq0$; the case $c\leq b\leq0$ follows by swapping $b$ and $c$. When $b\geq c\geq0$, $\delta>\frac{\log N}{\sqrt{N}}$ and $y<\frac{C}{\sqrt{N}}$, we have $b\geq C\delta$ and $c\leq Cy^{2}/\delta$. To estimate the ratio of determinants, define $\eta_{c}=\sqrt{\eta_{x}^{2}+c^{2}}$ and
\begin{align*}
    M_{1}&=b^{2}\sqrt{H_{x}}\left[1-X_{x}H_{x}(\eta_{c})X^{T}_{x}\right]\sqrt{H_{x}},\\
    M_{2}&=\sqrt{H_{x}}\left[bc\Re(X_{x}H_{x}(\eta_{c})X_{x})-c^{2}X_{x}^{T}H_{x}(\eta_{c})X_{x}\right]\sqrt{H_{x}}.
\end{align*}
Since $bc=y^{2}=O(N^{-1})$, $c=O(y^{2}/\delta)=o(N^{-1/2})$ and $\|X_{x}\sqrt{H_{x}}\|\leq1$, we have $\|M_{2}\|<C(Nt^{2})^{-1}$ when $\delta>\frac{\log N}{\sqrt{N}}$. Therefore we can write
\begin{align*}
    \frac{\det\left(\eta_{x}^{2}+|X_{u}|^{2}\right)}{\det^{1/2}\mc{M}_{z,r}}&=D_{1}D_{2}D_{3},
\end{align*}
where
\begin{align*}
    D_{1}&=\det^{-1/2}(1+c^{2}H_{x}),\\
    D_{2}&=\det^{-1/2}(1+M_{2}),\\
    D_{3}&=\det^{-1/2}(1+(1+M_{2})^{-1/2}M_{1}(1+M_{2})^{-1/2}).
\end{align*}
By the identity $X_{x}H_{x}(\eta)X_{x}^{T}=1-\eta^{2}\wt{H}_{x}(\eta)$, we have
\begin{align*}
    M_{1}&=\eta_{c}^{2}b^{2}\sqrt{H_{x}}\wt{H}_{x}(\eta_{c})\sqrt{H_{x}}>0
\end{align*}
and so using the inequality $\log(1+x)\geq x/(1+x)$ for $x>-1$ we have
\begin{align*}
    D_{3}&\leq\exp\left\{-\frac{(1-\|M_{2}\|)\tr M_{1}}{(1+\|M_{2}\|)(1-\|M_{2}\|+\|M_{1}\|)}\right\}\\
    &\leq\exp\left\{-\frac{CNb^{2}\eta_{x}^{2}\Tr{\wt{H}_{x}H_{x}}}{1+b^{2}/\eta_{x}^{2}}\right\}\\
    &\leq e^{-c\log^{2}N},
\end{align*}
when $\delta>\frac{\log N}{\sqrt{N}}$ since $b\geq C\delta$ and
\begin{align*}
    \Tr{\wt{H}_{x}(\eta_{c})H_{x}}&\geq C\Tr{\wt{H}_{x}H_{x}}\geq\frac{C}{\eta_{x}^{2}}
\end{align*}
by \eqref{eq:C2.1}.

Using $\log(1+x)=x-\int_{0}^{1}sx(1+sx)^{-1}ds$ we have
\begin{align*}
    D_{1}&=\exp\left\{-\frac{Nc^{2}}{2}\Tr{H_{x}}+Nc^{4}\int_{0}^{1}s\Tr{(1+sc^{2}H_{x})^{-1}H_{x}^{2}}ds\right\},\\
    D_{2}&=\exp\left\{-\frac{N}{2}\Tr{M_{2}}+N\int_{0}^{1}s\Tr{(1+sM_{2})^{-1}M_{2}^{2}}ds\right\}.
\end{align*}
By \eqref{eq:C1.2} and \eqref{eq:C3.1} we have
\begin{align*}
    c^{4}\Tr{H_{x}^{2}}&\leq\frac{C}{N^{2}t^{3}},\\
    y^{2}\left|\Re\Tr{H_{x}X_{x}H_{x}(\eta_{c})X_{x}}\right|&\leq\frac{C}{N}.
\end{align*}
We also have
\begin{align*}
    y^{4}\Tr{\left(\sqrt{H_{x}}\Re(X_{x}H_{x}(\eta_{c})X_{x})\sqrt{H_{x}}\right)^{2}}&=\frac{y^{4}}{2}\Tr{(H_{x}X_{x})^{4}}+\frac{y^{4}}{2}\Tr{(H_{x}X_{x})^{2}(H_{x}X^{T}_{x})^{2}}\\
    &\leq\frac{y^{2}}{2}\Tr{(H_{x}X_{x})^{4}}+\frac{y^{2}}{2}\Tr{H_{x}\wt{H}_{x}}\\
    &\leq\frac{C}{N^{2}t^{2}},
\end{align*}
by Cauchy-Schwarz and \eqref{eq:C3.1}, and
\begin{align*}
    y^{2}c^{2}\left|\Re\Tr{H_{x}X_{x}H_{x}(\eta_{c})X_{x}H_{x}X^{T}H_{x}(\eta_{c})X}\right|&\leq\frac{C}{N^{2}t^{5/2}},
\end{align*}
by the identity $X_{x}H_{x}X^{T}=1-\eta_{x}^{2}\wt{H}_{x}$, Cauchy-Schwarz and \eqref{eq:C3.1}.

Using these bounds and the identity
\begin{align*}
    \Tr{H_{x}(\eta_{c})X_{x}^{T}H_{x}X_{x}}&=\Tr{H_{x}}-c^{2}\Tr{(1+c^{2}H_{x})^{-1}H_{x}}-\eta_{x}^{2}\Tr{H_{x}(\eta_{c})\wt{H}_{x}}
\end{align*}
we obtain
\begin{align*}
    \det^{-1}(1+c^{2}H_{x})\det^{-1}(1+M_{2})&\leq\exp\left\{-Nc^{2}\Tr{H_{x}}-N\Tr{M_{2}}+\frac{Nc^{4}}{2}\Tr{H_{x}^{2}}+\frac{N\Tr{M_{2}^{2}}}{2(1-\|M_{2}\|)}\right\}\\
    &\leq e^{C(1+(Nt^{3})^{-1})}.
\end{align*}

For the final factor in \eqref{eq:Lnear}, we use Lemma \ref{lem:detLower}:
\begin{align*}
    h_{z,r}(0)&\leq\int_{\mbb{M}^{sym}_{2}(\mbb{R})}\det^{-n/4}\left(1+\frac{P^{2}}{\|\mc{M}_{z,r}\|^{2}}\right)\diff P\\
    &\leq Cn^{-3/2}\|\mc{M}_{z,r}\|^{3}\\
    &\leq Cn^{-3/2}\|X_{x}\|^{6}.
\end{align*}
These bounds imply \eqref{eq:Lnear2}.

When $\delta<\frac{\log N}{\sqrt{N}}$, we observe that 
\begin{align*}
    \det^{-1/2}\mc{M}_{z,r}&=\det^{-1/2}\begin{pmatrix}{W}_{x}-i\eta_{x}&-Y^{T}\\-Y&{W}_{x}-i\eta_{x}\end{pmatrix},
\end{align*}
from which we obtain
\begin{align*}
    \frac{\det\left(\eta_{x}^{2}+|X_{u}|^{2}\right)}{\det^{1/2}\mc{M}_{z,r}}&=\det^{-1/2}\left(1-G_{x}Y^{T}G_{x}Y\right).
\end{align*}
Since $b,c=O\left(\frac{\log N}{\sqrt{N}}\right)$ and $Nt^{2}\gg1$, we can make a Taylor expansion to second order and bound the remainder by Cauchy-Schwarz and \eqref{eq:C3.1}. We find
\begin{align*}
    \frac{\det\left(\eta_{x}^{2}+|X_{u}|^{2}\right)}{\det^{1/2}\mc{M}_{z,r}}&=\exp\left\{-\frac{1}{2}N(b^{2}+c^{2})\eta_{x}^{2}\Tr{H_{x}\wt{H}_{x}}-Nbc\Tr{(X_{x}H_{x})^{2}}+O\left(\frac{\log^{3}N}{\sqrt{Nt^{3}}}\right)\right\}\\
    &=\exp\left\{-\frac{1}{2}N\eta_{x}^{2}\Tr{H_{x}\wt{H}_{x}}\delta^{2}+2Ny^{2}\Tr{(G_{x}\Im F)^{2}}+O\left(\frac{\log^{3}N}{\sqrt{Nt^{3}}}\right)\right\}.
\end{align*}

Now we evaluate $h_{z,r}(0)$ when $\delta<\frac{\log N}{\sqrt{N}}$. First we truncate $P$ to the region $\|P\|<C\|\mc{M}_{z,r}\|<C\|X\|^{2}$, which we can do since by Lemma \ref{lem:detLower} we have
\begin{align*}
    \left|\det^{-1/2}\left[1+(\mc{M}_{z,r})^{-1/2}(iP\otimes1_{n})(\mc{M}_{z,r})^{-1/2}\right]\right|&\leq\det^{-n/2}\left(1+\frac{P^{2}}{\|\mc{M}_{z,r}\|^{2}}\right),
\end{align*}
so that the integrand is $O(e^{-n})=O(e^{-N})$ when $\|P\|>C\|\mc{M}_{z,r}\|$.

To restrict the integration domain further we use the inequality
\begin{align*}
    &\left|\det^{-1/2}\left[1+(\mc{M}_{z,r})^{-1/2}(iP\otimes1_{n})(\mc{M}_{z,r})^{-1/2}\right]\right|\\&\leq\exp\left\{-\frac{1}{4}\left(1+\frac{\|P\|^{2}}{\eta_{x}^{4}}\right)^{-1}\tr((P\otimes1_{n})(\mc{M}_{z,r})^{-1})^{2}\right\}.
\end{align*}
The trace can be estimated using the fact that
\begin{align*}
    (\mc{M}_{z,r})^{-1}&=\begin{pmatrix}\eta_{x}^{2}+|X_{x}|^{2}+b^{2}&-bX_{x}+cX_{x}^{T}\\-bX^{T}_{x}+cX_{x}&\eta_{x}^{2}+|X_{x}|^{2}+c^{2}\end{pmatrix}^{-1}\\
    &=\left(1_{2}\otimes\sqrt{H_{x}}\right)\left[1+O\left(\frac{\log N}{\sqrt{Nt^{2}}}\right)\right]\left(1_{2}\otimes\sqrt{H_{x}}\right).
\end{align*}
and so by \eqref{eq:C1.2}
\begin{align*}
    \tr((P\otimes1_{n})(\mc{M}_{z,r})^{-1})^{2}&=\left[1+O\left(\frac{\log N}{\sqrt{Nt^{2}}}\right)\right]N\Tr{H_{x}^{2}}\tr P^{2}\\
    &\geq\frac{CN}{t^{3}}\tr P^{2}.
\end{align*}
Therefore we can restrict to the region $|P_{ij}|<\sqrt{\frac{t^{3}}{N}}\log N$. In this region we make a Taylor expansion:
\begin{align*}
    &\det^{-1/2}\left[1+(\mc{M}_{z,r})^{-1/2}(iP\otimes1_{n})(\mc{M}_{z,r})^{-1/2}\right]\\&=\exp\left\{-\frac{i}{2}\tr(P\otimes1_{n})(\mc{M}_{z,r})^{-1}-\frac{1}{4}\tr((P\otimes1_{n})(\mc{M}_{z,r})^{-1})^{2}+O\left(\frac{\log^{3}N}{\sqrt{Nt}}\right)\right]\\
    &=\exp\left\{-\frac{iN}{2t}\tr P-i\delta p_{12}\tr X_{x}H_{x}^{2}-\frac{N}{4}\Tr{H_{x}^{2}}\tr P^{2}+O\left(\frac{\log^{3}N}{\sqrt{Nt^{3}}}\right)\right\}.
\end{align*}
To estimate $\Tr{(P\otimes1_{n})(\mc{M}_{z,r})^{-1}}$ we have used the identity
\begin{align*}
    \Tr{(P\otimes1_{n})(\mc{M}_{z,r})^{-1}}&=\frac{p_{11}}{\eta_{x}}\Im\Tr{(1-G_{x}Y^{T}G_{x}Y)^{-1}G_{x}E_{2}}+\frac{p_{22}}{\eta_{x}}\Im\Tr{(1-G_{x}YG_{x}Y^{T})^{-1}G_{x}E_{2}}\\
    &+\frac{2p_{12}}{\eta_{x}}\Im\Tr{(1-G_{x}Y^{T}G_{x}Y)^{-1}G_{x}Y^{T}G_{x}E_{2}},
\end{align*}
where we recall the definition of $E_{2}$ in \eqref{eq:Pauli}.

Now we can take the error outside the integral, extend to $\mbb{M}^{sym}_{2}(\mbb{R})$ and evaluate the resulting Gaussian integral to obtain
\begin{align*}
    h_{z,r}(0)&=\left[1+O\left(\frac{\log^{3}N}{\sqrt{Nt^{3}}}\right)\right]\frac{2^{5/2}\pi^{3/2}}{N^{3/2}\Tr{H_{x}^{2}}^{3/2}}\exp\left\{-\frac{\Tr{H_{x}^{2}X_{x}}^{2}}{2\Tr{H_{x}^{2}}}\delta^{2}\right\}.
\end{align*}
Combined with the estimates for the other factors in \eqref{eq:Lnear} we obtain \eqref{eq:Lnear3}.
\end{proof}
{In Appendix \ref{sec:uniform} we sketch a simpler proof of \eqref{eq:Lnear2} using an identity of Dubova and Yang \cite{dubova_bulk_2024}.}

For small $\delta$, we have concentration for quadratic forms in $V$.
\begin{lemma}\label{lem:nuConcNear}
Let $z=x+iy\in\mbb{D}_{+,\omega}$ such that $y<\frac{C}{\sqrt{N}}$, $B\in\mbb{M}^{sym}_{2n}(\mbb{R})$ and
\begin{align}
    \beta_{\nu,r}(B)&=\frac{N\eta_{x}^{2}}{4t\|B\|},\label{eq:beta_nuNear}\\
    \mc{K}_{\nu,r}(B)&=\frac{2t^{2}\Tr{\left((\mc{M}_{z,r})^{-1}B\right)^{2}}}{N}.\label{eq:mcK_nuNear}
\end{align}
Then for $\delta<\frac{\log N}{\sqrt{N}}$, $r>0$ and $|\beta|<\beta_{\nu,r}(B)$ we have
\begin{align}
    \nu_{n}\left(\left\{\left|\textup{vec}(V)^{T}B\textup{vec}(V)-t\Tr{(\mc{M}_{z,r})^{-1}B}\right|>r\right\}\right)&\leq C\exp\left\{-\beta r+\mc{K}_{\nu,r}(B)\beta^{2}\right\},\label{eq:nuConcNear}
\end{align}
for sufficiently large $N$.
\end{lemma}
\begin{proof}
Using \eqref{eq:Lduality}, we have for the moment generating function
\begin{align*}
    m(\beta,B)&:=e^{-\beta t\Tr{(\mc{M}_{z,r})^{-1}B}}\mbb{E}_{\nu}\left[e^{\beta\text{vec}(V)^{T}B\text{vec}(V)}\right]\\
    &=\left(\frac{\psi_{z}}{\psi_{x}}\right)^{2}\cdot\frac{h_{z}(\delta,\eta_{x},B)}{L_{n}(\delta,z,X)}\cdot\frac{\det\left(\eta_{x}^{2}+|X_{x}|^{2}\right)}{\det^{1/2}\mc{M}_{z,r}}\\
    &\cdot e^{-\beta t\Tr{(\mc{M}_{z,r})^{-1}B}}\det^{-1/2}\left[1-\frac{2\beta t}{N}(\mc{M}_{z,r})^{-1/2}B(\mc{M}_{z,r})^{-1/2}\right].
\end{align*}
We have already obtained bounds for the first and third factors on the second line in the proof of Lemma \ref{lem:Lnear}; for the second factor we use the fact that for $|\beta|<\beta_{\nu,r}(z,X,B)$ we have
\begin{align*}
    \frac{2|\beta|t}{N}\left\|(\mc{M}_{z,r})^{-1/2}B(\mc{M}_{z,r})^{-1/2}\right\|&<\frac{2|\beta|t\|B\|}{N\eta_{x}^{2}}<\frac{1}{2}
\end{align*}
and argue as in Lemma \ref{lem:muConc}. Thus we have
\begin{align*}
    m(\beta,B)&\leq Ce^{-\beta t\Tr{(\mc{M}_{z,r})^{-1}B}}\det^{-1/2}\left[1-\frac{2\beta t}{N}(\mc{M}_{z,r})^{-1/2}B(\mc{M}_{z,r})^{-1/2}\right]\\
    &\leq\exp\left\{\frac{2\beta^{2}t^{2}}{N}\Tr{((\mc{M}_{z,r})^{-1}B)^{2}}\right\}.
\end{align*}
\end{proof}

Using this lemma we estimate the determinant $\det\left[(1_{2}\otimes V^{T})G_{x}(1_{2}\otimes V)\right]$.
\begin{lemma}\label{lem:det2Near}
Let $\delta<\frac{\log N}{\sqrt{N}}$; then
\begin{align}
    \left|\det\left[(1_{2}\otimes V^{T})G_{x}(1_{2}\otimes V)\right]\right|&=\left[1+O\left(\frac{\log N}{\sqrt{Nt^{3}}}\right)\right]\left(t^{2}\Tr{H_{x}^{2}}\sigma_{x}\right)^{2},\label{eq:det2Near}
\end{align}
with probability $1-e^{-c\log^{2}N}$, where $\sigma_{x}$ was defined in \eqref{eq:sigma}.
\end{lemma}
\begin{proof}

The elements of the determinant are quadratic forms $\mbf{v}_{j}^{T}M\mbf{v}_{k}$ for $j,k=1,2$ and $M\in\{\eta_{x}H_{x},\eta_{x}\wt{H}_{x},X_{x}H_{x}\}$. Let
\begin{align*}
    B&=\begin{pmatrix}M&0\\0&0\end{pmatrix};
\end{align*}
then we have
\begin{align*}
    t\Tr{(\mc{M}_{z,r})^{-1}B}&=\frac{t}{\eta_{x}}\Im\Tr{\left(1-G_{x}Y^{T}G^{T}_{x}Y\right)^{-1}G_{x}\begin{pmatrix}0&0\\0&M\end{pmatrix}}\\
    &=t\Tr{H_{x}M}+\frac{t}{\eta_{x}}\Im\Tr{\frac{G_{x}Y^{T}G_{x}Y}{1-G_{x}Y^{T}G_{x}Y}G_{x}\begin{pmatrix}0&0\\0&M\end{pmatrix}}.
\end{align*}
The second term can be estimated by Cauchy-Schwarz and \eqref{eq:C3.1}:
\begin{align*}
    &\left|t\Tr{(\mc{M}_{z,r})^{-1}B}-t\Tr{H_{x}M}\right|\\
    &\leq\frac{t}{\eta_{x}(1-\|G_{x}\|^{2}\|Y\|^{2})}\Tr{G_{x}Y^{T}G_{x}YY^{*}(G_{x})^{*}\bar{Y}(G_{x})^{*}}^{1/2}\Tr{H_{x}MM^{*}}^{1/2}\\
    &\leq\frac{C\log^{2}N}{Nt}\Tr{H_{x}MM^{*}}^{1/2}.
\end{align*}
With $M=\frac{1}{2}\left(X_{x}H_{x}+H_{x}X^{T}_{x}\right)$ we obtain
\begin{align*}
    \left|t\Tr{(\mc{M}_{z,r})^{-1}B}-t\Tr{H_{x}M}\right|&\leq\frac{C\log^{2}N}{Nt^{9/4}}.
\end{align*}

To estimate $\mc{K}_{\nu,r}$, we note that when $\delta<\frac{\log N}{\sqrt{N}}$ and $y<\frac{C}{\sqrt{N}}$, we have $b,c<\frac{\log N}{\sqrt{N}}$ and so
\begin{align*}
    (\mc{M}_{z,r})^{-1}&=\begin{pmatrix}\eta_{x}^{2}+|X_{x}|^{2}+b^{2}&-bX_{x}+cX_{x}^{T}\\-bX^{T}_{x}+cX_{x}&\eta_{x}^{2}+|X_{x}|^{2}+c^{2}\end{pmatrix}^{-1}\\
    &=\left(1_{2}\otimes\sqrt{H_{x}}\right)\left[1+O\left(\frac{\log N}{\sqrt{Nt^{2}}}\right)\right]\left(1_{2}\otimes\sqrt{H_{x}}\right).
\end{align*}
If $B$ is symmetric we have
\begin{align*}
    \tr((\mc{M}_{z,r})^{-1}B)^{2}&\leq 2\tr(H_{x}M)^{2},
\end{align*}
for sufficiently large $N$. For $M=\frac{1}{2}\left(X_{x}H_{x}+H_{x}X^{T}_{x}\right)$ we obtain
\begin{align*}
    \mc{K}_{\nu,r}(B)&\leq\frac{C}{Nt^{2}}.
\end{align*}
Choosing $r=\frac{\log N}{\sqrt{Nt^{2}}}$ and $\beta=\sqrt{Nt^{2}}\log N$ in Lemma \ref{lem:nuConcNear} we find
\begin{align*}
    \mbf{v}_{1}^{T}H_{x}X_{x}\mbf{v}_{1}&=t\Tr{H_{x}^{2}X_{x}}+O\left(\frac{\log N}{Nt^{9/4}}\right)
\end{align*}
with probability $1-e^{-c\log^{2}N}$.

Now let
\begin{align*}
    B&=\frac{1}{2}\begin{pmatrix}0&M\\M^{T}&0\end{pmatrix};
\end{align*}
then
\begin{align*}
    t\Tr{(\mc{M}_{z,r})^{-1}B}&=\frac{t}{\eta_{x}}\Im\Tr{(1-G_{x}YG_{x}Y^{T})^{-1}G_{x}YG_{x}\begin{pmatrix}0&0\\0&M\end{pmatrix}},
\end{align*}
so that by Cauchy-Schwarz we have
\begin{align*}
    |t\Tr{(\mc{M}_{z,r})^{-1}B}|&\leq\frac{C\log N}{\sqrt{Nt^{2}}}\Tr{|M|^{2}}^{1/2}.
\end{align*}
With $M=\frac{1}{2}\left(X_{x}H_{x}+H_{x}X^{T}_{x}\right)$ we have
\begin{align*}
    |t\Tr{(\mc{M}_{z,r})^{-1}B}|&\leq\frac{C\log N}{\sqrt{Nt^{3}}},
\end{align*}
and
\begin{align*}
    \mc{K}_{\nu,r}(B)&\leq\frac{C}{Nt^{2}}.
\end{align*}
Therefore by Lemma \ref{lem:nuConcNear} we obtain
\begin{align*}
    \left|\mbf{v}_{1}^{T}H_{x}X_{x}\mbf{v}_{2}\right|&\leq\frac{C\log N}{\sqrt{Nt^{3}}},
\end{align*}
with probability $1-e^{-c\log^{2}N}$.

Repeating this for $\mbf{v}_{j}^{T}M\mbf{v}_{k}$ with $j,k=1,2$ and $M\in\{\eta_{x}H_{x},\eta_{x}\wt{H}_{x}\}$ we obtain
\begin{align*}
    \eta_{x}\mbf{v}_{1}^{T}H_{x}\mbf{v}_{1}&=t\eta_{x}\Tr{H_{x}^{2}}+O\left(\frac{\log N}{\sqrt{Nt^{3}}}\right),\\
    \eta_{x}\mbf{v}_{1}^{T}\wt{H}_{x}\mbf{v}_{1}&=t\eta_{x}\Tr{H_{x}\wt{H}_{x}}+O\left(\frac{\log N}{\sqrt{Nt^{2}}}\right),\\
    \eta_{x}\mbf{v}_{1}^{T}H_{x}\mbf{v}_{2}&=O\left(\frac{\log N}{\sqrt{Nt^{3}}}\right),\\
    \eta_{x}\mbf{v}_{1}^{T}\wt{H}_{x}\mbf{v}_{2}&=O\left(\frac{\log N}{\sqrt{Nt^{3}}}\right),
\end{align*}
with probability $1-e^{-c\log^{2}N}$. The same estimates hold after swapping $\mbf{v}_{1}$ and $\mbf{v}_{2}$. Inserting these estimates into $\det\left[(1_{2}\otimes V^{T})G_{x}(1_{2}\otimes V)\right]$ and using the definition of $\sigma_{x}$ in \eqref{eq:sigma} we obtain \eqref{eq:det2Near}.
\end{proof}

Now we prove the corresponding lemmas for the case in which $y>C>0$, starting with the asymptotics of $L_{n}$. In this case we set $\mc{M}_{z,c}=\mc{M}_{z}(\eta_{\delta})$ and $h_{z,c}(B)=h_{z}(\eta_{\delta},B)$, where
\begin{align}
    \eta^{2}_{\delta}&=\frac{4y^{2}\eta_{z}^{2}}{4y^{2}+\delta^{2}}.\label{eq:eta_delta}
\end{align}
\begin{lemma}\label{lem:Lfar}
Let $z\in\mbb{D}_{+,\omega}$. Then for $\delta>C\|X\|$ we have
\begin{align}
    L_{n}(\delta,z,X)&\leq e^{-\frac{cN}{t}\delta^{2}},\label{eq:Lfar1}
\end{align}
for $\frac{\log N}{\sqrt{N}}<\delta<C\|X\|$ and any $D>0$ we have
\begin{align}
    L_{n}(\delta,z,X)&\leq CN^{-D},\label{eq:Lfar2}
\end{align}
and for $\delta<\frac{\log N}{\sqrt{N}}$ we have
\begin{align}
    L_{n}(\delta,z,X)&=\left[1+O\left(\frac{\log^{2}N}{\sqrt{Nt}}\right)\right]\frac{2^{5/2}\pi^{3/2}}{N^{3/2}\Tr{H_{z}^{2}}^{1/2}\Tr{H_{z}H_{\bar{z}}}}\nonumber\\
    &\times\left[\exp\left\{-\frac{N\tau_{z}}{8y^{2}t^{2}\Tr{H_{z}H_{\bar{z}}}}\delta^{2}\right\}+O(N^{-D})\right],\label{eq:Lfar3}
\end{align}
for any $D>0$, where $\tau_{z}$ was defined in \eqref{eq:tau}.
\end{lemma}
{There is a proof of \eqref{eq:Lfar2} using an identity of Dubova and Yang \cite{dubova_bulk_2024} that works under the weaker assumption $\|X\|<e^{c\log^{2}N}$, which we outline in Appendix \ref{sec:uniform}.}
\begin{proof}
The proof of \eqref{eq:Lfar1} is exactly the same as the proof of the \eqref{eq:Lnear1} and so we focus on \eqref{eq:Lfar2} and \eqref{eq:Lfar3}. Using \eqref{eq:Lduality} for $B=0$ and $\eta=\eta_{\delta}$ we have
\begin{align*}
    L_{n}(\delta,z,X)&=\exp\left\{-\frac{N\eta_{z}^{2}\delta^{2}}{t(4y^{2}+\delta^{2})}\right\}\frac{\det\left[\eta_{z}^{2}+|X_{z}|^{2}\right]}{\det^{1/2}\mc{M}_{z,c}}h_{z,c}(0).
\end{align*}

Now we fix $\frac{\log N}{\sqrt{Nt}}<\delta<C\|X\|$. We can write the determinant of $\mc{M}_{z,c}$ as the determinant of a block matrix whose blocks are linear in $X$:
\begin{align*}
    \det^{1/2}\mc{M}_{z,c}&=(-1)^{n}\det^{1/2}\begin{pmatrix}-i\eta_{\delta}&1_{2}\otimes X-Z^{T}\otimes1_{n}\\1_{2}\otimes X^{T}-Z\otimes 1_{n}&-i\eta_{\delta}\end{pmatrix}.
\end{align*}
We want to rewrite the right hand side in terms of the Hermitisation ${W}_{z}$. To this end we diagonalise $Z$ as follows:
\begin{align*}
    Z&=S\begin{pmatrix}z&0\\0&\bar{z}\end{pmatrix}S^{-1},
\end{align*}
where
\begin{align}
    S&=\frac{1}{\sqrt{2}}\begin{pmatrix}(b/c)^{1/4}&(b/c)^{1/4}\\i(c/b)^{1/4}&-i(c/b)^{1/4}\end{pmatrix}.
\end{align}
For the transpose we write $Z^{T}=Z^{*}=(S^{*})^{-1}\diag(\bar{z},z)S^{*}$. Taking factors of $S$ outside the determinant we find
\begin{align*}
    \det^{1/2}\mc{M}_{z,c}&=(-1)^{n}\det^{1/2}\begin{pmatrix}-i\eta_{\delta} S^{*}S\otimes1_{n}&\begin{pmatrix}X_{\bar{z}}&0\\0&X_{z}\end{pmatrix}\\\begin{pmatrix}X^{*}_{\bar{z}}&0\\0&X^{*}_{z}\end{pmatrix}&-i\eta_{\delta}(S^{*}S)^{-1}\otimes1_{n}\end{pmatrix}\\
    &=(-1)^{n}\det^{-1/2}\begin{pmatrix}-iu&-iv&X_{\bar{z}}&0\\
    -iv&-iu&0&X_{z}\\
    X^{*}_{\bar{z}}&0&-iu&iv\\
    0&X^{*}_{z}&iv&-iu\end{pmatrix},
\end{align*}
where $u=\frac{(b+c)}{2\sqrt{bc}}\eta_{\delta}=\eta_{z}$ and
\begin{align}
    v&=\frac{(b-c)\eta_{\delta}}{2\sqrt{bc}}=\frac{\delta\eta_{z}}{\sqrt{4y^{2}+\delta^{2}}}.\label{eq:v}
\end{align}
Permuting the second and third rows and columns we find
\begin{align*}
    \det^{1/2}\mc{M}_{z,c}&=(-1)^{N}\det^{1/2}\begin{pmatrix}{W}_{\bar{z}}-i\eta_{z}&-ivE_{-}\\-ivE_{-}&{W}_{z}-i\eta_{z}\end{pmatrix}\\
    &=\det\left[\eta_{z}^{2}+|X_{z}|^{2}\right]\det^{1/2}\left(1+v^{2}G_{\bar{z}}E_{-}G_{z}E_{-}\right),
\end{align*}
where we recall the definition $E_{-}=\begin{pmatrix}1&0\\0&-1\end{pmatrix}\otimes1_{n}$. Hence we obtain
\begin{align}
    L_{n}(\delta,z,X)&= \exp\left\{-\frac{N\eta_{z}^{2}\delta^{2}}{t(4y^{2}+\delta^{2})}\right\}\det^{-1/2}\left(1+v^{2}G_{z}E_{-}G_{\bar{z}}E_{-}\right)h_{z,c}(0).
\end{align}
To bound the determinant we use the integral representation of the logarithm: with $M=G_{\bar{z}}E_{-}G_{z}E_{-}$ we have
\begin{align*}
    \left|\det^{-1/2}\left(1+v^{2}M\right)\right|&=\exp\left\{-\frac{v^{2}}{2}\Re\tr M+\frac{v^{4}}{2}\int_{0}^{1}s\Re\tr\left(1+sv^{2}M\right)^{-1}M^{2}\diff s\right\}.
\end{align*}
For the first order term we use \eqref{eq:C3.2}:
\begin{align*}
    v|\tr M|&=v|\tr G_{\bar{z}}E_{-}G_{z}E_{-}|\leq \frac{CN\eta_{z}^{2}\delta^{2}}{(\delta^{2}+4y^{2})(y^{2}+t)}.
\end{align*}
For the second order term we use Cauchy-Schwarz and \eqref{eq:C3.2}:
\begin{align*}
    v^{4}\left|\Re\tr\left(1+sv^{2}M\right)^{-1}M^{2}\right|&\leq\frac{v^{4}}{1-v^{2}/\eta_{z}^{2}}\tr|M|^{2}\\
    &\leq\frac{v^{2}/\eta_{z}^{2}}{1-v^{2}/\eta_{z}^{2}}\cdot v^{2}\tr \Im G_{\bar{z}}E_{-}\Im G_{z}E_{-}\\
    &\leq\frac{CN\eta_{z}^{2}\delta^{2}}{y^{2}(\delta^{2}+4y^{2})(y^{2}+t)}.
\end{align*}
It is here where we have used the fact that $\delta<C\|X\|<C$ by \eqref{eq:C0} so that 
\begin{align}
    \frac{v^{2}/\eta_{z}^{2}}{1-v^{2}/\eta_{z}^{2}}&=\frac{\delta^{2}}{4y^{2}}<\frac{C}{y^{2}}.\label{eq:normCondition}
\end{align}
Hence we find
\begin{align*}
    L_{n}(\delta,z,X)&\leq \exp\left\{-\frac{N\eta_{z}^{2}\left(1-Ct/y^{2}\right)}{t(\delta^{2}+4y^{2})}\delta^{2}\right\}h_{z,c}(0)\\
    &\leq e^{-c\log^{2}N},
\end{align*}
since $\|h_{z,c}(0)\|\leq n^{-3/2}\|\mc{M}_{z}(\eta_{\delta})\|^{3}<C\|X\|^{3}$ and $y>C>0$.

The decay in the region $\frac{\log N}{\sqrt{N}}<\delta<\frac{\log N}{\sqrt{Nt}}$ is controlled by the integral over $P$ in $h_{z,c}(0)$. First we observe that we can restrict to $\|P\|<C\|X\|^{2}$. Now we note that
\begin{align*}
    h_{z,c}(0)&=\int_{\|P\|<C\|X\|^{2}}e^{\frac{iN}{2t}\tr P}\det^{-1/2}\left(1+i\sqrt{\mc{Q}}\mc P\sqrt{\mc{Q}}\right)\diff P+O(e^{-N}),
\end{align*}
where
\begin{align*}
    \mc{P}&=\begin{pmatrix}0&0&0&0\\0&p_{1}&0&p_{2}-ip_{3}\\0&0&0&0\\0&p_{2}+ip_{3}&0&p_{1}\end{pmatrix}\otimes1_{n},\\
    \mc{Q}&=\frac{1}{\eta_{\delta}}\Im\begin{pmatrix}{W}_{\bar{z}}-i\eta_{z}&-ivE_{-}\\-ivE_{-}&{W}_{z}-i\eta_{z}\end{pmatrix}^{-1},
\end{align*}
and
\begin{align}
    p_{1}&=\frac{cp_{11}+bp_{22}}{2\sqrt{bc}},\quad p_{2}=\frac{cp_{11}-bp_{22}}{2\sqrt{bc}},\quad p_{3}=p_{12}.
\end{align}
We can take the square root of $\mc{Q}$ since it is conjugate to the positive definite matrix
\begin{align*}
    \frac{1}{\eta_{\delta}}\Im\begin{pmatrix}-i\eta_{\delta}&1_{2}\otimes X-Z^{T}\otimes1_{n}\\1_{2}\otimes X^{T}-Z\otimes1_{n}&-i\eta_{\delta}\end{pmatrix}^{-1}.
\end{align*}
When $\delta<\frac{\log N}{\sqrt{Nt}}$ we have
\begin{align}
    \mc{Q}\geq\left(1-\frac{C\log N}{\sqrt{Nt}}\right)\begin{pmatrix}|G_{\bar{z}}|^{2}&0\\0&|G_{z}|^{2}\end{pmatrix}.
\end{align}
Then we obtain the bound
\begin{align}
    \det^{-1/2}\left(1+\sqrt{\mc{Q}}\mc{P}\sqrt{\mc{Q}}\right)&\leq\exp\left\{-CN\frac{p_{1}^{2}\Tr{H^{2}_{z}}+(p_{2}^{2}+p_{3}^{2})\Tr{H_{z}H_{\bar{z}}}}{1+(p^{2}_{1}+p_{2}^{2}+p_{3}^{2})/\eta^{4}_{z}}\right\}.
\end{align}
Since $\Tr{H^{2}_{z}}>C/\eta_{z}^{3}$ and $\Tr{H_{z}H_{\bar{z}}}\geq C/\eta_{z}^{2}$ by \eqref{eq:C1.2} and \eqref{eq:C2.2} respectively, this bound allows us to restrict to the region $|p_{1}|<\sqrt{\frac{t^{3}}{N}}\log N,\,|p_{2}|,|p_{3}|<\frac{t\log N}{\sqrt{N}}$ (we have used the fact that $Nt^{2}\gg1$). In this region we make a Taylor expansion of the determinant:
\begin{align*}
    e^{\frac{iN}{2t}\tr P}\det^{-1/2}\left(1+\sqrt{\mc{Q}}\mc{P}\sqrt{\mc{Q}}\right)&=\exp\left\{\frac{iN}{2t}\tr P-\frac{i}{2}\tr\mc{P}\mc{Q}-\frac{1}{4}\tr(\mc{P}\mc{Q})^{2}\right.\\&\left.+i\int_{0}^{1}s^{2}\tr(1+is\sqrt{\mc{Q}}\mc{P}\sqrt{\mc{Q}})^{-1}(\sqrt{\mc{Q}}\mc{P}\sqrt{\mc{Q}})^{3}ds\right\}.
\end{align*}
The first order term is
\begin{align*}
    \frac{N}{2t}\tr P-\frac{1}{2}\tr\mc{P}\mc{Q}&=\sum_{i=1}^{3}\mu_{i}p_{i},
\end{align*}
where
\begin{align*}
    \mu_{1}&=\frac{N}{t}\sqrt{1+\frac{\delta^{2}}{4y^{2}}}-\frac{1}{\eta_{\delta}}\Im\tr(1+v^{2}G_{z}E_{-}G_{\bar{z}}E_{-})^{-1}G_{z}E_{2},\\
    \mu_{2}&=\frac{N\delta}{2yt}+\frac{v}{2N\eta_{\delta}}\Re\tr\left(\frac{G_{z}E_{-}G_{\bar{z}}E_{-}}{1+v^{2}G_{z}E_{-}G_{\bar{z}}E_{-}}+\frac{G_{\bar{z}}E_{-}G_{z}E_{-}}{1+v^{2}G_{\bar{z}}E_{-}G_{z}E_{-}}\right)E_{2},\\
    \mu_{3}&=\frac{v}{2N\eta_{\delta}}\Im\tr\left(\frac{G_{z}E_{-}G_{\bar{z}}E_{-}}{1+v^{2}G_{z}E_{-}G_{\bar{z}}E_{-}}-\frac{G_{\bar{z}}E_{-}G_{z}E_{-}}{1+v^{2}G_{\bar{z}}E_{-}G_{z}E_{-}}\right)E_{2}.
\end{align*}
Since $v^{2}/\eta_{z}^{2}=\delta^{2}/(\delta^{2}+4y^{2})<\frac{\log^{2} N}{Nt}$, we have
\begin{align*}
    \mu_{1}&=O\left(\frac{\log^{2} N}{t^{2}}\right),\\
    \mu_{2}&=\left[1+O\left(\frac{\log^{2}N}{Nt}\right)\right]\frac{N\delta}{2yt}\left[1-\frac{t}{2}\Re\Tr{G_{z}E_{-}G_{\bar{z}}+G_{\bar{z}}E_{-}G_{z}}\right],\\
    \mu_{3}&=-\left[1+O\left(\frac{\log^{2}N}{Nt}\right)\right]\frac{N\delta}{4y}\Im\Tr{G_{z}E_{-}G_{\bar{z}}-G_{\bar{z}}E_{-}G_{z}}.
\end{align*}
For the error bound in $\mu_{1}$ we can insert absolute values into the trace since $\Im G_{z}>0$; for $\mu_{2}$ and $\mu_{3}$ we use Cauchy-Schwarz and \eqref{eq:C3.2} to estimate 
\begin{align*}
    v^{2}\left|\Tr{\frac{(G_{z}E_{-}G_{\bar{z}}E_{-})^{2}}{1+v^{2}G_{z}E_{-}G_{\bar{z}}E_{-}}}\right|&\leq\frac{v^{2}/\eta_{z}^{2}}{1-v^{2}/\eta_{z}^{2}}\Tr{\Im (G_{z})E_{-}\Im (G_{\bar{z}})E_{-}}\\
    &\leq\frac{C\log^{2}N}{Nt}.
\end{align*}

We write the second order term as a quadratic form:
\begin{align*}
    \frac{1}{4}\tr(\mc{P}\mc{Q})^{2}&=\frac{1}{2}\sum_{i,j=1}^{3}\Theta_{ij}p_{i}p_{j}.
\end{align*}

By similar estimates we obtain
\begin{align*}
    \theta^{2}_{1}&:=\Theta_{11}=\left[1+O\left(\frac{\log^{2}N}{Nt}\right)\right]N\Tr{H_{z}^{2}},\\
    \theta^{2}_{2}&:=\Theta_{22}=\left[1+O\left(\frac{\log^{2}N}{Nt}\right)\right]N\Tr{H_{z}H_{\bar{z}}},\\
    \theta^{2}_{3}&:=\Theta_{33}=\left[1+O\left(\frac{\log^{2}N}{Nt}\right)\right]N\Tr{H_{z}H_{\bar{z}}},\\
    \Theta_{12},\Theta_{13}&=O\left(\frac{\sqrt{N}\log N}{t^{3/2}}\right),\\
    \Theta_{23}&=O\left(\frac{\log^{2}N}{t}\right).
\end{align*}
We separate the diagonal $\Theta_{0}=\diag(\theta^{2}_{1},\theta^{2}_{2},\theta^{2}_{3})$ and off-diagonal $\Theta_{1}=\Theta-\Theta_{0}$ terms. The latter will be grouped together with the remainder from the Taylor expansion of the determinant. Thus we have
\begin{align*}
    &\exp\left\{-\frac{1}{2}\sum_{i,j=1}^{3}\Theta_{1,ij}p_{i}p_{j}+i\int_{0}^{1}s^{2}\tr(1+is\sqrt{\mc{Q}}\mc{P}\sqrt{\mc{Q}})^{-1}(\sqrt{\mc{Q}}\mc{P}\sqrt{\mc{Q}})^{3}ds\right\}\\
    &=1+\Xi(\{p_{j}\})+O(N^{-D}),
\end{align*}
for any $D$, where $\Xi$ is a polynomial with finite degree. To obtain this we used the fact that $\|\sqrt{\mc{Q}}\mc{P}\sqrt{\mc{Q}}\|\leq\frac{C\log N}{\sqrt{Nt^{2}}}\leq N^{-1/6}$ so we can make the error decay as an arbitrarily high power $D$ of $N$ by truncating the Taylor expansion after a large but finite number $m$ terms depending only on $D$. The norm bound $\|\sqrt{\mc{Q}}\mc{P}\sqrt{\mc{Q}}\|$ also implies that
\begin{align*}
    \wt{\Xi}(\{p_{j}\})&:=\Xi\left(\left\{\frac{p_{j}}{\theta_{j}}\right\}\right)=\sum_{2\leq i+j+k\leq m}\xi_{jkl}p_{1}^{j}p_{2}^{k}p_{3}^{l},\\
    |\xi_{ijk}|&\leq\left(\frac{\log N}{\sqrt{Nt^{2}}}\right)^{(i+j+k)/2}.
\end{align*}

Changing variables from $p_{11},p_{12},p_{22}$ to $p_{1},p_{2},p_{3}$ we have
\begin{align*}
    &\int_{\mbb{M}^{sym}_{2}}e^{\frac{iN}{2t}\tr P}\det^{-1/2}\left(1+i\sqrt{\mc{Q}}\mc{P}\sqrt{\mc{Q}}\right)\diff P\\&=\frac{2}{\prod_{j=1}^{3}\theta_{j}}\left(\int_{\mbb{R}^{3}}e^{-\frac{1}{2}\sum_{j=1}^{3}\left(p_{j}^{2}-2i\mu_{j}p_{j}/\theta_{j}\right)}\left[1+\wt{\Xi}(\{p_{j}\})\right]\diff p_{1}\diff p_{2}\diff p_{3}+O(N^{-D})\right)\\
    &=\frac{2(2\pi)^{3/2}}{\prod_{j=1}^{3}\theta_{j}}\left[1+\wt{\Xi}\left(\{-i\partial_{\mu_{j}/\theta_{j}}\}\right)\right]e^{-\frac{1}{2}\sum_{j=1}^{3}\mu_{j}^{2}/\theta_{j}^{2}}+\frac{O(N^{-D})}{\prod_{j=1}^{3}\theta_{j}}
\end{align*}
The term $\wt{\Xi}(\{\partial_{\mu_{j}/\theta_{j}}\})$ generates a polynomial in $\mu_{j}/\theta_{j}$ of degree $m$. From the estimates of $\mu_{j}$ and $\theta_{j}$ we have
\begin{align*}
    \frac{\mu_{1}}{\theta_{1}}&\leq\frac{C\log^{2}N}{\sqrt{Nt}},\\
    C^{-1}\sqrt{N}\delta\leq\frac{\mu_{2}}{\theta_{2}}&\leq C\sqrt{N}\delta,\\
    \frac{\mu_{3}}{\theta_{3}}&\leq C\sqrt{N}t\delta.
\end{align*}
If $\frac{\log N}{\sqrt{N}}<\delta<\frac{\log N}{\sqrt{Nt}}$, the polynomial generated by $\wt{\Xi}$ is bounded by $N^{C}$ for some finite $C$, while the exponential term is $O(e^{-\log^{2}N})$, which proves \eqref{eq:Lfar2}.

When $\delta<\frac{\log N}{\sqrt{N}}$, we obtain
\begin{align*}
    \int_{\mbb{M}^{sym}_{2}}e^{\frac{iN}{2t}\tr P}\det^{-1/2}\left(1+\sqrt{\mc{Q}}\mc{P}\sqrt{\mc{Q}}\right)\diff P&=\left[1+O\left(\frac{\log ^{2}N}{\sqrt{Nt}}\right)\right]\frac{2(2\pi)^{3/2}}{\prod_{j=1}^{3}\theta_{j}}e^{-\frac{1}{2}\left(\mu_{2}^{2}/\theta_{2}^{2}+\mu_{3}^{2}/\theta_{3}^{2}\right)}.
\end{align*}
We also have
\begin{align*}
    &\exp\left\{\frac{N}{t}\left[\eta^{2}_{\delta}-\eta_{z}^{2}\right]\right\}\det^{-1/2}\left(1+v^{2}G_{z}E_{-}G_{\bar{z}}E_{-}\right)\\
    &=\left[1+O\left(\frac{\log^{2}N}{Nt}\right)\right]\exp\left\{-\frac{N\eta_{z}^{2}\delta^{2}}{4ty^{2}}\left(1+\frac{t}{2}\Tr{G_{z}E_{-}G_{\bar{z}}E_{-}}\right)\right\},
\end{align*}
by a Taylor expansion. Putting everything together we obtain
\begin{align*}
    L_{n}(\delta,z,X)&=\left[1+O\left(\frac{\log^{2}N}{\sqrt{Nt}}\right)\right]\frac{2^{5/2}\pi^{3/2}}{\prod_{j=1}^{3}\theta_{j}}\left[\exp\left\{-\frac{N\tau_{z}}{8y^{2}t^{2}\Tr{H_{z}H_{\bar{z}}}}\delta^{2}\right\}+O(N^{-D})\right],
\end{align*}
from which \eqref{eq:Lfar3} follows.
\end{proof}

The following concentration lemma is proven in the same way as Lemma \ref{lem:nuConcNear}.
\begin{lemma}\label{lem:nuConcFar}
Let $z\in\mbb{D}_{+,\omega}$, $B\in\mbb{M}^{h}_{2n}(\mbb{R})$ and
\begin{align}
    \beta_{\nu,c}(B)&=\frac{N\eta_{z}^{2}}{2t\|B\|},\label{eq:beta_nuFar}\\
    \mc{K}_{\nu,c}(B)&=\frac{2t^{2}\Tr{\left((\mc{M}_{z,c})^{-1}B\right)^{2}}}{N}.\label{eq:mcK_nuFar}
\end{align}
Then for $\delta<\frac{\log N}{\sqrt{N}}$, $r>0$ and $|\beta|<\beta_{\nu,c}(B)$ we have
\begin{align}
    \nu_{n}\left(\left\{\left|\textup{vec}(V)^{T}B\textup{vec}(V)-t\Tr{(\mc{M}_{z,c})^{-1}B}\right|>r\right\}\right)&\leq C\exp\left\{-\beta r+\mc{K}_{\nu,c}(B)\beta^{2}\right\},\label{eq:nuConcFar}
\end{align}
for sufficiently large $N$.
\end{lemma}

From this we obtain a different approximation for $\det\left[(1_{2}\otimes V^{T})G_{z}(1_{2}\otimes V)\right]$ when $y>0$.
\begin{lemma}\label{lem:det2Far}
Let $\delta<\frac{\log N}{\sqrt{N}}$; then
\begin{align}
    \left|\det\left[(1_{2}\otimes V^{T})G_{z}(1_{2}\otimes V)\right]\right|&=\left[1+O\left(\frac{\log N}{\sqrt{Nt^{3}}}\right)\right]\frac{t^{2}\Tr{H_{z}^{2}}\tau_{z}\sigma_{z}}{4y^{2}},\label{eq:det2Far}
\end{align}
with probability $1-e^{-c\log^{2}N}$, where $\tau_{z}$ and $\sigma_{z}$ are defined in \eqref{eq:tau} and \eqref{eq:sigma} respectively.
\end{lemma}
\begin{proof}
Note that
\begin{align}
    (\mc{M}_{z,c})^{-1}&=\left[\mc{S}\frac{1}{\eta_{\delta}}\Im\begin{pmatrix}{W}_{\bar{z}}-i\eta_{z}&-ivE_{-}\\-ivE_{-}&{W}_{z}-i\eta_{z}\end{pmatrix}^{-1}\mc{S}^{*}\right]_{22},\label{eq:identity}
\end{align}
where $22$ denotes the lower right $2n\times2n$ block,
\begin{align*}
    \mc{S}&=\begin{pmatrix}S\otimes1_{n}&0\\0&(S^{*})^{-1}\otimes1_{n}\end{pmatrix}\begin{pmatrix}1_{n}&0&0&0\\0&0&1_{n}&0\\0&1_{n}&0&0\\0&0&0&1_{n}\end{pmatrix},
\end{align*}
and
\begin{align*}
    S&=\frac{1}{\sqrt{2}}\begin{pmatrix}(b/c)^{1/4}&(b/c)^{1/4}\\i(c/b)^{1/4}&-i(c/b)^{1/4}\end{pmatrix}.
\end{align*}
Let $M\in\mbb{M}_{n}$, \[B=\begin{pmatrix}0&0\\0&M\end{pmatrix},\] $G_{z,w}=G_{z}E_{-}G_{w}E_{-}$ and recall the definition of $v$ in \eqref{eq:v}. We have the following identities:
\begin{align}
    t\Tr{(\mc{M}_{z,c})^{-1}\begin{pmatrix}M&0\\0&0\end{pmatrix}}&=\frac{t}{2\eta_{\delta}}\sqrt{\frac{c}{b}}\Im\Tr{\left[(1+v^{2}G_{\bar{z},z})^{-1}G_{\bar{z}}+(1+v^{2}G_{z,\bar{z}})^{-1}G_{z}\right]B}\nonumber\\
    &-\frac{itv}{2\eta_{\delta}}\sqrt{\frac{c}{b}}\Im\Tr{\left[(1+v^{2}G_{\bar{z},z})^{-1}G_{\bar{z},z}+(1+v^{2}G_{z,\bar{z}})^{-1}G_{z,\bar{z}}\right]B};\label{eq:F1}
\end{align}
\begin{align}
    t\Tr{(\mc{M}_{z,c})^{-1}\begin{pmatrix}0&0\\0&M\end{pmatrix}}&=\frac{t}{2\eta_{\delta}}\sqrt{\frac{b}{c}}\Im\Tr{\left[(1+v^{2}G_{\bar{z},z})^{-1}G_{\bar{z}}+(1+v^{2}G_{z,\bar{z}})^{-1}G_{z}\right]B}\nonumber\\
    &+\frac{itv}{2\eta_{\delta}}\sqrt{\frac{b}{c}}\Im\Tr{\left[(1+v^{2}G_{\bar{z},z})^{-1}G_{\bar{z},z}+(1+v^{2}G_{z,\bar{z}})^{-1}G_{z,\bar{z}}\right]B};\label{eq:F2}
\end{align}
\begin{align}
    t\Tr{(\mc{M}_{z,c})^{-1}\begin{pmatrix}0&M\\M^{*}&0\end{pmatrix}}&=-\frac{t}{\eta_{\delta}}\Im\Tr{\left[(1+v^{2}G_{\bar{z},z})^{-1}G_{\bar{z}}-(1+v^{2}G_{z,\bar{z}})^{-1}G_{z}\right]B}\nonumber\\
    &+\frac{tv}{\eta_{\delta}}\Im\Tr{\left[(1+v^{2}G_{\bar{z},z})^{-1}G_{\bar{z},z}-(1+v^{2}G_{z,\bar{z}})^{-1}G_{z,\bar{z}}\right]\Re B};\label{eq:F3}
\end{align}
\begin{align}
    t\Tr{(\mc{M}_{z,c})^{-1}\begin{pmatrix}0&-iM\\iM^{*}&0\end{pmatrix}}&=\frac{t}{\eta_{\delta}}\Im\Tr{\left[(1+v^{2}G_{\bar{z},z})^{-1}G_{\bar{z}}-(1+v^{2}G_{z,\bar{z}})^{-1}G_{z}\right]\Re B}\nonumber\\
    &+\frac{tv}{\eta_{\delta}}\Im\Tr{\left[(1+v^{2}G_{\bar{z},z})^{-1}G_{\bar{z},z}-(1+v^{2}G_{z,\bar{z}})^{-1}G_{z,\bar{z}}\right]B}.\label{eq:F4}
\end{align}
These identities follow from \eqref{eq:identity}. Since $\delta<\frac{\log N}{\sqrt{N}}$, we have 
\begin{align}
    v^{2}\|G_{z,\bar{z}}\|&\leq \frac{C\log^{2}N}{N},\label{eq:vG}
\end{align}
and can also use the identity in \eqref{eq:identity} to deduce that
\begin{align*}
    \mc{K}_{\nu,c}\left(\begin{pmatrix}M&0\\0&0\end{pmatrix}\right)&=\frac{2t^{2}}{N}\Tr{\left((\mc{M}_{z,c})^{-1}\begin{pmatrix}M&0\\0&0\end{pmatrix}\right)^{2}}\\
    &\leq \frac{Ct^{2}}{N}\Tr{(H_{z}\Re M)^{2}+(H_{\bar{z}}\Re M)^{2}+2H_{z}\Im M H_{\bar{z}}\Im M}.
\end{align*}

We can now proceed as in Lemma \ref{lem:det2Near} by bounding the error by Cauchy-Schwarz {using \eqref{eq:vG} and the fact that 
\begin{align*}
    b,c&=y+O\left(\frac{\log N}{\sqrt{N}}\right),\\
    \eta_{\delta}&=\left[1+O\left(\frac{\log N}{\sqrt{N}}\right)\right]\eta_{z},
\end{align*}
when $\delta<\frac{\log N}{\sqrt{N}}$ and $y>\omega>0$. Consider for example a term in the first line of the right hand side of \eqref{eq:F1}:
\begin{align*}
    \left|\Tr{(1+v^{2}G_{z,\bar{z}})^{-1}G_{z}B}-\Tr{G_{z}B}\right|&=v^{2}\left|\Tr{(1+v^{2}G_{z,\bar{z}})^{-1}G_{z}E_{-}G_{\bar{z}}E_{-}G_{z}B}\right|\\
    &\leq \frac{v^{2}}{1-v^{2}/\eta_{z}^{2}}\Tr{|G_{z}|^{2}E_{-}|G_{\bar{z}}|^{2}E_{-}}^{1/2}\Tr{|G_{z}|^{2}|B|^{2}}^{1/2}\\
    &\leq\frac{v^{2}/\eta_{z}}{1-v^{2}/\eta^{2}}\Tr{\Im G_{z}E_{-}\Im G_{\bar{z}}E_{-}}^{1/2}\Tr{H_{z}MM^{*}}^{1/2}.
\end{align*}
By \eqref{eq:C3.1} and \eqref{eq:C3.2} we have
\begin{align*}
    \Tr{\Im G_{z}E_{-}\Im G_{\bar{z}}E_{-}}&\leq C,
\end{align*}
(this uses the identity $G_{z}E_{-}=-E_{-}G^{*}_{z}$) and so
\begin{align*}
    \left|\Tr{(1+v^{2}G_{z,\bar{z}})^{-1}G_{z}B}-\Tr{G_{z}B}\right|&\leq\frac{Ct\log^{2}N}{N}\Tr{H_{z}MM^{*}}^{1/2}.
\end{align*}
Similarly, for a term on the second line of the right hand side of \eqref{eq:F1} we have
\begin{align*}
    \left|v\Tr{(1+v^{2}G_{z,\bar{z}})^{-1}G_{z,\bar{z}}B}\right|&\leq\frac{v}{1-v^{2}/\eta_{z}^{2}}\Tr{\frac{\Im G_{z}}{\eta_{z}}}^{1/2}\Tr{|G_{z}|^{2}|B|^{2}}^{1/2}\\
    &\leq C\sqrt{\frac{t}{N}}\log N\Tr{H_{z}MM^{*}}^{1/2},
\end{align*}
by \eqref{eq:C1.1}. Thus with $M=X_{z}H_{z}+H_{z}X_{z}^{*}$ we have
\begin{align*}
    t\Tr{(\mc{M}_{z,c})^{-1}\begin{pmatrix}M&0\\0&0\end{pmatrix}}&=\frac{\eta_{z}t}{2\eta_{\delta}}\sqrt{\frac{c}{b}}\Tr{(H_{z}+H_{\bar{z}})M}+O\left(\frac{\log N}{\sqrt{Nt}}\right)\\
    &=\left[1+O\left(\frac{\log N}{\sqrt{N}}\right)\right]\frac{t\Tr{(H_{z}+H_{\bar{z}})M}}{2}+O\left(\frac{\log N}{\sqrt{Nt}}\right),
\end{align*}
and
\begin{align*}
    \mc{K}_{z,c}\left(\begin{pmatrix}M&0\\0&0\end{pmatrix}\right)&\leq\frac{C}{Nt^{2}}.
\end{align*}
Choosing $r=\frac{\log N}{\sqrt{Nt^{2}}}$ and $\beta=\sqrt{Nt^{2}}\log N$ in Lemma \ref{lem:nuConcFar}, we obtain
\begin{align*}
    \mbf{v}_{1}^{T}M\mbf{v}_{1}&=\frac{t\Tr{(H_{z}+H_{\bar{z}})M}}{2}+O\left(\frac{\log N}{\sqrt{Nt^{2}}}\right),
\end{align*}
with probability $1-e^{-c\log^{2}N}$.}

Repeating these arguments for the remaining terms and $M\in\{\eta_{z}H_{z},\eta_{z}\wt{H}_{z},X_{z}H_{z}\}$, we find
\begin{align*}
    \mbf{v}_{1}^{T}M\mbf{v}_{1}&=\frac{t\Tr{(H_{z}+H_{\bar{z}})M}}{2}+O\left(\frac{\log N}{\sqrt{Nt^{3}}}\right),\\
    \mbf{v}_{1}^{T}M\mbf{v}_{2}&=-\frac{it\Tr{(H_{z}-H_{\bar{z}})M}}{2}+O\left(\frac{\log N}{\sqrt{Nt^{3}}}\right),
\end{align*}
with probability $1-e^{-c\log^{2}N}$. The bounds for $\mbf{v}_{2}^{T}M\mbf{v}_{j}$ can be obtained from the above and the identity $\mbf{v}_{2}^{T}M\mbf{v}_{j}=\overline{\mbf{v}_{j}^{T}M^{*}\mbf{v}_{2}}$. Inserting these into the determinant we obtain \eqref{eq:det2Far} after noting the identity {(see Appendix \ref{sec:trace} for a proof)}
\begin{align*}
    \tau_{z}&=4y^{2}t^{2}\left[\eta_{z}^{2}\Tr{H_{\bar{z}}H_{z}}\Tr{H_{\bar{z}}\wt{H}_{z}}+\left|\Tr{H_{\bar{z}}X_{z}H_{z}}\right|^{2}\right].
\end{align*}
\end{proof}

\section{Proof of Theorem \ref{thm2}}\label{sec:thm2proof}
Let $\gamma\in(0,1/3)$, $\omega\in(0,1)$ and $X=\frac{1}{\sqrt{N}}(\xi_{jk})_{j,k=1}^{N}$ such that $\xi_{jk}$ are i.i.d. real random variables with zero mean, unit variance and finite moments. In the following we abbreviate $G_{z}(w):=G^{X}_{z}(w)$. {We start by showing that $X\in\mc{X}_{N}(\gamma,\omega)$ with very high probability. To do so we first need to review some facts about the deterministic approximation to resolvents.}

Let $m_{z}(w)$ be the unique solution to
\begin{align}
    {-}\frac{1}{m_{z}({w})}&=m_{z}({w})+{w}-\frac{|z|^{2}}{m_{z}({w})+{w}}
\end{align}
such that ${\Im w\Im m_{z}(w)}>0$, and define
\begin{align}
    u_{z}(w)&:=\frac{m_{z}(w)}{m_{z}(w)+w}.
\end{align}
{The deterministic approximation of a single resolvent is given by}
\begin{align}
    M_{z}(w)&=\begin{pmatrix}m_{z}(w)&-zu_{z}(w)\\-\bar{z}u_{z}(w)&m_{z}(w)\end{pmatrix},
\end{align}
and we have the following single resolvent local law.
{\begin{proposition}[Theorem 5.2 in \cite{alt_local_2018}]\label{prop:single}
Fix $z\in\mbb{C}$ and $\epsilon>0$ and let $\eta>N^{-1+\epsilon}$. Then for any $\xi,D>0$ we have
\begin{align}
    \Tr{G_{z}(i\eta)-M_{z}(i\eta)}&\leq\frac{N^{\xi}}{N\eta},\\
    \max_{\mu,\nu\in[1,...,2N]}|(G_{z}(i\eta)-M_{z}(i\eta))_{\mu,\nu}|&\leq\frac{N^{\xi}}{\sqrt{N\eta}}.
\end{align}
\end{proposition}
The local law in \cite[Theorem 5.2]{alt_local_2018} is more general than the above statement since it allows for non-identical variances but we do not need this generality.}

Let $S:\mbb{M}_{2N}\to\mbb{M}_{2N}$ and $\mc{B}_{z_{1},z_{2}}(w_{1},w_{2},A):\mbb{M}_{2N}\to\mbb{M}_{2N}$ be defined by
\begin{align}
    S\left[\begin{pmatrix}A&B\\C&D\end{pmatrix}\right]&=\begin{pmatrix}\Tr{D}&0\\0&\Tr{A}\end{pmatrix},\\
    \mc{B}_{z_{1},z_{2}}(w_{1},w_{2})[F]&=1-M_{z_{1}}(w_{1})S\left[F\right]M_{z_{2}}(w_{2}).
\end{align}
The deterministic approximation to $G_{z_{1}}(w_{1})BG_{z_{2}}(w_{2})$ is given by 
\begin{align}
    M_{z_{1},z_{2}}(w_{1},w_{2},F)&=(\mc{B}_{z_{1},z_{2}}(w_{1},w_{2}))^{-1}\left[M_{z_{1}}(w_{1})BM_{z_{2}}(w_{2})\right],
\end{align}
and we have the following two-resolvent local law.
\begin{proposition}[\cite{cipolloni_central_2023}, Theorem 5.2 and {\cite{cipolloni_maximum_2024}, Corollary B.4}]\label{prop:double}
Let $\epsilon>0$, $z_{1},z_{2}\in\mbb{C}$ and $\eta_{1},\eta_{2}\in\mbb{R}$ such that $\eta_{*}=\min(|\eta_{1}|,|\eta_{2}|)>N^{-1+10\epsilon}$. For any $D>0$ and $B_{1},B_{2}\in\mbb{M}_{2N}$ we have
\begin{align}
    \left|\Tr{\left[G_{z_{1}}(i\eta_{1})B_{1}G_{z_{1}}(i\eta_{2})-M_{z_{1},z_{2}}(i\eta_{1},i\eta_{2},B_{1})\right]B_{2}}\right|&\leq\frac{N^{2\epsilon}\|B_{1}\|\cdot\|B_{2}\|}{N\eta_{*}^{2}},
\end{align}
with probability $1-N^{-D}$ for sufficiently large $N>N(D,\epsilon)$.
\end{proposition}
We have combined the statement of \cite[Theorem 5.2]{cipolloni_central_2023}, which is valid for any $z_{1},z_{2}\in\mbb{C}$ but contains an extra factor of $\|\mc{B}_{z_{1},z_{2}}^{-1}\|\simeq(|z-w|^{2}+\eta_{1}+\eta_{2})^{-1}$ in the error term, with the statement of {\cite[Corollary B.4]{cipolloni_maximum_2024}}, which is valid for $|z-w|<N^{-\epsilon}$ and does not contain this extra factor. Note that the error bound would allow $\eta_{*}>N^{-1/2+\epsilon}$ but we only require $\eta_{*}>N^{-\gamma}>N^{-1/3+\epsilon}$. {We also observe that \cite[Theorem 5.2]{cipolloni_central_2023} is enough for our purposes if we restrict to $\eta_{*}>N^{-1/4+\epsilon}$ (note that $\eta_{*}>N^{-\epsilon}$ for any $\epsilon>0$ is sufficient for the proof of Theorem \ref{thm2}).}

{Let $m:=m_{z}(i\eta)$, $m_{1}:=m_{z_{1}}(i\eta_{1})$ and $m_{2}:=m_{z_{2}}(i\eta_{2})$.} Following \cite[Lemma 6.1]{cipolloni_central_2023}, if we make the identification
\begin{align*}
    \begin{pmatrix}a&b\\c&d\end{pmatrix}&\mapsto(a\,d\,b\,c)^{T},
\end{align*}
the inverse of $\mc{B}_{z_{1},z_{2}}(i\eta_{1},i\eta_{2})$ has the matrix representation
\begin{align*}
    (\mc{B}_{z_{1},z_{2}}(i\eta_{1},i\eta_{2}))^{-1}&=\begin{pmatrix}T_{1}^{-1}&0\\-T_{2}T_{1}^{-1}&0\end{pmatrix},
\end{align*}
where
\begin{align*}
    T_{1}&=\begin{pmatrix}1-z_{1}\bar{z}_{2}u_{1}u_{2}&-m_{1}m_{2}\\-m_{1}m_{2}&1-\bar{z}_{2}z_{1}u_{1}u_{2}\end{pmatrix},\\
    T_{2}&=\begin{pmatrix}z_{1}m_{2}u_{1}&z_{2}m_{1}u_{2}\\\bar{z}_{2}m_{1}u_{2}&\bar{z}_{1}m_{2}u_{1}\end{pmatrix}.
\end{align*}
From this it follows that for any {$B_{j}\in\{E_{+},E_{-},F,F^{*}\}$} we have
\begin{align}
    \Tr{G_{z}(i\eta_{1})B_{1}G_{w}(i\eta_{2})B_{2}}&=\frac{P[m_{1},m_{2},u_{1},u_{2}]}{\beta_{z_{1},z_{2}}(i\eta_{1},i\eta_{2})}+O\left(\frac{N^{2\epsilon}}{N\eta^{2}_{*}}\right),\label{eq:detApprox}
\end{align}
for some polynomial $P$, where
\begin{align*}
    \beta_{z_{1},z_{2}}(i\eta_{1},i\eta_{2})&=[1-z_{1}\bar{z}_{2}u_{1}u_{2}][1-\bar{z}_{1}z_{2}u_{1}u_{2}]+m^{2}_{1}m^{2}_{2}\\
    &=u_{1}u_{2}|z_{1}-z_{2}|^{2}+(1-u_{1})(1-u_{2})-i\eta_{1} m_{1}u_{2}-i\eta_{2} m_{2}u_{1}.
\end{align*}
The second equality is eq. (6.6) in \cite{cipolloni_central_2023}. {The asymptotics of $\Im m$ are given by
\begin{proposition}[Proposition 3.2 in \cite{alt_local_2018}]
Let $m:=m_{z}(\eta)$; then
\begin{align}
    \Im m&\simeq\begin{cases}|\eta|^{-1}&\quad|\eta|>1\quad |z|<C\\
    |\eta|^{1/3}+\sqrt{1-|z|^{2}}&\quad |\eta|\leq 1\quad |z|<1\\
    \frac{|\eta|}{|z|^{2}-1+|\eta|^{2/3}}&\quad |\eta|\leq1\quad1\leq|z|<C
    \end{cases},\label{eq:mAsymp}
\end{align}
where $f\simeq g$ means $C^{-1}g<f<Cg$.
\end{proposition}}
When $|z|>C$ for sufficiently large $C$, we have
\begin{align}
    \Im m&\simeq \frac{\eta}{|z|^{2}+\eta^{2}},
\end{align}
which follows from the cubic equation. {Indeed, with $\rho=\Im m$ we have
\begin{align*}
    \rho^{3}+2\eta^{2}\rho^{2}+(|z|^{2}-1+\eta^{2})\rho&=\eta.
\end{align*}
For $|z|^{2}+\eta^{2}-1$ large enough, we can find a $c>0$ and $C>0$ such that the left hand side is strictly less than $\eta$ when $\rho<c\eta/(|z|^{2}+\eta^{2})$ and strictly greater than $\eta$ when $\rho>C\eta/(|z|^{2}+\eta^{2})$.} We also have for the derivative with respect to $\eta$
\begin{align}
    m'&=\frac{m^{2}+|z|^{2}u^{2}}{1-m^{2}-|z|^{2}u^{2}}.
\end{align}
When $z\in\mbb{D}_{\omega}$, the denominator is greater than $1-|z|^{2}>\omega$ and so $|m'|<C$ uniformly in $\eta\geq0$.

\begin{lemma}\label{lem:AinX}
{Let $\epsilon>0$ and $N^{-1/3+\epsilon}\leq t\leq N^{-\epsilon}$.} For any $D>0$ we have
\begin{align}
    X\in\mc{X}_{N}(\gamma,\omega),
\end{align}
{and
\begin{align}
    \sigma^{X}_{z}&=1+O(t)
\end{align}}
with probability $1-N^{-D}$ for sufficiently large $N>N_{0}(D)$.
\end{lemma}
\begin{proof}
The fact that for such $X$ we have $\|X\|<C$ with high probability follows from the local law for sample covariance matrices in \cite[Theorem 3.7]{knowles_anisotropic_2017}, or alternatively Theorem 5.9 in the book by Bai and Silverstein \cite{bai_spectral_2010}. The condition in \eqref{eq:C1.1} follows from the single resolvent local law in Proposition \ref{prop:single} and the asymptotics in \eqref{eq:mAsymp}:
{\begin{align*}
    \eta\Tr{H_{z}(\eta)}&=\frac{1}{2}\Tr{\Im G_{z}(i\eta)}\simeq\Im m_{z}(i\eta)\simeq\eta^{1/3}+\sqrt{1-|z|^{2}}\simeq1.
\end{align*}}
The remaining conditions follow from the two-resolvent local laws. 

Fix $z_{1},z_{2}\in\mbb{D}_{\omega}$ and $N^{-{1/2}+2\epsilon}\leq|\eta_{1}|,|\eta_{2}|<10$. Note that for such $z_{j}$ and $\eta_{j}$ we have $C^{-1}<\Im m_{j}<C$. Since $m_{j}=i\Im m_{j}$ on the imaginary axis, we introduce $\rho_{j}:=\Im m_{j}$. By eq. (B.3) in \cite{cipolloni_mesoscopic_2023} we have
\begin{align}
    \beta_{z_{1},z_{2}}(i\eta_{1},i\eta_{2})&\simeq\left(|z-w|^{2}+\eta_{1}+\eta_{2}\right).\label{eq:beta}
\end{align}
{Combining \eqref{eq:beta} and \eqref{eq:detApprox} we obtain}
\begin{align*}
    \Tr{G_{z}(i\eta)B_{1}G_{\bar{z}}(i\eta)B_{2}}&\leq\frac{C}{(\Im z)^{2}+\eta},
\end{align*}
for any $B_{j}\in\{E_{+},E_{-},F,F^{*}\}$ (recall the definitions of $F$ and $E_{\pm}$ in \eqref{eq:Pauli} and \eqref{eq:Epm}).

Consider \eqref{eq:C3.1} with $B_{1}=B_{2}=1_{2N}$, for which we obtain
\begin{align*}
    \Tr{G_{z_{1}}(i\eta_{1})G_{z_{2}}(i\eta_{2})}&=\frac{2\left[1-\rho_{1}\rho_{2}-\Re(z_{1}\bar{z}_{2})u_{1}u_{2}\right]}{\beta_{z_{1},z_{2}}(i\eta_{1},i\eta_{2})}-2+O\left(\frac{N^{2\epsilon}}{N\eta^{2}_{*}}\right).
\end{align*}
When $\eta_{1},\eta_{2}\downarrow0$, the numerator becomes
\begin{align*}
    1-\rho_{1}\rho_{2}-\Re(z_{1}\bar{z}_{2})u_{1}u_{2}&\to1-\Re(z_{1}\bar{z}_{2})-\sqrt{(1-|z|^{2})(1-|w|^{2})}\\
    &=\frac{|z_{1}-z_{2}|^{2}}{2}+\frac{(|z_{1}|^{2}-|z_{2}|^{2})^{2}}{2}\\
    &\leq C|z_{1}-z_{2}|^{2}.
\end{align*}
Thus by Taylor expansion in $\eta_{1}$ and $\eta_{2}$ (using the uniform bound on $m'_{j}$) and the lower bound on $\beta$ in \eqref{eq:beta} we have
\begin{align*}
    \frac{2\left[1-\rho_{1}\rho_{2}-\Re(z_{1}\bar{z}_{2})u_{1}u_{2}\right]}{\beta_{z_{1},z_{2}}(i\eta_{1},i\eta_{2})}&\leq C.
\end{align*}
{This also implies \eqref{eq:C1.2}, since
\begin{align*}
    \eta^{3}\Tr{H_{z}^{2}(\eta)}&=\frac{1}{2}\left(\Im\Tr{G_{z}(i\eta)}-\eta\Tr{G^{2}_{z}(i\eta)}\right).
\end{align*}}

For $B_{1}=1_{2N}$ and $B_{2}=F^{*}=\begin{pmatrix}0&0\\1_{N}&0\end{pmatrix}$ we have
\begin{align*}
    \Tr{G_{z_{1}}(i\eta_{1})G_{z_{2}}(i\eta_{2})F^{*}}&=\frac{u_{1}u_{2}\left(\rho_{1}z_{1}+\rho_{2}z_{2}\right)-\rho_{1}u_{2}z_{2}-\rho_{2}u_{1}z_{1}}{\beta_{z_{1},z_{2}}(i\eta_{1},i\eta_{2})}\\&+O\left(\frac{N^{2\epsilon}}{N\eta^{2}_{*}}\right).
\end{align*}
For $\eta_{1},\eta_{2}\downarrow0$ the numerator becomes
\begin{align*}
    (z_{1}-z_{2})\left(\sqrt{1-|z_{1}|^{2}}-\sqrt{1-|z_{2}|^{2}}\right)&\leq C|z_{1}-z_{2}|^{2},
\end{align*}
so that the ratio is again bounded above.

Now consider \eqref{eq:C3.1} with $B_{1}=B_{2}=F=\begin{pmatrix}0&1_{N}\\0&0\end{pmatrix}$, for which
\begin{align*}
    \Tr{G_{z_{1}}(i\eta_{1})FG_{z_{2}}(i\eta_{2})F}&=\frac{u_{1}u_{2}(u_{1}\bar{z}_{1}-\bar{z}_{2})(u_{2}\bar{z}_{2}-\bar{z}_{1})}{\beta_{z_{1},z_{2}}(i\eta_{1},i\eta_{2})}+O\left(\frac{N^{2\epsilon}}{N\eta^{2}_{*}}\right).
\end{align*}
Now the numerator tends to $(\bar{z}_{1}-\bar{z}_{2})^{2}$ and so the ratio is also bounded above. The same is true for $B_{1}=B_{2}^{*}=F$. Note that here we can take $\eta_{1},\eta_{2}\to0$ from above or below, whereas before it was important that $\eta_{1}$ and $\eta_{2}$ have the same sign.

For \eqref{eq:C2.2} we have
\begin{align*}
    \eta_{1}\eta_{2}\Tr{H_{z_{1}}(\eta_{1})H_{z_{2}}(\eta_{2})}&=\frac{\rho_{1}\rho_{2}}{\beta_{z_{1},z_{2}}(\eta_{1},\eta_{2})}+O\left(\frac{N^{2\epsilon}}{N\eta^{2}_{*}}\right).
\end{align*}
When $w=\bar{z}$ and $\eta_{2}=\eta_{1}=\eta$, we have $\rho_{2}=\rho_{1}\geq C$ and so
\begin{align*}
    \eta^{2}\Tr{H_{z_{1}}(\eta)H_{\bar{z}}(\eta)}&\geq \frac{C}{(\Im z)^{2}+\eta}.
\end{align*}

Finally, for \eqref{eq:C2.1} we have
\begin{align*}
    \eta_{1}\eta_{2}\Tr{\wt{H}_{z_{1}}(\eta_{1})H_{z_{2}}(\eta_{2})}&=\frac{\rho_{1}\rho_{2}(1-u_{1}u_{2})}{\beta_{z_{1},z_{2}}(\eta_{1},\eta_{2})}+O\left(\frac{N^{2\epsilon}}{N\eta^{2}_{*}}\right).
\end{align*}
We now need to consider arbitrary $z_{2}\in\mbb{C}$ and $N^{-\gamma}<\eta_{2}<10\|X\|$. When $\eta_{2}>1$ and $|z_{2}|<C$ we have $\rho_{2}\simeq \eta_{2}^{-1}$, which implies $u_{2}\simeq\eta_{2}^{-2}$ and $\beta_{z_{1},z_{2}}(i\eta_{1},i\eta_{2})<C$. Therefore we have
\begin{align*}
    \frac{\rho_{1}\rho_{2}[1-u_{1}u_{2}]}{\beta_{z_{1},z_{2}}(\eta_{1},\eta_{2})}&\geq\frac{C}{\eta_{2}}\\
    &\geq\frac{C\eta_{1}\wedge\eta_{2}}{|z_{1}-z_{2}|^{2}+\eta_{1}\vee\eta_{2}}.
\end{align*}
If $\eta_{2}<1$ and $|z_{2}|<C$, then we have $\eta_{2}<C\rho_{2}$, $1-u_{2}>c\eta_{2}/\rho_{2}$ and $\beta_{z_{1},z_{2}}(i\eta_{1},i\eta_{2})<C(|z_{1}-z_{2}|^{2}+\eta_{1}+\eta_{2})$, from which we obtain
\begin{align*}
    \frac{\rho_{1}\rho_{2}[1-u_{1}u_{2}]}{\beta_{z_{1},z_{2}}(\eta_{1},\eta_{2})}&>\frac{\rho_{1}\rho_{2}[1-u_{2}]}{\beta_{z_{1},z_{2}}(\eta_{1},\eta_{2})}\\
    &>\frac{C\eta_{1}\wedge\eta_{2}}{|z_{1}-z_{2}|^{2}+\eta_{1}\vee\eta_{2}}.
\end{align*}
If $|z_{2}|>C$, then $u_{2}\simeq(|z_{2}|^{2}+\eta_{2}^{2})^{-1}$ and so $\beta_{z_{1},z_{2}}(i\eta_{1},i\eta_{2})<C$. Therefore we have
\begin{align*}
    \frac{\rho_{1}\rho_{2}[1-u_{1}u_{2}]}{\beta_{z_{1},z_{2}}(\eta_{1},\eta_{2})}&>C\rho_{2}\\
    &>\frac{C\eta_{2}}{|z_{2}|^{2}+\eta_{2}^{2}}\\
    &>\frac{C\eta_{1}\wedge\eta_{2}}{|z_{1}-z_{2}|^{2}+\eta_{1}\vee\eta_{2}}.
\end{align*}

Combining these bounds we deduce that \eqref{eq:C2.1}, \eqref{eq:C2.2}, \eqref{eq:C3.1} and \eqref{eq:C3.2} hold with probability $1-N^{-D}$. {The fact that $\sigma^{X}_{z}=1+O(t)$ for $z\in\mbb{D}_{\omega}$ with probability $1-N^{-D}$ follows by Taylor expansion using the explicit formulae above and is shown in \cite[Proposition 8.1]{maltsev_bulk_2023}}.
\end{proof}

{Theorem \ref{thm2} now follows by the moment matching argument of Tao and Vu \cite{tao_random_2015}.} {Let $A_{1}$ and $A_{2}$ be two real random matrices such that
\begin{align*}
    \mbb{E}a^{p}_{1,jk}&=\mbb{E}a_{2,jk}^{p},\qquad p=0,1,2,3,
\end{align*}
and
\begin{align*}
    \left|\mbb{E}a^{4}_{1,jk}-\mbb{E}a_{2,jk}^{4}\right|&\leq \frac{Ct}{N^{2}}.
\end{align*}
Using \cite[Lemma 3.4]{erdos_universality_2011-2} we can find an $\wt{A}$ such that $A_{1}:=A$ and $A_{2}:=\frac{1}{\sqrt{1+t}}\left(\wt{A}+\sqrt{t}B\right)$ satisfy these conditions. Therefore it is enough to prove that for some $\delta>0$ we have
\begin{align}
    \mbb{E}_{1}\prod_{j=1}^{m_{r}}X_{u_{0}}(f_{j})\prod_{k=1}^{m_{c}}Y_{u_{0}}(g_{k})&=\mbb{E}_{2}\prod_{j=1}^{m_{r}}X_{u_{0}}(f_{j})\prod_{k=1}^{m_{c}}Y_{u_{0}}(g_{k})+O(N^{-\delta}),\label{eq:XY}\\
    \mbb{E}_{1}\prod_{k=1}^{m}Y_{z_{0}}(g_{k})&=\mbb{E}_{2}\prod_{k=1}^{m}Y_{z_{0}}(g_{k})+O(N^{-\delta}),\label{eq:YY}
\end{align}
where $\mbb{E}_{j}$ denotes the expectation with respect to $A_{j}$,
\begin{align*}
    X_{u_{0}}(f)&:=\sum_{n=1}^{N_{\mbb{R}}}f(N^{1/2}(u_{n}-u_{0})),\\
    Y_{z_{0}}(g)&:=\sum_{n=1}^{N_{\mbb{C}}}g(N^{1/2}(z_{n}-z_{0})),
\end{align*}
and $f_{j}\in C^{2}(\mbb{R}),\,g_{k}\in C^{2}(\mbb{C}_{+})$ have compact support.}

{The first observation is that \eqref{eq:XY} and \eqref{eq:YY} can be deduced from
\begin{align}
    \mbb{E}_{1}\prod_{l}Z_{z_{0}}(h_{l})&=\mbb{E}_{2}\prod_{l}Z_{z_{0}}(h_{l})+O(N^{-\delta}),\label{eq:Z}
\end{align}
where
\begin{align*}
    Z_{z_{0}}(h)&:=\sum_{n=1}^{N}h(N^{1/2}(z_{n}-z_{0}))
\end{align*}
for $h\in C^{2}(\mbb{C})$ with compact support. Here the sum is over the whole spectrum, i.e. $z_{n}$ can be a real eigenvalue or an eigenvalue in either the lower or upper half-plane. Indeed, taking $h$ to be supported in \[\{x+iy\in\mbb{C}:|x|<N^{-\tau},\,|y|<N^{-\tau}\}\] for some $\tau>0$ and using \eqref{eq:Z} and Theorem \ref{thm1}, we conclude that for any fixed $u\in(-1,1)$,
\begin{align}
    P_{1}\left(|\{n:|z_{n}-u|<N^{-1/2-\tau}\}|>1\right)&\leq CN^{-\tau},
\end{align}
for sufficiently small $\tau$ (in particular $\tau$ needs to be smaller than the $\delta$ from \eqref{eq:Z} and the $\delta$ from Theorem \ref{thm1}). Since the complex eigenvalues come in conjugate pairs, this implies that the event $\mc{E}_{\tau}$ that there are no strictly complex eigenvalues in the strip \[\{x+iy\in\mbb{C}:|x-u|<N^{-1/2},|y|<N^{-\tau}\}\] has probability $1-CN^{-\tau}$. By the local law in Proposition \ref{prop:single}, we have $|X_{u}(f)|\leq N^{\xi}\|f_{\infty}\|$ with probability $1-N^{-D}$ for any $\xi,D>0$, which allows us to insert the indicator $1_{\mc{E}_{\tau}}$. We have
\begin{align*}
    1_{\mc{E}_{\tau}}X_{u_{0}}(f)&=1_{\mc{E}_{\tau}}Z_{u_{0}}(\wt{f}),
\end{align*}
where $\wt{f}$ is a smooth extension of $f$ to $\mbb{C}$ supported in \[\text{supp}(f)\times[-N^{-\tau},N^{-\tau}].\] Likewise, we have
\begin{align*}
    1_{\mc{E}_{\tau}}Y_{z_{0}}(g)&=1_{\mc{E}_{\tau}}Z_{z_{0}}(\wt{g}),
\end{align*}
where $\wt{g}\in C^{2}(\mbb{C})$ agrees with $g$ in $\{\Im z>N^{-\tau}\}$ and is supported in $\{\Im z>0\}$.}

{To prove \eqref{eq:Z}, we follow the same argument as in \cite[Section 6]{tao_random_2015}, with two differences. Firstly, we replace the local circular law in \cite[Theorem 20]{tao_random_2015} with the one in \cite[Theorem 2.2]{bourgade_local_2014} since the latter does not require the additional assumption of third moment matching with a Gaussian. Secondly, we observe that an error of $t/N^{2}$ in the fourth moment is enough for the argument in the proof of \cite[Theorem 23]{tao_random_2015} (as was the case for Hermitian matrices \cite{erdos_bulk_2010-1}). If $G^{(0)}$, $G^{(1)}$ and $G^{(2)}$ denote the resolvents of Hermitisations whose $(i,N+j)$ entries are 0, $x_{ij}$ and $y_{ij}$ respectively (all other entries being equal), the proof of the four moment theorem amounts to comparing the expectation values of functions of 
\begin{align*}
    G^{(1)}&=\sum_{p=0}^{4}x_{ij}^{p}(G^{(0)}\Delta_{ij})^{p}G^{(0)}+x_{ij}^{5}(G^{(0)}\Delta_{ij})^{5}G^{(1)},\\
    G^{(2)}&=\sum_{p=0}^{4}y_{ij}^{p}(G^{(0)}\Delta_{ij})^{p}G^{(0)}+y_{ij}^{5}(G^{(0)}\Delta_{ij})^{5}G^{(2)},
\end{align*}
where
\begin{align*}
    \Delta_{ij}&=\mbf{e}_{i}\mbf{e}_{N+j}^{*}+\mbf{e}_{N+j}\mbf{e}_{i}^{*}
\end{align*}
(see Eq. (8.3) and the display before Remark 47 in \cite{tao_random_2015}). The moment matching condition ensures that the first three terms in the sum have the same expectation, while the difference in the fourth terms is bounded by
\begin{align}
    \frac{t}{N^{2}}\cdot N\max_{\mu,\nu\in[1,...,2N]}\mbb{E}|G^{(0)}_{\mu,\nu}|^{5}.\label{eq:fourth}
\end{align}
By the local law in Proposition \ref{prop:single} (see e.g. \cite[Lemma 10.2]{benaych-georges_lectures_2016} for the extension to $\eta>0$) we have \[|G^{(l)}_{\mu,\nu}|<N^{\xi}\min\left(1,\frac{1}{N\eta}\right)\] with probability $1-N^{-D}$ for any $\xi,D>0$. The factor $1/N^{2}$ in \eqref{eq:fourth} compensates for the sum over all $N^{2}$ entries.}

\appendix
\section{Additional Proofs for Section \ref{sec3}}
\begin{proof}[Proof of Lemma \ref{lem:detLower}]
Since $X>0$ we can define inverse of the square root $X^{-1/2}$ and obtain
\begin{align*}
    \frac{\left|\det(X+iY)\right|}{\det X}&=\left|\det\left(1+iX^{-1/2}YX^{-1/2}\right)\right|\\
    &=\det^{1/2}\left(1+X^{-1/2}YX^{-1}YX^{-1/2}\right).
\end{align*}
Since $YX^{-1}Y\geq Y^{2}/\|X\|$ (we used the Hermiticity of $Y$) we have
\begin{align*}
    \det^{1/2}\left(1+X^{-1/2}YX^{-1}YX^{-1/2}\right)&\geq\det^{1/2}\left(1+\frac{1}{\|X\|}X^{-1/2}Y^{2}X^{-1/2}\right)\\
    &=\det^{1/2}\left(1+\frac{1}{\|X\|}YX^{-1}Y\right)\\
    &\geq\det^{1/2}\left(1+\frac{Y^{2}}{\|X\|^{2}}\right).
\end{align*}
In the second line we used the identity $\det(1+XY)=\det(1+YX)$.
\end{proof}

\begin{proof}[Proof of Lemma \ref{lem:sphericalint}]
Let
\begin{align}
    I_{\epsilon}&=\frac{1}{2^{k/2}\left(\pi\epsilon\right)^{k(k+1)/4}}\int_{\mbb{R}^{n\times k}}e^{-\frac{1}{2\epsilon}\tr\left(M^{T}M-1_{k}\right)^{2}}f\left(M\right)\diff M.
\end{align}
Any $M\in\mbb{R}^{n\times k}$ has a polar decomposition $M=UP^{\frac{1}{2}}$, where $U\in O(n,k)$ and $P\geq0$. The Jacobian of this transformation is (see e.g. \cite{diaz-garcia_wishart_2011}, Proposition 4)
\begin{align*}
    \diff M&=2^{-k}\det^{(n-k+1)/2-1}\left(P\right)\diff P\,\mathrm{d}_{H}U.
\end{align*}
Thus we have
\begin{align*}
    I_{\epsilon}&=\frac{1}{2^{k/2}\left(\pi\epsilon\right)^{k(k+1)/4}}\int_{P\geq0}e^{-\frac{1}{2\epsilon}\tr\left(P-1_{k}\right)^{2}}g\left(P\right)\diff P,
\end{align*}
where
\begin{align*}
    g\left(P\right)&=\frac{1}{2^{k}}\det^{(n-k+1)/2-1}\left(P\right)\int_{O(n,k)}f\left(UP^{\frac{1}{2}}\right)\,\mathrm{d}_{H}U.
\end{align*}
Since $g\in L^{1}$ and is continuous at $P=1$ we can take the limit $\epsilon\to0$ to obtain
\begin{align}
    \lim_{\epsilon\to0}I_{\epsilon}&=g\left(1_{k}\right)=\frac{1}{2^{k}}\int_{O(n,k)}f\left(U\right)\,\mathrm{d}_{H}U.\label{eq:Ieps1}
\end{align}
On the other hand, by linearising the quadratic exponent through the introduction of an auxilliary variable (i.e. a Hubbard-Stratonovich transformation), we have
\begin{align*}
    I_{\epsilon}&=\frac{1}{2^{k}\pi^{k(k+1)/2}}\int_{\mbb{R}^{n\times k}}\int_{ \mbb{M}^{sym}_{k}}e^{-\frac{\epsilon}{2}\tr P^{2}+i\tr P\left(1_{k}-M^{T}M\right)}f(M)\diff P\diff M\\
    &=\frac{1}{2^{k}}\int_{ \mbf{M}^{sym}_{k}}e^{-\frac{\epsilon}{2}\tr P^{2}+i\tr P}\hat{f}(P)\diff P.
\end{align*}
In the second line we have interchanged the $P$ and $M$ integrals, which is justified by the assumption $f\in L^{1}$ and the factor $e^{-\frac{\epsilon}{2}\tr P^{2}}$. Since moreover $\hat{f}\in L^{1}$, we can take the limit $\epsilon\to0$ inside the integral:
\begin{align}
    \lim_{\epsilon\to0}I_{\epsilon}&=\frac{1}{2^{k}}\int_{ \mbf{M}^{sym}_{k}}e^{i\tr P}\hat{f}(P)\diff P.\label{eq:Ieps2}
\end{align}
Comparing \eqref{eq:Ieps1} and \eqref{eq:Ieps2} we obtain \eqref{eq:sphericalint}.
\end{proof}

\section{Trace Identities}\label{sec:trace}
All resolvents in this section are evaluated at the point $\eta_{z}$ satisfying $t\Tr{H_{z}(\eta_{z})}=1$. We first list some basic identites:
\begin{align*}
    H_{z}X^{*}_{z}&=X^{*}_{z}\wt{H}_{z},\\
    X_{z}\wt{H}_{z}&=H_{z}X_{z},\\
    \eta^{2}_{z}\wt{H}_{z}&=1-X_{z}H_{z}X_{z}^{*},\\
    H_{\bar{z}}&=H_{z}^{T},\\
    X_{z}&=X_{\bar{z}}+\bar{z}-z.
\end{align*}
Combining these identites we find
\begin{align}
    \eta_{z}^{2}\Tr{\wt{H}_{z}\wt{H}_{\bar{z}}}&=\Tr{\wt{H}_{z}(1-X_{\bar{z}}H_{\bar{z}}X_{\bar{z}}^{*})}\nonumber\\
    &=\Tr{\wt{H}_{z}}-\Tr{(X_{z}^{*}+\bar{z}-z)\wt{H}_{z}(X_{z}+z-\bar{z})H_{\bar{z}}}\nonumber\\
    &=\Tr{\wt{H}_{z}}-\Tr{(1-\eta_{z}^{2}H_{z})H_{\bar{z}}}+(z-\bar{z})\Tr{\wt{H}_{z}X_{z}H_{\bar{z}}-H_{\bar{z}}X_{z}^{*}\wt{H}_{z}}-|z-\bar{z}|^{2}\Tr{H_{\bar{z}}\wt{H}_{z}}\nonumber\\
    &=\eta_{z}^{2}\Tr{H_{z}H_{\bar{z}}}+(z-\bar{z})\Tr{H_{\bar{z}}X_{z}H_{z}-H_{z}X_{z}^{*}H_{\bar{z}}}-|z-\bar{z}|^{2}\Tr{H_{\bar{z}}\wt{H}_{z}}.\label{eq:a1}
\end{align}
Recall the definition of $\tau_{z}$:
\begin{align*}
    \tau_{z}&=\left|1-t\Tr{H_{z}X_{z}^{*}X_{\bar{z}}H_{\bar{z}}}\right|^{2}-t^{2}\eta_{z}^{2}\Tr{H_{z}H_{\bar{z}}}{\wt{H}_{z}\wt{H}_{\bar{z}}}.
\end{align*}
Note that
\begin{align*}
    \Tr{H_{z}X_{z}^{*}X_{\bar{z}}H_{\bar{z}}}&=\Tr{H_{z}X^{*}_{z}(X_{z}+z-\bar{z})H_{\bar{z}}}\\
    &=\Tr{(1-\eta^{2}_{z}H_{z})H_{\bar{z}}}+(z-\bar{z})\Tr{H_{z}X^{*}_{z}H_{\bar{z}}},
\end{align*}
and so recalling that $t\Tr{H_{z}}=t\Tr{H_{\bar{z}}}=\eta_{z}$ we find
\begin{align*}
    1-t\Tr{H_{z}X_{z}^{*}X_{\bar{z}}H_{\bar{z}}}&=\eta_{z}^{2}\Tr{H_{z}H_{\bar{z}}}-(z-\bar{z})\Tr{H_{z}X^{*}_{z}H_{\bar{z}}}.
\end{align*}
Inserting this into the expression for $\tau_{z}$ we find
\begin{align*}
    \tau_{z}&=t^{2}\eta^{2}_{z}\Tr{H_{z}H_{\bar{z}}}\left[\eta^{2}_{z}\Tr{H_{z}H_{\bar{z}}-\wt{H}_{z}\wt{H}_{\bar{z}}}+(z-\bar{z})\Tr{H_{\bar{z}}X_{z}H_{z}-H_{z}X_{z}^{*}H_{\bar{z}}}\right]\\
    &+t^{2}|z-\bar{z}|^{2}|\Tr{H_{z}X_{z}^{*}H_{\bar{z}}}|^{2}.
\end{align*}
Using \eqref{eq:a1} in the first line we obtain
\begin{align*}
    \tau_{z}&=t^{2}|z-\bar{z}|^{2}\left[\eta_{z}^{2}\Tr{H_{z}H_{\bar{z}}}\Tr{H_{\bar{z}}\wt{H}_{z}}+|\Tr{H_{z}X_{z}^{*}H_{\bar{z}}}|^{2}\right].
\end{align*}
Recall that $\sigma_{z}$ is defined by
\begin{align*}
    \sigma_{z}&=\eta_{z}^{2}\Tr{H_{z}\wt{H}_{z}}+\frac{|\Tr{H_{z}X_{z}H_{z}}|^{2}}{\Tr{H_{z}^{2}}}.
\end{align*}
Thus we see that
\begin{align}
    \lim_{y\to0}\frac{\tau_{z}}{4y^{2}}&=t^{2}\Tr{H^{2}_{z}}\sigma_{z}.\label{eq:a2}
\end{align}
This confirms that the asymptotics of $L_{n}(\delta,z,X)$ in \eqref{eq:Lnear3} and \eqref{eq:Lfar3} coincide in the limit $y\to0$.

\section{Uniform Asymptotics for the Measure $\nu_{n}$}\label{sec:uniform}
In this appendix we sketch an alternative analysis of the measure $\nu_{n}$ based on an identity proven by Dubova and Yang \cite[Eq. (3.2)]{dubova_bulk_2024}. In our notation this identity reads
\begin{align}
    \det\mc{M}_{z}(\eta)&=\det^{2}\left(\eta^{2}+|X_{z}|^{2}\right)\cdot\det\left[1+\eta^{2}\delta^{2}H_{z}(\eta)\wt{H}_{\bar{z}}(\eta)\right].
\end{align}
To prove this, let $\omega=(c/b)^{1/4}$ and
\begin{align*}
    Y&=\begin{pmatrix}-i\eta&1_{2}\otimes X-Z^{T}\otimes1_{N}\\1_{2}\otimes X^{T}-Z\otimes1_{N}&-i\eta\end{pmatrix},\\
    S&=\frac{1}{\sqrt{2}}\begin{pmatrix}\omega&\omega\\-i/\omega&i/\omega\end{pmatrix}.
\end{align*}
Now note that
\begin{align*}
    Z^{T}&=S\begin{pmatrix}z&0\\0&\bar{z}\end{pmatrix}S^{-1},
\end{align*}
and
\begin{align*}
    Z&=2x-Z^{T}+\delta\begin{pmatrix}0&1\\1&0\end{pmatrix}\\
    &=S\begin{pmatrix}\bar{z}&0\\0&z\end{pmatrix}S^{-1}+\delta\begin{pmatrix}0&1\\1&0\end{pmatrix}\\
    &=S\left[\begin{pmatrix}\bar{z}&0\\0&z\end{pmatrix}+i\delta\begin{pmatrix}-v&u\\-u&v\end{pmatrix}\right]S^{-1},
\end{align*}
where $v=\delta/2y,\,u=\sqrt{\delta^{2}+4y^{2}}/2y$. Therefore we have
\begin{align*}
    Y&=\begin{pmatrix}S&0\\0&S\end{pmatrix}P\left[\begin{pmatrix}W_{z}-i\eta&0\\0&W_{\bar{z}}-i\eta\end{pmatrix}+i\delta\begin{pmatrix}vF^{*}&-uF^{*}\\uF^{*}&-vF^{*}\end{pmatrix}\right]P\begin{pmatrix}S^{-1}&0\\0&S^{-1}\end{pmatrix},\label{eq:Y}
\end{align*}
where
\begin{align*}
    P&=\begin{pmatrix}1_{n}&0&0&0\\0&0&1_{n}&0\\0&1_{n}&0&0\\0&0&0&1_{n}\end{pmatrix}
\end{align*}
is the permutation matrix that exchanges the second and third block rows/columns.

Taking the determinant of $Y$ we find
\begin{align*}
    \frac{\det Y}{\det^{2}\left(\eta^{2}+|X_{z}|^{2}\right)}&=\det\begin{pmatrix}1+i\delta v G_{z}F^{*}&-i\delta uG_{z}F^{*}\\i\delta u G_{\bar{z}}F^{*}&1-i\delta vG_{\bar{z}}F^{*}\end{pmatrix}\\
    &=\det\left[1+i\delta v(G_{z}-G_{\bar{z}})F^{*}+\delta^{2}(v^{2}-u^{2})G_{z}F^{*}G_{\bar{z}}F^{*}\right]\\
    &=\det\left[1-2y\delta v G_{z}FG_{\bar{z}}F^{*}+(2y\delta v+\delta^{2}(v^{2}-u^{2}))G_{z}F^{*}G_{\bar{z}}F^{*}\right]\\
    &=\det\left(1-\delta^{2}G_{z}FG_{\bar{z}}F^{*}\right)\\
    &=\det\left[1+\eta^{2}\delta^{2}\wt{H}_{z}(\eta)H_{\bar{z}}(\eta)\right].
\end{align*}
Using this identity we easily obtain the bound
\begin{align}
    \frac{\det\left(\eta_{z}^{2}+|X_{z}|^{2}\right)}{\det^{1/2}\mc{M}_{z}}&=\det^{-1/2}\left(1+\eta^{2}_{z}\delta^{2}\sqrt{H_{\bar{z}}}\wt{H}_{z}\sqrt{H_{\bar{z}}}\right)\nonumber\\
    &\leq\exp\left\{-\frac{N\eta_{z}^{2}\delta^{2}\Tr{H_{\bar{z}}\wt{H}_{z}}}{1+\delta^{2}/\eta^{2}_{z}}\right\}\nonumber\\
    &\leq \exp\left\{-\frac{CNt\delta^{2}}{(1+\delta^{2}/t^{2})(y^{2}+t)}\right\}.
\end{align}
Thus when $y=O(N^{-1/2})$ we can restrict to $\delta<\frac{\log N}{\sqrt{N}}$ and when $y>C>0$ to $\delta<\frac{\log N}{\sqrt{Nt}}$, as long as $Nt^{3}\gg1$. In the latter case we then proceed as in Lemma \ref{lem:Lfar} to further restrict to $\delta<\frac{\log N}{\sqrt{N}}$. Moreover, combining this with the identity in \eqref{eq:a2}, we can prove that the asymptotics for $L_{n}(\delta,z,X)$ in \eqref{eq:Lfar3} hold uniformly in $\sqrt{N}y>\omega>0$.

Using \eqref{eq:Y} we can also evaluate traces $\Tr{\mc{M}_{z}^{-1}B}$ which appear in Lemma \ref{lem:det2Far}. For example, with $B=\begin{pmatrix}0&M\\0&0\end{pmatrix}$ we have
\begin{align*}
    \Tr{\mc{M}_{z}^{-1}B}&=\frac{1}{2\eta\omega^{2}}\Tr{\begin{pmatrix}1+i\delta vG_{z}F^{*}&-i\delta u G_{z}F^{*}\\i\delta uG_{\bar{z}}F^{*}&1-i\delta vG_{\bar{z}}F^{*}\end{pmatrix}^{-1}\begin{pmatrix}G_{z}&0\\0&G_{z}\end{pmatrix}\begin{pmatrix}-B^{*}&B\\-B^{*}&B\end{pmatrix}}.
\end{align*}
Once we have restricted to $\delta<\frac{\log N}{\sqrt{N}}$ we make a series expansion of the first matrix and estimate higher order terms using Cauchy-Schwarz. The leading term in this case will be
\begin{align*}
    -\frac{i}{2\omega^{2}}\Tr{(H_{z}-H_{\bar{z}})M}.
\end{align*}
The constant $\omega$ has the asymptotics
\begin{align*}
    \omega^{2}&\simeq\begin{cases}y/\delta&\quad\delta\simeq y\\
    1&\quad \delta\ll y
    \end{cases}.
\end{align*}
Thus we recover the results of Lemmas \ref{lem:det2Near} and \ref{lem:det2Far} in the respective limits.

The claim that we can restrict $\delta$ to certain regions requires $t\gg N^{-1/3}$ since we have used the crude bound $\|\sqrt{H_{\bar{z}}}\wt{H}_{z}\sqrt{H_{\bar{z}}}\|\leq\eta^{-4}$, whereas we expect that the optimal bound (in the bulk) is
\begin{align*}
    \|\sqrt{H_{\bar{z}}}\wt{H}_{z}\sqrt{H_{\bar{z}}}\|&\leq\frac{C}{\eta^{2}(y^{2}+\eta)}.
\end{align*}

\section{Jacobian of $S=U\Sigma U^{T}$: proof of \eqref{eq:Jacobian}}\label{sec:jacobian}
Let $\mbb{R}^{m}_{>,+}=\left\{\bs\sigma\in\mbb{R}^{m}:0\leq\sigma_{1}<\cdots<\sigma_{m}\right\}$ and $G=U(2m)/(USp(2))^{m}$. We want to calculate the Jacobian of the map
\begin{align*}
    \mbb{R}^{m}_{>,+}\times G&\to\mbb{M}^{skew}_{2m}(\mbb{C})\\
    (\bs\sigma,U)&\mapsto S=U\Sigma U^{T}
\end{align*}
where
\begin{align*}
    \Sigma=\bigoplus_{j=1}^{m}\begin{pmatrix}0&\sigma_{j}\\-\sigma_{j}&0\end{pmatrix}.
\end{align*}
We use the standard method of determining the metric from the 2-form $\diff s^{2}=\tr \diff S\diff S^{*}$ (see e.g. Section 4.1 in the book by Pastur and Shcherbina \cite{pastur_eigenvalue_2011}). If we denote by $\diff A=U^{*}\diff U$ the invariant 1-form on $U(2m)$, then we have $\diff A^{*}=-\diff A$. Taking the quotient by $(USp(2))^{m}$ imposes the additional constraints $\diff A_{2j-1,2j-1}=\diff A_{2j,2j}$ and $\diff A_{2j-1,2j}=0$ for $j=1,...,m$. Now we calculate
\begin{align*}
    \mathrm{d} s^{2}&=\tr|\mathrm{d} S|^{2}\\
    &={\frac{1}{2}}\tr |U^{*}\mathrm{d} S\bar{U}|^{2}\\
    &={\frac{1}{2}}\tr|\mathrm{d} A\Sigma+\mathrm{d}\Sigma+\Sigma \mathrm{d} A^{T}|^{2}\\
    &=-{\frac{1}{2}}\tr \mathrm{d} \Sigma^{2}+\tr|\mathrm{d} A\Sigma+\Sigma \mathrm{d} A^{T}|^{2}\\
    &=\sum_{j=1}^{m}\mathrm{d} \sigma_{j}^{2}+2\sum_{j<k}^{m}(\sigma_{j}^{2}+\sigma_{k}^{2})(|\mathrm{d} A_{2j-1,2k-1}|^{2}+|\mathrm{d} A_{2j-1,2k}|^{2}+|\mathrm{d} A_{2j,2k-1}|^{2}+|\mathrm{d} A_{2j,2k}|^{2})\\
    &+2\sum_{j<k}^{m}\sigma_{j}\sigma_{k}\left[\Re(\mathrm{d} A^{2}_{2j-1,2k})+\Re(\mathrm{d} A^{2}_{2j,2k-1})-2\Re(\mathrm{d} A_{2j-1,2k-1}\mathrm{d} A_{2j,2k})\right]\\
    &+\sum_{j=1}^{m}\sigma_{j}^{2}\left[|\mathrm{d} A_{2j-1,2j-1}+\mathrm{d} A_{2j,2j}|^{2}+4\mathrm{d}(\Re(A_{2j-1,2j}))^{2}\right].
\end{align*}
Reading off the elements of the metric $g$ in the basis $\{\diff\sigma_{j},\diff A_{2j,2k}\}$ we find
\begin{align*}
    \sqrt{|g|}&={2^{m}}\prod_{j=1}^{m}\sigma_{j}\prod_{j<k}^{m}(\sigma_{j}^{2}-\sigma_{k}^{2})^{4}.
\end{align*}

\paragraph{Acknowledgements}
This work was supported by the Royal Society, grant number \\RF/ERE210051.


\begin{thebibliography}{10}

\bibitem{aggarwal_bulk_2019}
Amol Aggarwal.
\newblock Bulk universality for generalized {{Wigner}} matrices with few
  moments.
\newblock {\em Probability Theory and Related Fields}, 173(1):375--432,
  February 2019.

\bibitem{alt_local_2018}
Johannes Alt, L{\'a}szl{\'o} Erd{\H o}s, and Torben Kr{\"u}ger.
\newblock Local inhomogeneous circular law.
\newblock {\em The Annals of Applied Probability}, 28(1), February 2018.


\bibitem{bai_spectral_2010}
Zhidong Bai and Jack~W. Silverstein.
\newblock {\em Spectral {{Analysis}} of {{Large Dimensional Random Matrices}}}.
\newblock Springer {{Series}} in {{Statistics}}. {Springer New York}, {New
  York, NY}, 2010.

\bibitem{benaych-georges_lectures_2016}
Florent Benaych-Georges and Antti Knowles.
\newblock Lectures on the local semicircle law for Wigner matrices.
\newblock arXiv:1601.04055

\bibitem{bercovici_brown_2022}
Hari Bercovici and Ping Zhong.
\newblock The Brown Measure of a Sum of Two Free Random Variables, One of Which is R-Diagonal.
\newblock arXiv:2209.12379

\bibitem{borodin_ginibre_2009}
A.~Borodin and C.~D. Sinclair.
\newblock The {{Ginibre Ensemble}} of {{Real Random Matrices}} and its
  {{Scaling Limits}}.
\newblock {\em Communications in Mathematical Physics}, 291(1):177--224,
  October 2009.

\bibitem{bourgade_fixed_2016}
Paul Bourgade, Laszlo Erd{\H o}s, Horng-Tzer Yau, and Jun Yin.
\newblock Fixed {{Energy Universality}} for {{Generalized Wigner Matrices}}.
\newblock {\em Communications on Pure and Applied Mathematics},
  69(10):1815--1881, 2016.

\bibitem{bourgade_local_2014}
Paul Bourgade, Horng-Tzer Yau, and Jun Yin.
\newblock Local circular law for random matrices.
\newblock {\em Probability Theory and Related Fields}, 159(3):545--595, August
  2014.

\bibitem{cipolloni_edge_2021}
Giorgio Cipolloni, L{\'a}szl{\'o} Erd{\H o}s, and Dominik Schr{\"o}der.
\newblock Edge universality for non-{{Hermitian}} random matrices.
\newblock {\em Probability Theory and Related Fields}, 179(1-2):1--28, February
  2021.

\bibitem{cipolloni_central_2023}
Giorgio Cipolloni, L{\'a}szl{\'o} Erd{\H o}s, and Dominik Schr{\"o}der.
\newblock Central {{Limit Theorem}} for {{Linear Eigenvalue Statistics}} of
  {{Non-Hermitian Random Matrices}}.
\newblock {\em Communications on Pure and Applied Mathematics},
  76(5):946--1034, 2023.

\bibitem{cipolloni_mesoscopic_2023}
Giorgio Cipolloni, L{\'a}szl{\'o} Erd{\H{o}}s, and Dominik Schr{\"o}der.
\newblock Mesoscopic central limit theorem for non-hermitian random matrices.
\newblock {\em Probability Theory and Related Fields}, 2023.

\bibitem{cipolloni_maximum_2024}
Giorgio Cipolloni and Benjamin Landon.
\newblock Maximum of the Characteristic Polynomial of I.I.D. Matrices.
\newblock arXiv:2405.05045

\bibitem{diaz-garcia_wishart_2011}
Jos{\'e}~A. {D{\'i}az-Garc{\'i}a} and Ram{\'o}n {Guti{\'e}rrez-J{\'a}imez}.
\newblock On {{Wishart}} distribution: {{Some}} extensions.
\newblock {\em Linear Algebra and its Applications}, 435(6):1296--1310,
  September 2011.

\bibitem{dubova_bulk_2024}
Sofiia Dubova and Kevin Yang.
\newblock Bulk universality for complex eigenvalues of real non-symmetric random matrices with i.i.d. entries.
\newblock arXiv:2402.10197

\bibitem{edelman_probability_1997}
Alan Edelman.
\newblock The {{Probability}} that a {{Random Real Gaussian Matrix haskReal
  Eigenvalues}}, {{Related Distributions}}, and the {{Circular Law}}.
\newblock {\em Journal of Multivariate Analysis}, 60(2):203--232, February
  1997.

\bibitem{edelman_how_1994}
Alan Edelman, Eric Kostlan, and Michael Shub.
\newblock How {{Many Eigenvalues}} of a {{Random Matrix}} are {{Real}}?
\newblock {\em Journal of the American Mathematical Society}, 7(1):247--267,
  1994.

\bibitem{erdos_dynamical_2017}
L.~Erd{\H o}s and H.T. Yau.
\newblock {\em A Dynamical Approach to Random Matrix Theory}.
\newblock Courant Lecture Notes. {Courant Institute of Mathematical Sciences,
  New York University}, 2017.

\bibitem{erdos_bulk_2010}
L{\'a}szl{\'o} Erd{\H o}s, Sandrine P{\'e}ch{\'e}, Jos{\'e}~A. Ram{\'i}rez,
  Benjamin Schlein, and Horng-Tzer Yau.
\newblock Bulk universality for {{Wigner}} matrices.
\newblock {\em Communications on Pure and Applied Mathematics}, 63(7):895--925,
  2010.

\bibitem{erdos_bulk_2010-1}
L{\'a}szl{\'o} Erd{\"o}s, Jos{\'e} Ram{\'i}rez, Benjamin Schlein, Terence Tao, van Vu, and Horng-Tzer Yau.
\newblock Bulk universality for {{Wigner}} hermitian matrices with subexponential decay.
\newblock {\em Mathematical Research Letters}, 17(4):667--674, 2010.

\bibitem{erdos_universality_2011}
L{\'a}szl{\'o} Erd{\H o}s, Benjamin Schlein, and Horng-Tzer Yau.
\newblock Universality of random matrices and local relaxation flow.
\newblock {\em Inventiones mathematicae}, 185(1):75--119, July 2011.

\bibitem{erdos_universality_2011-2}
L{\'a}szl{\'o} Erd{\"o}s, Horng-Tzer Yau, and Jun Yin.
\newblock Universality for generalized {{Wigner}} matrices with {{Bernoulli}}
  distribution.
\newblock {\em Journal of Combinatorics}, 2(1):15--81, 2011.

\bibitem{erdos_rigidity_2012}
L{\'a}szl{\'o} Erd{\H o}s, Horng-Tzer Yau, and Jun Yin.
\newblock Rigidity of eigenvalues of generalized {{Wigner}} matrices.
\newblock {\em Advances in Mathematics}, 229(3):1435--1515, February 2012.

\bibitem{fischer_uber_1908}
Ernst Fischer.
\newblock \"uber den {{Hadamardschen}} determinantensatz.
\newblock {\em Arch. Math. (Basel)}, 13:32--40, 1908.

\bibitem{forrester_eigenvalue_2007}
Peter~J. Forrester and Taro Nagao.
\newblock Eigenvalue {{Statistics}} of the {{Real Ginibre Ensemble}}.
\newblock {\em Physical Review Letters}, 99(5):050603, August 2007.

\bibitem{grela_diffusion_2016}
Jacek Grela.
\newblock Diffusion method in {{Random Matrix Theory}}.
\newblock {\em Journal of Physics A: Mathematical and Theoretical},
  49(1):015201, January 2016.

\bibitem{horn_matrix_2012}
Roger~A. Horn and Charles~R. Johnson.
\newblock {\em Matrix Analysis}.
\newblock {Cambridge University Press}, {Cambridge ; New York}, 2nd ed edition,
  2012.

\bibitem{johansson_universality_2001}
Kurt Johansson.
\newblock Universality of the {{Local Spacing Distribution}} in {{Certain
  Ensembles}} of {{Hermitian Wigner Matrices}}.
\newblock {\em Communications in Mathematical Physics}, 215(3):683--705,
  January 2001.

\bibitem{knowles_anisotropic_2017}
Antti Knowles and Jun Yin.
\newblock Anisotropic local laws for random matrices.
\newblock {\em Probability Theory and Related Fields}, 169(1):257--352, October
  2017.

\bibitem{landon_fixed_2019}
Benjamin Landon, Philippe Sosoe, and Horng-Tzer Yau.
\newblock Fixed energy universality of {{Dyson Brownian}} motion.
\newblock {\em Advances in Mathematics}, 346:1137--1332, April 2019.

\bibitem{liu_phase_2022}
Dang-Zheng Liu and Lu~Zhang.
\newblock Phase transition of eigenvalues in deformed {{Ginibre}} ensembles.
\newblock arXiv:2204.13171

\bibitem{maltsev_bulk_2023}
Anna Maltsev and Mohammed Osman.
\newblock Bulk {{Universality}} for {{Complex}} non-{{Hermitian Matrices}} with
  {{Independent}} and {{Identically Distributed Entries}}, 
\newblock {\em Probability Theory and Related Fields}, 2024

\bibitem{mehta_random_1990}
Madan~Lal Mehta.
\newblock {\em Random {{Matrices}}}.
\newblock {Academic Press}, 1990.

\bibitem{nock_characteristic_2016}
Andre Nock.
\newblock {\em Characteristic {{Polynomials}} of {{Random Matrices}} and
  {{Quantum Chaotic Scattering}}}.
\newblock PhD thesis, Queen Mary, University of London, 2016.

\bibitem{pastur_eigenvalue_2011}
Leonid Pastur and Mariya Shcherbina.
\newblock {\em Eigenvalue {{Distribution}} of {{Large Random Matrices}}},
  volume 171 of {\em Mathematical {{Surveys}} and {{Monographs}}}.
\newblock {American Mathematical Society}, {Providence, Rhode Island}, July
  2011.

\bibitem{shcherbina_transfer_2016}
Mariya Shcherbina and Tatyana Shcherbina.
\newblock Transfer {{Matrix Approach}} to 1d {{Random Band Matrices}}:
  {{Density}} of {{States}}.
\newblock {\em Journal of Statistical Physics}, 164(6):1233--1260, September
  2016.

\bibitem{shcherbina_universality_2018}
Mariya Shcherbina and Tatyana Shcherbina.
\newblock Universality for 1d {{Random Band Matrices}}: {{Sigma-Model
  Approximation}}.
\newblock {\em Journal of Statistical Physics}, 172(2):627--664, July 2018.

\bibitem{sommers_general_2008}
Hans-J{\"u}rgen Sommers and Waldemar Wieczorek.
\newblock General eigenvalue correlations for the real {{Ginibre}} ensemble.
\newblock {\em Journal of Physics A: Mathematical and Theoretical},
  41(40):405003, September 2008.

\bibitem{soshnikov_universality_1999}
Alexander Soshnikov.
\newblock Universality at the {{Edge}} of the {{Spectrum}}{\textparagraph}in
  {{Wigner Random Matrices}}.
\newblock {\em Communications in Mathematical Physics}, 207(3):697--733,
  November 1999.

\bibitem{tao_random_2011}
Terence Tao and Van Vu.
\newblock Random matrices: {{Universality}} of local eigenvalue statistics.
\newblock {\em Acta Mathematica}, 206(1):127--204, January 2011.

\bibitem{tao_random_2015}
Terence Tao and Van Vu.
\newblock Random matrices: {{Universality}} of local spectral statistics of
  non-{{Hermitian}} matrices.
\newblock {\em The Annals of Probability}, 43(2):782--874, March 2015.

\bibitem{tribe_ginibre_2014-1}
Roger Tribe and Oleg Zaboronski.
\newblock The {{Ginibre}} evolution in the large-{{N}} limit.
\newblock {\em Journal of Mathematical Physics}, 55(6):063304, June 2014.

\bibitem{tribe_averages_2023}
Roger Tribe and Oleg Zaboronski.
\newblock Averages of products of characteristic polynomials and the law of
  real eigenvalues for the real {{Ginibre}} ensemble.
\newblock arXiv:2308.06841.

\bibitem{zhong_brown_2021}
Ping Zhong.
\newblock Brown Measure of the Sum of an Elliptic Operator and a Free Random Variable in a Finite Von Neumann Algebra.
\newblock arXiv:2108.09844

\end{thebibliography}
\end{document}